\documentclass[12pt, a4paper]{amsart}
\usepackage{amscd,amssymb,mathtools}

\usepackage{enumitem}
\usepackage{tensor}
\usepackage{wrapfig}
\usepackage{tikz}

\allowdisplaybreaks

\usepackage{color, url, hyperref}

\hypersetup{colorlinks=true, linkcolor=blue, linkbordercolor=blue, citecolor=blue, citebordercolor=blue, linktocpage=true}

\def\Aut{\operatorname{Aut}} 

\def\D{\operatorname{dom}}
\def\c{\operatorname{cod}}

\def\clsp{\overline{\operatorname{span}}}

\def\id{\operatorname{id}}

\def\max{\operatorname{max}}
\def\min{\operatorname{min}}

\def\PI{\operatorname{PIso}}
\def\Per{\operatorname{Per}}

\def\C{\mathbb{C}}
\def\F{\mathbb{F}}
\def\R{\mathbb{R}}
\def\N{\mathbb{N}}
\def\Z{\mathbb{Z}}
\def\T{\mathbb{T}}

\def\CC{\mathcal{C}}

\def\KK{\mathcal{K}}
\def\LL{\mathcal{L}}

\def\OO{\mathcal{O}}

\def\TT{\mathcal{T}}

\def\NT{\mathcal{NT}}

\newcommand{\inv}{^{-1}}

\newtheorem{thm}{Theorem}[section]
\newtheorem{cor}[thm]{Corollary}

\newtheorem{lemma}[thm]{Lemma}
\newtheorem{prop}[thm]{Proposition}

\theoremstyle{definition}
\newtheorem{definition}[thm]{Definition}

\theoremstyle{remark}
\newtheorem{remark}[thm]{Remark}

\newtheorem{example}[thm]{Example}

\numberwithin{equation}{section}

\tikzstyle{vertex}=[circle]
\tikzstyle{goto}=[->,shorten >=1pt,>=stealth,semithick]

\begin{document}

\date{\today}
\title[Self-similar actions of groupoids on $k$-graphs]{$C^*$-algebras of self-similar actions of groupoids on higher-rank graphs and their equilibrium states}
\author{Zahra Afsar}
\author{Nathan Brownlowe}
\address{Zahra Afsar and Nathan Brownlowe, School of Mathematics and Statistics, University of Sydney,
NSW 2006, Australia.}
\email{za.afsar@gmail.com,Nathan.Brownlowe@sydney.edu.au}
\author[Jacqui Ramagge]{Jacqui Ramagge}
\address{Jacqui Ramagge, University of South Australia, Adelaide, SA 5001, Australia.}
\email{Jacqui.Ramagge@unisa.edu.au}
\author[Mike Whittaker]{Michael F. Whittaker}
\address{Michael F. Whittaker, School of Mathematics and Statistics, University of Glasgow, University Place, Glasgow Q12 8QQ, United Kingdom}
\email{Mike.Whittaker@glasgow.ac.uk}

 \thanks{This research was partially supported by ARC Discovery Project grant DP170101821, and EPSRC grant EP/R013691/1.}

\thanks{}

\begin{abstract}
We introduce the notion of a self-similar action of a groupoid $G$ on a finite higher-rank graph. To these actions we associate a compactly aligned product system of Hilbert bimodules, and thereby obtain corresponding universal Nica--Toeplitz and Cuntz--Pimsner algebras. We consider natural actions of the real numbers on both algebras and study the KMS states of the associated dynamics. For large inverse temperatures, we describe the simplex of KMS states on the Nica--Toeplitz algebra in terms of traces on the full $C^*$-algebra of $G$. 
We prove that if the graph is $G$-aperiodic and the action satisfies a finite-state condition, then there is a unique KMS state on the Cuntz--Pimsner algebra. 
\end{abstract}

\maketitle

\setcounter{tocdepth}{1}


\section{Introduction}\label{sec:intro}
The introduction of self-similar groups in the past 40 years have been used to solve open problems and provide new and interesting examples of groups. Perhaps the most famous among them is the Grigorchuk group \cite{Gr} providing the first example of a group with intermediate growth, solving a problem of Milnor from 1968. Since the appearance of the Grigorchuk group, the theory of self-similar groups, or self-similarity more generally, has grown rapidly through the contributions of many authors, and it is now a major area of research within the broader field of geometric group theory. Recent applications of self-similarity include Nekrashevych's example of a finitely generated \textit{simple} group with intermediate growth \cite{Nek-simp} and the work of Belk, Bleak, and Matucci \cite{BBM} that uses graph automata to prove that all Gromov hyperbolic groups embed into the asynchronous rational group of \cite{GrigNS}. 

There is now a rich interplay between self-similar groups and operator algebras, pioneered by Nekrashevych's construction in \cite{nekra0,nekra} of a Cuntz--Pimsner algebra from a self-similar action. The work in \cite{nekra} was revisited by the last two named authors, along with Laca and Raeburn, in \cite{lrrw}, where they also studied the Toeplitz extension. Motivated by operator algebraic considerations, Exel and Pardo introduced the notion of a self-similar action of a group on the path space of a directed graph \cite{EP2}, and used this new notion of self-similarity to produce a class of $C^*$-algebras which unified Nekrashevych's Cuntz--Pimsner algebras from \cite{nekra0} and the Katsura algebras from \cite{Kat}. Working in parallel to \cite{EP2}, the authors of \cite{lrrw} produced a follow-up \cite{LRRW} in which they defined a self-similar \textit{groupoid} action on the path space of a directed graph, and associated Toeplitz and Cuntz--Pimsner algebras to these new self-similar actions. 

The successful jump to actions on graphs in \cite{EP2} and \cite{LRRW} begs the question of whether the theory can be extended to actions on \textit{higher-rank} graphs. This has been successfully achieved on the group-action side through the work of Li and Yang in \cite{LY0,LY,LY1}, thus giving a higher-rank analogue of \cite{EP2}. The first main goal of this paper is to build a higher-rank analogue of \cite{LRRW} by constructing a theory of self-similar groupoid actions on higher-rank graphs. For this we consider finite higher-rank graphs (or $k$-graphs) $\Lambda$. We view $\Lambda$ as a forest of rooted trees whose set of roots is $\Lambda^0$. We define a partial isomorphism $g$ to be a degree-preserving isomorphism between two trees with the extra property that $g(\lambda e) \in g(\lambda)\Lambda$ for all $\lambda \in v\Lambda$ and edges $e \in s(\lambda)\Lambda$. The set of partial isomorphisms of $\Lambda$ is a groupoid under the natural structure maps. We say a groupoid $G$ with unit space $\Lambda^0$ acts self-similarly on $\Lambda$ if there is an embedding of $G$ into the groupoid of partial isomorphisms of $\Lambda$ with the property that for every $g\in G$ and edge $e\in \D(g)\Lambda$, there exists $h\in G$ satisfying $g\cdot (e\lambda)=(g\cdot e)(h\cdot \lambda)$ for all $\lambda\in s(e)\Lambda$. (Here, we are viewing $g$ as a partial isomorphism, and writing $g\cdot \mu$ for the image of $\mu\in \Lambda$ under $g$.) 

To each self-similar action of a groupoid $G$ on a higher-rank graph $\Lambda$ we associate both a Nica--Toeplitz algebra $\TT(G,\Lambda)$ and a Cuntz--Pimsner algebra $\OO(G,\Lambda)$, which we do by first constructing a product system of Hilbert bimodules. This product system generalises the construction in \cite{LRRW} in which the Toeplitz and Cuntz--Pimsner algebras of a single Hilbert bimodule are considered. We describe our $C^*$-algebras $\TT(G,\Lambda)$ and $\OO(G,\Lambda)$ succinctly through generators and relations that encode the self-similarity of the action. 

Following the work in \cite{lrrw,LRRW} (see also \cite{CS}), as well as \cite{LY} for group actions on $k$-graphs, we devote the majority of this paper to calculating the equilibrium states of our $C^*$-algebras $\TT(G,\Lambda)$ and $\OO(G,\Lambda)$ under a natural time evolution lifted from a gauge action. Our work generalises the results in \cite{lrrw,LRRW}, and it has a number of advantages over the work in \cite{LY}. Firstly, we give a description of the equilibrium states of the Nica--Toeplitz algebra, which has a richer KMS structure than its Cuntz--Pimsner quotient. More importantly, we do not assume that our actions are pseudo free, in the sense of \cite[Definition~2.3(i)]{LY}. The presence of psuedo freeness in \cite{LY} means that the results do not apply to many classical self-similar group actions like the basilica group. To see this, consider the sets $F_g^k$ from \cite[Equation (6.5)]{lrrw}. The set $F_b^1$ from the basilica action is nonempty, and hence the action is not pseudo free. However, the action is \emph{regular}, see \cite[Definition 6.1]{nekra}, and hence the groupoid of germs is Hausdorff by \cite[Lemma 5.4]{nekra}.

Pseudo freeness is a hypothesis in \cite{LY} (see also \cite{EP2,LY}) because it ensures that the groupoid model for the Cuntz--Pimsner algebra is Hausdorff. This allows them, for example, to use groupoid results to prove the surjectivity of the isomorphism in \cite[Theorem~6.12]{LY}, whereas the surjectivity in our analogous Theorem~\ref{thm3} is proved by constructing KMS states directly using 
Perron--Frobenius theory.  Moreover,  this property makes characterising and computation of KMS states significantly easier. In particular, imposing  pseudo freeness means that \cite[Theorem~6.3]{LY0} gives only those KMS$_1$-states coming from the underlying higher-rank graph algebra, and the group action plays no part in the KMS structure.

For the Toeplitz algebra of a Hilbert bimodule there is general machinery in \cite{ln} which characterises the KMS states in terms of the traces of the coefficient algebra satisfying a subinvariance relation. The first-named author, along with Larsen and Neshveyev, recently generalised this result to the Nica--Toeplitz algebra of a product system of Hilbert bimodules \cite{ALN} (see also \cite{Kak}). In the last decade, operator algebraists have shown interest  in providing a more direct approach to the results of \cite{ln}. This direct approach has a number of advantages, including that it provides concrete solutions to the subinvariant relations of \cite{ln, ALN}; for each trace on the coefficient algebra, it gives an explicit formula for the corresponding KMS state on the Toeplitz algebras, leading to KMS states that are computable; and it reveals many interesting properties of both the underling objects and the corresponding $C^*$-algebras. We follow this trend by taking a direct approach to studying the KMS states of $\TT(G,\Lambda)$. For the dynamics, we fix a vector $ r\in [0, \infty)^k$ and compose the embedding $t\mapsto e^{irt}$ of $\R$ into $\T^k$ with the gauge action. Then for $\beta$ greater than a standard critical inverse temperature, our first main theorem shows that the simplex of the KMS$_\beta$-states for $\TT(G,\Lambda)$ is isomorphic to the simplex of tracial states of the full groupoid $C^*$-algebra $C^*(G)$. Our Theorem~\ref{Thm:KMS_beta_tau} indicates that all tracial states on $C^*(G)$ satisfy the subinvariance relation of \cite[Theorem ~2.1 and  Theorem ~3.1]{ALN}.  

In general, Cuntz--Pimsner algebras admit KMS states at very few inverse temperatures with respect to the gauge dynamics and their values are difficult to compute. Our Cuntz--Pimsner algebra $\OO(G,\Lambda)$ is no exception due to fact that each KMS state of $\OO(G,\Lambda)$ restricts to a KMS state of the higher-rank graph algebra $C^*(\Lambda)$ via the natural embedding $C^*(\Lambda)\hookrightarrow \OO(G,\Lambda)$. Thus, the simplex of KMS states of $\OO(G,\Lambda)$ is a subset of the simplex of KMS states on $C^*(\Lambda)$. In the case of $C^*(\Lambda)$ it was shown in \cite{aHLRS} that the KMS states are non-trivial only at the inverse temperature $\beta=1$ in the direction of the preferred dynamics given by taking $r:=(\ln(\rho(B_1),\dots, \ln(\rho(B_k)))$, where each $\rho(B_i)$ is the spectral radius of the adjacency matrix $B_i$ of $\Lambda$. We note that the KMS states of $\TT(G,\Lambda)$ may not factor through $\OO(G,\Lambda)$ for general strongly connected $\Lambda$; this happens only for a small family of graphs that are coordinatewise strongly connected and that have rationally independent $\ln(\rho(B_i))$s.  This means we need to study KMS states of $\OO(G,\Lambda)$ using a different approach.  

Motivated by \cite{aHLRS,LY}, we associate to $(G,\Lambda)$ the $G$-periodicity group $\Per(G,\Lambda)$, which is a subgroup of $\Z^k$ that measures the periodicity of $\Lambda$ in the presence of the self-similar action of $G$. Given $p-q\in \Per(G,\Lambda)$, we introduce a bijection $\theta_{p,q}$ that maps $\Lambda^p$ onto $\Lambda^q$; and a function $h_{p,q}$ from $\Lambda^p$  into $G$. The $\theta$ bijections reduce to those of \cite{aHLRS} when there is no groupoid action. The $h$ functions reflect the effects of the groupoid action on the $G$-periodicity of the graph. These functions play crucial roles in our analysis of the KMS states of $\OO(G,\Lambda)$, so we take our time to study their proprieties in Lemmas~\ref{lemma:htehta} and \ref{lemma:mor-prop-htheta}. We then use them to represent the full group $C^*$-algebra $C^*(\Per(G,\Lambda))$ in $\OO(G,\Lambda)$ via a central representation $V$ in Proposition~\ref{prop:rep-per}. This gives a homomorphism $\pi_V$ of the space of tracial states of  $C^*(\Per(G,\Lambda))$ into the simplex of KMS$_1$-states of $\OO(G,\Lambda)$. We then focus on proving that $\pi_V$ is an isomorphism. We first observe that if the unimodular Perron--Frobenius eigenvector $x_\Lambda$ of the $k$-graph is not invariant under the action of $G$, then there is no KMS$_1$-state for $\OO(G,\Lambda)$. So to prove the surjectivity of $\pi_V$, we need to assume that $x_\Lambda$ is invariant. Then given a tracial state $\tau$ of $C^*(\Per(G,\Lambda))$, we use the Perron--Frobenius measure $M$ of \cite{aHLRS} to construct a KMS$_1$-state directly. The process is very involved and requires a careful study of the effects of the groupoid action on $M$. In particular, we need to compute the measure of the cylinder sets $Z(\lambda,g,\mu)$ of \eqref{Z-set}. Theorem ~\ref{thm2} shows that this effect will be captured by  numbers $c_{d(\lambda),d(\mu),\lambda}(g)$. These numbers have appeared before in a much simpler form  in the  self-similar actions on $1$-graphs \cite{lrrw,LRRW} when the underling graph automatically has trivial $G$-periodicity group. In our setting, having a nontrivial $\Per(G,\Lambda)$ adds extra levels of difficulty to computing these numbers. Interestingly, as Theorem~\ref{thm:formula for kms1} shows, these numbers are crucial in characterising of the KMS$_1$-states of $\OO(G,\Lambda)$ and consequently the injectivity of the homomorphism $\pi_V$ in Theorem~\ref{thm3}. In Theorem~\ref{thm2} we characterise the uniqueness of the KMS$_1$-state for $\OO(G,\Lambda)$. We show that a self-similar action with $G$-invariant unimodular eigenvector and with a finite-state condition has a unique KMS$_1$-state if and only if the $k$-graph is $G$-aperiodic.

All of the theory described above is supported by examples. In particular, we introduce the notion of a $k$-coloured graph automaton, and we use them to build examples of self-similar groupoid actions on $k$-graphs. In Section~\ref{sec:automaton}, we focus on some particular examples coming from automata. The first example is an automaton on a $2$-coloured graph with a single vertex, which gives rise to a self-similar group action on a $2$-graph that is not pseudo free. We also construct a version of the basilica group acting on
a $2$-coloured graph, and obtain a self-similar groupoid action on a $2$-graph. In both examples, the unimodular eigenvectors are invariant for the groupoid actions, and the $G$-periodicity groups are trivial, meaning that there is a unique KMS$_1$-state. In Section~\ref{examples} we then compute their unique KMS$_1$-states.

\section{Preliminaries} 
\subsection{Groupoids}
A \textit{groupoid} $G$ is a small category with a partially-defined product in which every morphism is invertible. We write $G^0$ for the set of objects (which we usually call the \textit{unit space} of $G$), and $G^2:=\{(g,h):\D(g)=c(h)\}$ for the set of \textit{composable pairs}. We write $\c,\D:G\to G^{(0)}$ for the range and source maps of $G$.
Following \cite{LRRW}, a \textit{unitary representation} of $G$ in a $C^*$-algebra $B$ is a set of partial isometries $\{U_g: g\in G\}\subseteq B$ such that 
\begin{enumerate}[label=(U\arabic*),ref=U\arabic*]
\item \label{U1}$\{U_v: v\in G^0\}$ is a set of mutually orthogonal projections;
\item \label{U2}$U_g^*=U_{g^{-1}}$ for all $g\in G$; and
\item \label{U3} $U_g U_h=\delta_{\D(g),\c(h)}U_{gh}$ for all $g,h\in G$.
\end{enumerate}
The set $\{i_g:g\in G\}$ of point masses is a unitary representation of $G$ in the $C^*$-algebra $C^*(G)$ described in \cite{e3}. By \cite[Proposition~4.1]{LRRW}, the pair $(C^*(G),i)$ is universal for the unitary representation of $G$. 

\subsection{$k$-graphs}
Let $ 1\leq k\in \N$, and $\Lambda$ be a $k$-graph as in \cite{KP}. We write $\Lambda^0$ for the vertex set, $r$ and $s$ for the range and source maps, and $d$ for the degree map. We say $\Lambda$ is \textit{finite} if $\Lambda^p:=d^{-1}(p)$ is finite for all $p\in\N^k$. We denote the standard generators of $\N^k$ by $e_1,\dots,e_k$, and we say $\lambda\in \Lambda$ is an \textit{edge} if $d(\lambda)=e_i$ for some $1\leq i\leq k$. Given $v,w\in \Lambda^0$, $v\Lambda^p w$ denotes the set $\{\lambda\in \Lambda^p: r(\lambda)=v \text { and } s(\lambda)=w\}$. 
We say $\Lambda$ has \textit{no sources} if $v\Lambda^p\neq\emptyset$ for every $v\in \Lambda^0$ and $p\in \N^k$. We denote by $B_1,\dots ,B_k \in M_{\Lambda^0}[0,\infty)$ the coordinate adjacency matrices of $\Lambda$, in the sense that
$B_i (v,w)=:|v\Lambda^{e_i}w|$ for all $1\leq i\leq k$ and $v,w\in \Lambda^0$. We write $\rho(B_i)$ for the spectral radius of $B_i$.

A\textit{Toeplitz--Cuntz--Krieger $\Lambda$-family} in a $C^*$-algebra $B$ is a set of partial isometries $\{T_\lambda: \lambda\in \Lambda\}\subseteq B$ such that
	\begin{enumerate}[label=(TCK\arabic*),ref=TCK\arabic*]
		\item \label{TCK1} $\{T_v: v\in \Lambda^0\}$ is a set of mutually orthogonal projections;
		\item \label{TCK2} $T_\lambda T_\mu=T_{\lambda \mu}$ whenever $s(\lambda)=r(\mu)$; and
		\item \label{TCK3} $T_\lambda^*T_\mu=\sum_{(\eta,\zeta)\in\Lambda^{\min}(\lambda,\mu)} T_\eta T_\zeta ^*$ for all $\lambda,\mu\in \Lambda$. 
	\end{enumerate}
We interpret empty sums as $0$. A Toeplitz--Cuntz--Krieger $\Lambda$-family $\{T_\lambda: \lambda\in \Lambda\}$ is a \textit{Cuntz--Krieger $\Lambda$-family} if we also have 
	\begin{enumerate}[label=(CK),ref=CK]
	 \item\label{CK}	$\quad T_{v}= \sum_{\lambda\in v\Lambda^p}T_\lambda T_\lambda^*$ for all $v\in \Lambda^0$ and $p\in \N^k$.
	\end{enumerate}
	
Recall that $\Omega_k:=\{(p,q):p,q\in \N^k, p\leq q\}$ is the $k$-graph with structure maps $(p,q)(q,n)=(p,n)$, $r(p,q)=(p,p)$, $s(p,q)=(q,q)$, and $d(p,q)=q-p$. For a $k$-graph $\Lambda$, we let $\Lambda^\infty$ denote the set of all degree-preserving functors from $\Omega_k$ to $\Lambda$. We call $\Lambda^\infty$ the \textit{infinite path space of} $\Lambda$. We write $\sigma:\Lambda^\infty\to \Lambda^\infty$ for the \textit{shift map} defined by 
$\sigma^n(x)(p,q)=x(p-n,q-n)$. 
	
\subsection{Product systems of Hilbert bimodules}
Let $A$ be a $C^*$-algebra. A right Hilbert $A$-bimodule is a right Hilbert $A$-module $M$ with a left action of $A$ defined by a homomorphism $\varphi:A\to \LL(M)$ from $A$ into the $C^*$-algebra $ \LL(M)$ of adjointable operators on $M$. We write $_A{A}_A$ for the standard bimodule with the inner product defined by 
$\langle a,b\rangle=ab^*$, and the actions given by multiplication of $A$. 
For $x,y\in M$, let $\Theta_{x,y}$ be the rank-one opeartor on $M$ defined by $\Theta_{x,y}(z)=x\cdot\langle y,z\rangle$. Then $\KK(M):=\clsp\{\Theta_{x,y}:x,y\in M\}$ is the algebra of \emph{compact operators on $M$}. 
Provided that $A$ is a unital $C^*$-algebra, $M$ is called \textit{essential} if $\varphi(1_A)x=x$ for all $x\in M$. 

Given two right Hilbert $A$-bimodules $M$ and $N$, let $M\odot N$ be the algebraic tensor product of $M$ and $N$, and let 
 $M\odot_A N$ be the quotient of $M\odot N$ by the subspace
\begin{align}\label{balanced}
\{(x\cdot a)\odot y-x\odot(a\cdot y): x\in M, y\in N, a\in A\}.
\end{align}
There is a well-defined right action of $A$ on $M\odot_A N$ such that 
$(x\odot_A y)\cdot a=x\odot_A y\cdot a$ for $x\odot_A y\in X\odot_A N$ and $a\in A$.
We can equip $M\odot_A N$ with the right $A$-valued inner product 
	\begin{align}\label{bipro}
	\big\langle x\odot_A y, x'\odot_A y' \big\rangle =\big\langle y, \varphi_N(\langle x,x'\rangle) y'\big\rangle\quad \text{ for }x\odot_A y, x'\odot_A y' \in M\odot_A N.
	\end{align}
Writing $M\otimes_A N$ for the completion of $M\odot_A N$ in the inner product \eqref{bipro}, Lemma~2.16 in  \cite{tfb} implies that \eqref{bipro} extends to a right $A$-valued inner product on $M\otimes_A N$. Since for each $S \in \LL(M)$ there is an adjointable operator $S\otimes 1_N$ such that $S\otimes 1_N(x\otimes y)=S(x)\otimes y$ for $x\otimes y\in M\otimes_A N$, the map $a\mapsto \varphi_M(a)\otimes 1_N$ defines a left action of $A$ by adjointable operators on $M\otimes_A N$. Thus $M\otimes_A N$ becomes a right Hilbert $A$--$A$ bimodule called the \textit{balanced tensor product} of $M$ and $N$. We denote elements of $M\odot N$ by $x\odot y$, and write $x\otimes y$ for elements of both $M\odot_A N$ and $M\otimes_A N$.

Let $P$ be a semigroup with identity $e$, and let $A$ be a $C^*$-algebra. For each $p\in P$ let $M_p$ be a right Hilbert $A$-bimodule and let $\varphi_p:A\rightarrow \LL(M_p)$ be the homomorphism that determines the left action of $A$ on $M_p$. A \textit{product system of right Hilbert $A$-bimodules over $P$} is the disjoint union $M:=\bigsqcup_{p\in P}M_p$
	such that
	\begin{enumerate}[label=(P\arabic*),ref=P\arabic*]
		\item \label{P1} $M_e$ is equal to the standard bimodule $_A{A}_A$;
		\item\label{P2} $M$ is a semigroup, and for each $p,q \in P\setminus\{e\}$ the map $(x,y)\mapsto xy:M_p\times M_q \rightarrow M_{pq}$ descends to an isomorphism $\pi_{p,q}: M_p\otimes_A M_q\rightarrow M_{pq}$; and
		\item\label{P3} the multiplications $M_e\times X_p\rightarrow M_p$ and $M_p\times M_e\rightarrow M_p$ satisfy
		\[ax=\varphi_p(a)z,\quad xa=x\cdot a\, \text{ for }\, a\in M_e \text { and } x\in M_p.\]
	\end{enumerate}
Each bimodule is called a \textit{fiber}. If each fiber is essential, then (\ref{P2}) also holds for $p,q=e$. The product systems of interest to us have essential fibers, and so (\ref{P2}) holds for all $p\in P$.
	
If $P$ is a subsemigroup of a group $G$ with $P\cap P^{-1}=\{e\}$, then a partial order on $G$ is defined by $p\leq q$ if $p^{-1}q\in P$. The pair $(G,P)$ is called a \textit{quasi-lattice ordered group} if for each pair $p,q\in G$ with a common upper bound in $P$ there is a least upper bound $p\vee q$ in $P$. If $p,q\in G$ have no common upper bound, we write $p\vee q=\infty$. 

Given a quasi-lattice ordered group $(G,P)$ and elements $p,q\in P$, the isomorphism $\pi_{p,q}$ from \eqref{P2} gives rise to an isomorphism from $\LL(M_p)$ to $\LL(M_{p\vee q})$ given by $S\mapsto S^{p\vee q}_p:=\pi_{p,q}\circ (S\otimes 1_{M_p})\circ\pi_{p,q}^{-1}$; so we have
\begin{align}\label{for-compacts}
S^{p\vee q}_p(x\otimes y)&=S(x) y \text{ for all }x\otimes y \in M_p\otimes M_{(p\vee q)p^{-1}}.
\end{align}
A product system of right Hilbert $A$-bimodules over $P$ is called \textit{compactly aligned}, if for all $p,q\in P$ with $p\vee q<\infty$, $S\in \KK(X_p)$ and $T\in\KK(X_q)$, we have $S^{p\vee q}_pT^{p\vee q}_q\in \KK(X_{p\vee q})$.
	
\subsection{$C^*$-algebras of product systems of Hilbert bimodules}
Let $P$ be a semigroup with identity $e$, and $M$ be a product system of right Hilbert $A$-bimodules over $P$. Suppose that $\psi$ is a function from $M$ to a $C^*$-algebra $B$, and write $\psi_p$ for the restriction of $\psi$ to $X_p$. Then we say $\psi$ is a \textit{Toeplitz representation} of $M$ if
	\begin{enumerate}[label=(T\arabic*),ref=T\arabic*]
		\item\label{T1} $\psi_e:A\rightarrow B$ is a homomorphism, and $\psi_p:X_p\rightarrow B$ is linear for each $p\in P\setminus\{e\}$;
		\item\label{T2}$\psi_p(x)^*\psi_p(y)=\psi_e(\langle x,y\rangle)$ for $p\in P$, and $x,y\in M_p$;
		\item \label{T3} $\psi_{pq}(xy)=\psi_p(x)\psi_q(y)$ for $p,q\in P$, $x\in X_p$, and $y\in M_q$.
	\end{enumerate}
	Conditions (\ref{T1}) and (\ref{T2}) imply that $(\psi_p,\psi_e)$ is a representation of the Hilbert $A$-bimodule $M_p$.  We know from \cite{P} that there is a homomorphism $\psi^{(p)}:\KK(M_p)\rightarrow B$ such that $\psi^{(p)}(\Theta_{x,y})=\psi_p(x)\psi_p(y)^*$.	
	
Recall from \cite{F} that a Toeplitz representation $\psi$ of a compactly aligned product system $M$ in a $C^*$-algebra $B$ is \textit{Nica covariant} if and only if for every $p,q\in P$, $S\in \KK(M_p)$, and $T\in \KK(M_q)$, we have
	\begin{align*}
	\psi^{(p)}(S)\psi^{(q)}(T)=\begin{cases}
	\psi^{(p\vee q)}\big(S^{p\vee q}_pT^{p\vee q}_q\big)&\text{if $p\vee q<\infty$}\\
	0&\text{otherwise.}\\
	\end{cases}
	\end{align*}
The \textit{Nica-Toeplitz algebra } $\NT(M)$ 
is the $C^*$-algebra generated by a universal Nica covariant representation of $M$. We write $\psi$ for the universal Nica covariant representation.

Recall from \cite{F} that a Toeplitz representation $T$ of $M$ is \textit{Cuntz--Pimsner covariant} if
	\begin{align}\label{cunzcon}	
	T_e(a)=T^{(p)}(\varphi_p(a))\quad\text{for all } p\in P, a\in \varphi_p^{-1}(\KK(M_p)).
	\end{align}
The \textit{Cuntz--Pimsner algebra} $\OO(X)$ is generated by a universal Cuntz--Pimsner covariant representation of $M$. 
If the Cuntz--Pimsner covariance implies Nica covariance, then \cite[Lemma~2.2]{AaHR} implies that $\OO(M)$ is the quotient of $\NT(M)$ by the ideal
	\begin{align}\label{C-P-Q}
\big\langle\psi(a)-\psi^{(p)}(\varphi_p(a)):p\in P, a\in \varphi_p^{-1}(\KK(M_p))\big\rangle.
\end{align}

\section{Self-similar groupoid actions on $k$-graphs}

In this section we introduce the notion of a self-similar action of a groupoid on a $k$-graph. We start with the definition of a partial isomorphism of a $k$-graph.

\begin{definition}\label{def: partial iso}
Let $\Lambda$ be a $k$-graph. A \emph{partial isomorphism} of $\Lambda$ consists of vertices $v,w \in \Lambda^0$ and a bijection $g:v\Lambda \to w\Lambda$ satisfying
\begin{enumerate}[label=(\arabic*),ref=\arabic*]
\item\label{prope1} for all $p\in \N^k$, the restriction $g|_{v\Lambda^p}$ is a bijection of $v\Lambda^p$ onto $w\Lambda^{p}$; and
\item \label{prope2} $g(\lambda e) \in g(\lambda)\Lambda$ for all $\lambda \in v\Lambda$ and edges $e \in s(\lambda)\Lambda$.
\end{enumerate}
We write $\PI(\Lambda)$ for the set of all partial isomorphisms of $\Lambda$.
\end{definition}

For each $v\in \Lambda^0$, the identity map $\id_{v}:v\Lambda \to v\Lambda$ is a partial isomorphism. We define domain and codomain maps $\D,\c\colon \PI(\Lambda) \to \Lambda^0$ by 
\[\D(g):=v\text{ and } \c(g):=w, \text{ for }g:v\Lambda \to w\Lambda.\]

\begin{lemma}\label{lemma:groupoid}
Let $\Lambda$ be a $k$-graph. Then $\PI(\Lambda)$ is a groupoid, with unit space $\Lambda^0$; range map $\c$ and source map $\D$; and composition and inverse given by composition and inverse of functions, respectively.
\end{lemma}

\begin{proof}
A similar argument to the one found in the proof of \cite[Proposition~3.2]{LRRW} shows that $\PI(\Lambda)$ is a category and we have $\id_{\c(g)}g=g=g\id_{\D(g)}$ for each $g$. 
Since $g:\D(g)\Lambda\to \c(g)\Lambda$ is a bijection satisfying Definition~\ref{def: partial iso}\eqref{prope1}, the inverse function $g^{-1}:\c(g)\Lambda\to \D(g)\Lambda$ also satisfies Definition~\ref{def: partial iso}\eqref{prope1}. To check Definition~\ref{def: partial iso}\eqref{prope2} for $g^{-1}$, let $\mu\in \c(g) \Lambda$ and $ e\in s(\mu)\Lambda$. Let $\lambda=g^{-1}(\mu e)$. Since $d(g^{-1}(\mu e))=d(\mu e)$, we can factor $\lambda=\lambda'\lambda''$ such that $d(\lambda'')=d(e)$. Now since $g$ satisfies Definition~\ref{def: partial iso}\eqref{prope2}, we have 
\[g^{-1}(\mu e)=\lambda'\lambda''\Longrightarrow \mu e=g(\lambda'\lambda'')\in g(\lambda')\lambda^{d(\lambda'')}.\]
 By the factorisation property, $g(\lambda')=\mu$ and hence $\lambda' =g^{-1}(\mu)$. Thus $g^{-1}(\mu e)\in g^{-1}(\mu)\Lambda$ giving Definition~\ref{def: partial iso}\eqref{prope2}. Hence
$\PI(\Lambda)$ is a category with inverses; i.e. a groupoid. 
\end{proof}

Let $\Lambda$ be a $k$-graph and let $G$ be a groupoid with unit space $G^{(0)}=\Lambda^0$. 
An \emph{action} of $G$ on $\Lambda$ is 
a groupoid homomorphism $\varphi:G\to \PI(\Lambda)$. Identifying $\id_v$ with $v$ for $v\in \Lambda^0$, we see that $\varphi$ is unit preserving. We say $\varphi$ is \emph{faithful} if it is injective. When it is not ambiguous to do so, we write $g\cdot \mu$ for $\varphi_g(\mu)$. 

\begin{definition}
\label{faithful self-similar groupoid action}
A \emph{self-similar groupoid action} $(G,\Lambda)$ consists of a $k$-graph $\Lambda$, a groupoid $G$ with unit space $\Lambda^0$, and a faithful action of $G$ on $\Lambda$ such that
for every $g\in G$ and edge $e \in \D(g)\Lambda$, there exists $h\in G$ satisfying
\begin{equation}
\label{selfsimilar groupoid defn}
g\cdot(e\lambda)=(g\cdot e)(h\cdot \lambda) \text{ for all } \lambda\in s(e)\Lambda.
\end{equation}
Since the action is faithful, there is a unique $h$ satisfying \eqref{selfsimilar groupoid defn}, and we write $g|_e:=h$.
\end{definition}

Let $p\in \N^k$, $\lambda\in \D(g)\Lambda^p$, and let $\lambda_1\lambda_2\dots\lambda_l$ and $\lambda'_1\lambda'_2\dots\lambda'_l$ be different factorisations of $\lambda$ into edges in $\Lambda$. Then for every $\mu\in s(\lambda)\Lambda$, we have 
\[
(g\cdot \lambda)((\cdots (g|_{\lambda_1})|_{\lambda_2})\cdots )|_{\lambda_{l}})\cdot \mu=g\cdot(\lambda\mu)=(g\cdot \lambda)((\cdots (g|_{\lambda'_1})|_{\lambda'_2})\cdots )|_{\lambda'_{l}})\cdot \mu
\]
The factorisation property now implies that 
\[
((\cdots (g|_{\lambda_1})|_{\lambda_2})\cdots )|_{\lambda_{l}})\cdot \mu=((\cdots (g|_{\lambda'_1})|_{\lambda'_2})\cdots )|_{\lambda'_{l}})\cdot \mu,
\]
and since the action is faithful, we have 
\[
((\cdots (g|_{\lambda_1})|_{\lambda_2})\cdots )|_{\lambda_{l}})=((\cdots (g|_{\lambda'_1})|_{\lambda'_2})\cdots )|_{\lambda'_{l}}).
\]
We then define 
\begin{equation}\label{rest-formula}
g|_\lambda:=(\cdots (g|_{\lambda_1})|_{\lambda_2})\cdots )|_{\lambda_{l}}.
\end{equation}

The following lemma shows some properties of a self-similar groupoid action. 
\begin{lemma}\label{lem:properties}
Let $(G,\Lambda)$ be a self-similar groupoid action. Then for $g, h \in G$ with $\D(h) =\c(g)$, $\lambda\in \D(g)\Lambda$, and $\mu\in s(\lambda)\Lambda$ we have

\begin{enumerate}[label=(S\arabic*),ref=S\arabic*]
 \item\label{S1}$g\cdot(\lambda \mu)=(g\cdot \lambda)(g|_\lambda \cdot \mu)$;
 \item\label{S2} $g|_{\lambda \mu}=(g|_\lambda)|_\mu$;
 \item\label{S3}$\D(g|_\lambda)=s(\lambda)$ and $\c(g|_\lambda)=s(g\cdot \lambda)$;
 \item\label{S4} $r(g\cdot \lambda) =g\cdot r(\lambda)$ and $s(g\cdot \lambda) =g|_\lambda \cdot s(\lambda)$;
\item \label{S5} $\id_{r(\lambda)}|_{\lambda}=\id_{s(\lambda)}$;
\item\label{S6} $(gh)|_\lambda=g|_{h\cdot \lambda} h|_\lambda$;
\item\label{S7} $g^{-1}|_{\lambda}=\big(g|_{g^{-1}\cdot \lambda}\big)^{-1}$; and 
\item\label{S8} $g|_{\D(g)}=g$.
\end{enumerate}
\end{lemma}
\begin{proof}
Parts (\ref{S1}) and (\ref{S2}) follow from iterated applications of the Definition~\ref{faithful self-similar groupoid action} and \eqref{rest-formula}. With (\ref{S1}) in mind, a similar argument to \cite[Lemma~ 3.4]{LRRW} gives (\ref{S3})--(\ref{S6}). Also using (\ref{S5})--(\ref{S6}) in the proof of \cite[Proposition~3.5~(4)]{LRRW} we can prove (\ref{S7}).
For (\ref{S8}), take $\lambda\in \D(g)\Lambda$. Applying (\ref{S3}) in the second, we have
\[r(g|_{\D(g)}\cdot \lambda)=\c(g|_{\D(g)})=s(g\cdot \D(g))=s(\c(g))=\c(g).\]
 Therefore 
\[g\cdot \lambda=g\cdot (\D(g)\lambda)=g\cdot \D(g)(g|_{\D(g)}\cdot \lambda)=\c(g)(g|_{\D(g)}\cdot \lambda)=g|_{\D(g)}\cdot \lambda\]
for all $\lambda\in \D(g)\Lambda$. Since the action is faithful, we have $g|_{\D(g)}=g$.
\end{proof}

\begin{example}\label{eg:LYexample}
In \cite{LY0}, the authors introduce the notion of a self-similar action of a \textit{group} $F$ on a $k$-graph $\Lambda$. For such an action, let $F\times \Lambda^0$ be the transformation groupoid with structure maps $\D(f,v)=v$, $\c(f,v)=f\cdot v$, $(f,f'\cdot v)(f',v)=(ff',v)$ and $(f,v)^{-1}=(f^{-1},k\cdot v)$. For each $(f,v)\in F\times \Lambda^0$, define 
\begin{equation}\label{action-apx}
\varphi_{(f,v)}(\lambda):=f\cdot \lambda, \quad \text{for } \lambda\in v\Lambda.
\end{equation}
Then by \cite[Definition~3.1]{LY0}, $\varphi_{(f,v)}:\D(f,v)\Lambda\to \c (f, v)\Lambda$ is a bijection. Also 
for $\lambda \in v\Lambda$ and an edge $e \in s(\lambda)\Lambda$, we have 
\begin{align*}
\varphi_{(f,v)}(\lambda e)&=f\cdot (\lambda e)=(f\cdot \lambda)(f|_\lambda \cdot e)=
\varphi_{(f,v)}(\lambda)(f|_\lambda)\cdot e \in \varphi_{(f,v)}(\lambda)\Lambda.
\end{align*}
Hence $F\times \Lambda^0$ acts on $\Lambda$ by partial isomorphisms. If the action \eqref{action-apx} is faithful, then it is self-similar with restriction map $(f,v)|_{\lambda}:=(f|_ \lambda, s(\lambda))$, for $\lambda\in v \Lambda$.
\end{example}

\section{Coloured-graph automata}\label{sec:automaton}

In this section we introduce the notion of a coloured-graph automaton, and we show how a coloured-graph automaton gives rise to a self-similar groupoid action on a $k$-graph. We begin with a subsection devoted to a recap on the basics of coloured graphs from \cite{HRSW}.

\subsection{$k$-coloured graphs and their $k$-graphs}
Fix $k\in \N\setminus\{0\}$, and let $\F_k$ be the free semigroup generated by $\{c_1,\dots,c_k\}$.
A \textit{$k$-coloured graph} (or just \textit{coloured graph}) is a directed graph $E$ together with a function $c : E^1 \to \{c_1,\dots,c_k\}$. We call $c$ a \textit{colour map}. We can extend $c$ to a functor $c : E^* \to \F_k$.
We define $\omega:\F_k\to\N^k$ by $\omega(c_i)=e_i$. Given two $k$-coloured graphs $E,F$ with colour maps $c_E,c_F$, a \emph{coloured-graph morphism} $\theta$ consists of functions $\theta_0:E^0\to F^0$ and $\theta_1:E^1\to F^1$ such that 
$ r_F\circ\theta_1= \theta_0\circ r_E$, $s_F\circ\theta_1= \theta_0\circ s_E$, and $c_F\circ\theta_1=c_E$.

Let $p \in (\N \cup \{\infty\})^k$, and let $p+v_1,\ldots,p+v_k$ be formal symbols. Recall from \cite[Example~3.1]{HRSW} the $k$-coloured graph $E_{k,p}$, where 
\[
E^0_{k,p} := \{q \in \N^k : 0 \leq q \leq p\}\quad\text{and}\quad E^1_{k,p} := \{q + v_i : q, q + e_i \in E^0_{k,p}\};
\]
$r(q+ v_i) := q; s(q + v_i): = q + e_i$; and $c(q + v_i): = c_i.$

Given a $k$-coloured graph $E$ and distinct $i,j \in \{1,\dots, k\}$, a \emph{square} in $E$ is a coloured-graph morphism $\theta: E_{k,e_i+e_j} \to E$. Given a coloured-graph morphism $\lambda: E_{k,p} \to E$ and a square $\theta$ in $E$, we say $\theta$ \emph{occurs in $\lambda$} if there exists
$q\in \N^k$ such that $\theta(x) = \lambda(x + q)$ for all $x\in E_{k,e_i+e_j }$.
A \emph{complete collection of squares} is a collection $\CC$ of
squares in $E$ such that, for each $x \in E^*$ with $c(x) = c_ic_j$ and $i \neq j$, there is a unique
$\theta\in \CC$ such that $x = \theta(v_i)\theta(e_i + v_j).$ For such a $\theta,$ we will write $\theta(v_i)\theta(e_i + v_j ) \sim \theta(v_j )\theta(e_j + v_i)$. In other words,
for each $c_ic_j$-coloured path $x \in E^*$, there is a unique $c_j c_i$-coloured path $y$ such that
$x\sim y$. A coloured-graph morphism
$\lambda: E_{k,p} \to E$ is \emph{$\CC$-compatible} if each square occurring in $\lambda$ belongs to $\CC$.
A complete collection of squares is $\CC$-\emph{associative} if for
every path $xyz$ in $E$ such that $x,y,z$ are edges of distinct colour, the edges $x_1, x_2, y_1, y_2,z_1, z_2 $ and $x^1, x^2, y^1, y^2,
z^1, z^2 $ determined by
\[xy\sim x^1y^1,x^1z\sim z^1x^2 \text{ and } y^1z^1\sim z^2y^2\]
\[yz\sim z_1y_1,x z_1\sim z_2 x_1\text{ and } x_1y_1\sim x_2 y_2\]
satisfy $x_2=x^2$, $y_2=y^2$ and $z_2=z^2$. 

Let $E$ be a $k$-coloured graph, and $\CC$ be an associative complete collection of squares in $E$. We now recall from \cite[Theorems~4.4~and~4.5]{HRSW} that this data gives rise to a unique $k$-graph $\Lambda_{E,\CC}$. For each $p\in \N^k$, let $\Lambda_{E,\CC}^p$ be the set of all $\CC$-compatible coloured-graph morphisms $\lambda:E_{k,p}\to E$. Define 
\[d(\lambda):=p,\quad r(\lambda):=\lambda(0),\quad\text{and}\quad s(\lambda):=\lambda(p).\] In \cite[Theorem~4.4]{HRSW} it is proved that if $\lambda:E_{k,p}\to E$ and $\mu:E_{k,q}\to E$ are $\CC$-compatible coloured-graph morphisms such that $s(\lambda)=r(\mu)$, then there exists a unique $\CC$-compatible coloured-graph morphism $\mu\nu:E_{k,p+q}\to E$ that restricts to be $\lambda$ and $\mu$. Under this composition, and the structure maps above, 
\begin{align}\label{k-graph from couloured}
\Lambda_{E,\CC}:=\bigcup_{p\in \N^k}\Lambda_{E,\CC}^p
\end{align}
is a $k$-graph. 

Given $E,\CC$ and a $\CC$-compatible coloured-graph morphism $\lambda : E_{k,p} \to E$, we say $x\in E^*$ \emph{traverses} $\lambda$ if $\omega (c(x)) = p $ and $\lambda(\omega(c(x_1 \dots x_{l-1})) + v_{c(x_l)}) = x_l$ for all $0<l\leq |p|$. When $p = 0$, then $x\in E^0$ and we say $x$ traverses $\lambda$ if $x = \lambda(0)$. In particular, for each coloured-graph
morphism $\lambda\in \Lambda_{E,\CC}^p$, and every decomposition $d(\lambda) = e_{j_1} + e_{j_2} +\cdots + e_{j_l}$ the path
$x:=\lambda(0+v_{j_1})\lambda(e_{j_1}+v_{j_2})\lambda ((d(\lambda)-e_{j_1})+v_{j_1})$ in $E^*$
 traverses $\lambda$. Conversely,
for every $x \in E^*$ there is a unique coloured-graph morphism $\lambda_x :E_{k,\omega(c(x))} \to E$ such that $ x$ traverses $\lambda_x$
see \cite[Proposition~4.7]{HRSW}. Moreover by \cite[Remark~4.12]{HRSW} we have 
\begin{align}\label{lambda-homo}
\lambda_{x}\lambda_{y}=\lambda_{x,y} \quad \text{ whenever } x,y\in E^*\text{ and } r(y)=s(x).
\end{align}

\subsection{Automata associated to coloured graphs}\label{subsec: automatons}

In this subsection we introduce the notion of a coloured-graph automaton, and we construct a self-similar groupoid action on a $k$-graph from a coloured-graph automaton.

\begin{definition}\label{defn: (E,C)-automaton}
Let $E$ be a $k$-coloured graph with colour map $c$, and let $\CC$ be an associative complete collection of squares in $E$. An \emph{$(E,\CC)$-automaton} is a finite set $A$ containing $E^0$,
together with functions $r_A,s_A\colon A\to E^0$ such that $r_A(v)=v=s_A(v)$ if $v\in E^0\subseteq A$, 
and a function 
\begin{equation}\label{autmap}
A \tensor[_{s_A}]{\times}{_{r_E}} E^1 \ni (a,e) \mapsto (a\cdot e, a|_e)\in E^1 \tensor[_{s_E}]{\times}{_{r_A}} A
\end{equation} 
such that for all $a\in A$ we have
\begin{enumerate}[label=(A\arabic*),ref=A\arabic*]
\item \label{item: bijection} $e\mapsto a\cdot e$ is a bijection of $s_A(a)E^1$ onto $r_A(a)E^1$;
\item \label{item: colour_preserve} $c(a\cdot e)=c(e)$ for all $e\in s_A(a)E^1$;
\item \label{second property of automaton} $s_A(a|_e)=s_E(e)$ for all $e\in s_A(a)E^1$;
\item \label{item: range_and_source} $r_E(e)\cdot e= e$ and $r_E(e)|_e=s_E(e)$ for all $e\in E^1$;
\item \label{item: square_preserving_1} for each $\lambda \in \CC$ with $\lambda:= ef\sim f'e'$ we have $(a \cdot e)(a|_e \cdot f) \sim (a \cdot f')(a|_{f'} \cdot e')$; and
\item \label{item: square_preserving_2} for each $\lambda \in \CC$ with $\lambda:= ef\sim f'e'$ we have $(a|_e)|_f = (a|_{f'})|_{e'}$.
\end{enumerate}
\end{definition}

 We find the following diagram a helpful illustration of properties \eqref{item: square_preserving_1} and \eqref{item: square_preserving_2}, where $a \in A$ maps the top commuting square to the bottom via the vertical arrows.
\[
\begin{tikzpicture}[yscale=0.15,xscale=0.2]
 \node[inner sep=0pt] (m) at (0,0){} ;
 \node[inner sep=0pt] (n) at (0,20) {};
 \node[inner sep=0pt] (e) at (18,0) {};
 \node[inner sep=0pt] (ne) at (18,20) {};
 \node[inner sep=0pt] (m') at (8,8) {};
 \node[inner sep=0pt] (n') at (8,28) {};
 \node[inner sep=0pt] (e') at (26,8) {};
 \node[inner sep=0pt] (ne') at (26,28) {};
 \draw[-latex, blue] (ne.west)--(n.east) node[pos=0.3,above] 
{\color{blue}$e$};
 \draw[-latex, blue] (e.west)--(m.east) node[pos=0.5,below] 
{\color{blue}$a.e$};
 \draw[-latex, black,dotted] (n.south)--(m.north) node[pos=0.4,left] 
{\color{black}$a$};
 \draw[-latex, black,dotted] (ne.south)--(e.north) 
node[pos=0.3,right] {\color{black}$a|_e$};
 \draw[-latex, red,dashed] (n'.south)--(n.north) node[pos=0.5,left] 
{\color{red}$f'$};
\draw[-latex, red,dashed] (ne'.south)--(ne.north) node[pos=0.5,left] 
{\color{red}$f$};
 \draw[-latex, black,dotted] (ne'.south)--(e'.north) node[pos=0.5,right] 
{\color{black}$(a|_{e})|_{f}=(a|_{f'})|_{e'}$};
 \draw[-latex, black,dotted] (n'.south)--(m'.north) node[pos=0.6,left] 
{\color{black}$a|_{f'}$};
\draw[-latex, blue] (ne'.west)--(n'.east) node[pos=0.4,above] 
{\color{blue}$e'$};
\draw[-latex, blue] (e'.west)--(m'.east) node[pos=0.7,below] 
{\color{blue}$a|_{f'}.e'$};
\draw[-latex, red,dashed] (m'.south)--(m.north) node[pos=0.4,left] 
{\color{red}$a. f'$};
\draw[-latex, red,dashed] (e'.south)--(e.north) node[pos=0.7,right] 
{\color{red}$a|_e. f$};
\end{tikzpicture}
\]

\begin{definition}
For $a\in A$, and $x_1\cdots x_l\in s_A(a)E^*$ we define 
\begin{equation}\label{aut1-2}
a|_{ x_1\cdots x_l}:=(\cdots ((a|_{x_1})|_{x_2})\cdots )|_{x_{l}}
\end{equation}
and
\begin{equation}\label{aut2-2}
a\cdot (x_1\cdots x_l):= (a\cdot x_1)(a|_{x_1}\cdot x_2)\cdots (a|_{x_1\cdots x_{l-1}}\cdot x_{l}).
\end{equation}
\end{definition}

For simplicity, we write $r$, $s$ for the range and source maps of $k$-coloured graphs and their corresponding $k$-graphs as in \eqref{k-graph from couloured}.

\begin{lemma}\label{autpropert}
Let $A$ be an $(E,\CC)$-automaton.
\begin{enumerate}
\item \label{autpropert-1}For all $x\in s_A(a) E^*$ and $y\in s(x) E^*$ we have $a\cdot(xy)=(a\cdot x)(a|_x\cdot y)$.
\item \label{autpropert-2}For all $a\in A$ and $l\in \N$, the map $x\mapsto a\cdot x$ is a bijection of $s_A(a)E^l$ onto $r_A(a)E^l$.
\end{enumerate}
\end{lemma}
\begin{proof}
Part~\eqref{autpropert-1} follows from \eqref{aut1-2} and \eqref{aut2-2}. We prove \eqref{autpropert-2} by induction on $l$. The base case $l=1$ is precisely \eqref{item: bijection}. Now suppose that the claim is true for $l$. To prove the case $l+1$, 
note that $x\mapsto a\cdot x$ is a colour preserving map from $s_A(a)E^{l+1}$ to $r_A(a)E^{l+1}$. To see that this map is injective, let $x,x'\in s_A(a)E^{l+1}$ such that $a\cdot x=a\cdot x'$. Write $x=ey$ and $x'=e'y'$, where $e,e'\in E^1$ and $y,y'\in E^l$. Then by part \eqref{autpropert-1}, we have 
\[(a\cdot e)(a|_e\cdot y)=a\cdot x=a\cdot x'=(a\cdot e')(a|_{e'}\cdot y').\]
 It follows that $a\cdot e=a\cdot e'$ and $a|_e\cdot y= a|_{e'}\cdot y'$. The relation \eqref{item: bijection} implies that $e=e'$ and hence $a|_{e}=a|_{e'}$. Now applying the inductive hypothesis with $a|_e\cdot y= a|_{e}\cdot y'$ shows that $y=y'$ and hence $x=x'$.
 
For surjectivity, let $x\in r_A(a)E^{l+1}$. Write $x=ey$ where $e\in r_A(a)E^1$ and $y\in s(e)E^l$. By \eqref{item: bijection}, there is an $e'\in s_A(a)E^1$ such that $e=a\cdot e'$ and $r_A(a|_{e'})=s(a\cdot e')=s(e)=r(y)$. Therefore $y\in r_A(a|_{e'})E^l$ and applying the inductive hypothesis gives $y'\in s_A(a|_{e'})E^l$ such that $y=a\cdot y'$. Applying \eqref{second property of automaton} gives $r(y')=s_A(a|_{e'})=s(e')$. Part~\eqref{autpropert-1} now implies that $a\cdot(e'y')=(a\cdot e')(a|_{e'}\cdot y')=ey=x$, giving surjectivity.
\end{proof}

\begin{definition}\label{def: differing_square}
We say two paths $x=x_1\cdots x_{l}$ and $ y=y_1\cdots y_{l}$ in $E^*$ \emph{differ by a square} if there is an $1\leq i\leq l-1$ such that $x_ix_{i+1}\sim y_iy_{i+1}$ and $x_j= y_j$ for all $j\neq i,i+1.$
\end{definition}

\begin{lemma}\label{lemma:k-aut}
Let $A$ be an $(E,\CC)$-automaton. If $x_1\cdots x_{l}, y_1\cdots y_{l}\in E^*$ differ by a square, then
\begin{enumerate}
\item \label{lemma:k-aut-1}$a|_{x_1\cdots x_{l}}=a|_{y_1\cdots y_{l}}$, and 
\item \label{lemma:k-aut-2}$a\cdot(x_1\cdots x_{l})$ and $a\cdot(y_1\cdots y_{l})$ differ by a square. 
\end{enumerate}
\end{lemma}
\begin{proof}

Let $\alpha,\beta\in E^*$ such that $x_1\cdots x_l=\alpha x_ix_{i+1}\beta$ and $y_1\cdots y_{l}=\alpha y_i y_{i+1}\beta$. For \eqref{lemma:k-aut-1}, let $b:= a|_{\alpha}$. Then \eqref{item: square_preserving_2} gives
\[
a|_{x_1\dots x_l} = \left(b|_{x_ix_{i+1}}\right)|_\beta = \left(b|_{y_iy_{i+1}}\right)|_\beta = a|_{y_1\dots y_l}.
\]
For \eqref{lemma:k-aut-2}, first note that $\alpha x_ix_{i+1}$ and $\alpha y_i y_{i+1}$ differ by a square, and hence by \eqref{lemma:k-aut-1} we have $a|_{\alpha x_ix_{i+1}}\cdot \beta = a|_{\alpha y_iy_{i+1}}\cdot \beta$. Using this identity gives
\begin{align*}
a\cdot (x_1\cdots x_l)&= (a\cdot \alpha)(a|_{\alpha}\cdot x_i)
 (a|_{\alpha x_{i}}\cdot x_{i+1}) (a|_{\alpha x_{i}x_{i+1}}\cdot \beta)\\
&=(a\cdot \alpha)(a|_{\alpha}\cdot x_i) (a|_{\alpha x_{i}}\cdot x_{i+1}) (a|_{\alpha y_{i}y_{i+1}}\cdot \beta).
\end{align*}
Since we know from \eqref{item: square_preserving_1} that $(a|_{\alpha}\cdot x_i) (a|_{\alpha x_{i}}\cdot x_{i+1})\sim(a|_{\alpha}\cdot y_i) (a|_{\alpha y_{i}}\cdot y_{i+1})$, we see that $a\cdot (x_1\cdots x_l)$ differs from
\[
a\cdot (y_1\cdots y_l)=(a\cdot \alpha)(a|_{\alpha}\cdot y_i) (a|_{\alpha y_{i}}\cdot y_{i+1}) (a|_{\alpha y_{i}y_{i+1}}\cdot \beta)
\]
by a square. 
\end{proof}

\begin{prop}\label{prop:aut}
Let $A$ be an $(E,\CC)$-automaton, and let $\Lambda:=\Lambda_{E,\CC}$ be the associated $k$-graph constructed in \eqref{k-graph from couloured}. For $a\in A$, and $x,y\in E^*$ traversing the same element of $s_A(a)\Lambda$, we have $\lambda_{a\cdot x}=\lambda_{a\cdot y}$. For each $p\in\N^k$ we then have a bijection $g_{a,p}:s_A(a)\Lambda^p\to r_A(a)\Lambda^p$ given by
\begin{equation}\label{deffa}
g_{a,p}(\mu)=\lambda_{a\cdot x}, \text{ where $x\in E^*$ traverses $\mu$.} 
\end{equation}
Moreover, $g_a:=\{g_{a,p}:p\in \N^k\}$ is a partial isomorphism of $s_A(a)\Lambda$ onto $r_A(a)\Lambda$ so that $\D(g_a)=s_A(a)$, $\c(g_a)=r_A(a)$.
\end{prop}

\begin{proof}
We first note that if $ef\sim f'g'$, then we have $\lambda_{ef}=\lambda_{f'e'}$. Hence if $x,y\in E^*$ differ by a square, then there are $\alpha,\beta\in E^*$ with $x=\alpha x_ix_{i+1}\beta, y=\alpha y_iy_{i+1}\beta$, and 
\[
\lambda_x=\lambda_\alpha \lambda_{x_ix_{i+1}}\lambda_\beta =\lambda_\alpha \lambda_{y_iy_{i+1}}\lambda_\beta = \lambda_y.
\] 
Now let $a\in A$, $\mu\in s_A(a)\Lambda$, and $x,y\in E^*$ that both traverse $\mu$. Then there exists a finite sequence $\{\alpha_j:1\leq j\leq l\}$ of paths in $E^*$ such that $\alpha_1= x$, $\alpha_l=y$, and $\alpha_j$ and $\alpha_{j+1}$ differ by a square for all $1\leq j\leq l-1$. It now follows from Lemma~\ref{lemma:k-aut}\eqref{lemma:k-aut-2} that
\[\lambda_{a\cdot x}=\lambda_{a\cdot \alpha_1}=\lambda_{a\cdot \alpha_2}=\cdots =\lambda_{a\cdot \alpha_l}=\lambda_{a\cdot y}.\]

For each $a\in A$ and $p\in\N^k$, the formula $g_{a,p}(\mu):=\lambda_{a\cdot x}$ does not depend on the choice of path $x\in E^*$ that traverses $\mu$, and hence we get a well-defined function $g_{a,p}\colon s_A(a)\Lambda^p\to r_A(a)\Lambda^p$. 

We now prove that $g_{a,p}$ is bijective for each $p\in \N^k$. For injectivity, we argue by induction on $|p|:=\sum_{i=1}^k p_i$. The base case $|p|=1$ follows from \eqref{item: bijection}. So we assume that the claim holds for $|p|=l\geq 1$. Suppose that $|p|=l+1$ and let $\mu,\nu\in s_A(a)\Lambda^p$ such that $g_{a,p}(\mu)=g_{a,p}(\nu)$. Since $|p|>1$, there is $1\leq j\leq k$ with $p_j>0$. Factor $\mu=\mu' \mu''$ and $\nu=\nu'\nu''$ such that $d(\mu')=d(\nu')=e_j$.
Let $x,y\in E^1$ such that $x$ traverses $\mu'$, and $y$ traverses $\nu'$. Let $x_1\dots x_l,y_1\dots y_l\in E^*$ such that $x_1\dots x_l$ traverses $\mu''$ and $y_1\dots y_l$ traverses $\nu''$. Since we know from \eqref{item: colour_preserve} that the action is colour preserving, we have
\[
a\cdot x=\lambda_{a\cdot (xx_1\dots x_l)}(v_j)=g_{a,p}(\mu)(v_j)=g_{a,p}(\nu)(v_j)=\lambda_{a\cdot(yy_1\dots y_l)}(v_j)=a\cdot y,
\]
and so $x=y$. Hence $\mu'=\lambda_x=\lambda_y=\nu'$. Now we have 
\begin{align*}
\lambda_{a\cdot x}\lambda_{a|_x\cdot(x_1\dots x_l)}=g_{a,p}(\mu)=g_{a,p}(\nu)=\lambda_{a\cdot y}\lambda_{a|_y\cdot (y_1\dots y_l)} &\implies  \lambda_{a|_x\cdot(x_1\dots x_l)}=\lambda_{a|_y\cdot (y_1\dots y_l)} \\
&\implies g_{a|_x,p-e_j}(\mu'') = g_{a|_y,p-e_j}(\nu'').
\end{align*}
The inductive hypothesis now gives $\mu''=\nu''$. Hence $\mu=\nu$. 

To see that $g_{a,p}$ is onto, let $\mu\in r_A(a)\Lambda^p$ and take $x\in r_A(a)E^{|p|}$ that traverses $\mu$. By Lemma~\ref{autpropert}\eqref{autpropert-2}, there is $y\in s_A(a)E^{|p|}$ such that $a\cdot y=x$. Let $\nu$ be the unique element of $\Lambda $ that is traversed by $a\cdot y$. Then $g_{a,p}(\nu)=\lambda_{a\cdot y}=\lambda_x=\mu$. Thus $g_{a,p}$ is onto.

We have now proved that $g_a:=\{g_{a,p}:p\in \N^k\}$ satisfies  Definition~\ref{def: partial iso}\eqref{prope1}. For Definition~\ref{def: partial iso}\eqref{prope2}, let $\mu\in s_A(a)\Lambda^p$, and take an edge $\xi\in s(\mu)\Lambda^{e_j}$. Let $x\in E^{|p|}$ and $y\in s(\mu)E^1$ such that $x$ traverses $\mu$, and $y$ traverses $\xi$. Note that $xy$ will traverse $\mu\xi$. Now we have 
\[g_a(\mu \xi)=g_{a,p+e_j}(\mu\xi)=\lambda_{xy}=\lambda_x\lambda_y=g_{a,p}(\mu)\lambda_y=g_a(\mu)\lambda_y\in g_a(\mu)\Lambda^{e_j}.\]
Hence Definition~\ref{def: partial iso}\eqref{prope2} holds, and so $g_a$ is a partial isomorphism.
\end{proof}

\begin{remark}\label{rem: gvisidentity}
If $a=v\in A\cap E^0$, then since $r_A(v)=v=s_A(v)$, and $v\cdot e=e$ and $v|_e=v$ for each edge $e\in vE^1$, it follows that $v\cdot x=x$ for all $x\in E^*$. Now, given $\mu\in v\Lambda^p$, and $x\in E^*$ that traverses $\mu$, we have
$g_v(\mu)=\lambda_{v\cdot x}=\lambda_{x}=\mu$. Hence $g_v$ is the identity on $v\Lambda$. 
\end{remark}

We can now prove that each $(E,\CC)$-automaton $A$ gives rise to a groupoid action on the $k$-graph $\Lambda_{E,\CC}$.

\begin{thm}\label{def:GA}
Let $A$ be an $(E,\CC)$-automaton, and let $\Lambda:=\Lambda_{E,\CC}$ be the associated $k$-graph constructed in \eqref{k-graph from couloured}. For each $a\in A$, let $g_a$ be the partial isomorphism of $\Lambda$ described in Proposition~\ref{prop:aut}, and let $G_A$ be the subgroupoid of $\PI(\Lambda)$ generated by $\{g_a:a\in A\}$. Then $(G_A,\Lambda)$ is a self-similar groupoid action.
\end{thm}

\begin{proof} First note that since $G_A$ is a subgroupoid of $\PI(\Lambda)$, it acts faithfully on $\Lambda$ by definition. 

Fix $a\in A$. Let $\xi \in \D(g_a)\Lambda^{e_i}$ and $\mu\in s(\xi)\Lambda^p$. Suppose that $e\in E^1$ traverses $\xi$, and $x\in E^{|p|}$ traverses $\mu$. Now, \eqref{lambda-homo} implies that
\begin{equation}\label{restriction-thm}
g_a(\xi\mu)=\lambda_{a\cdot(ex)}=\lambda_{(a\cdot e)(a|_e\cdot x)}=\lambda_{a\cdot e}\lambda_{a|_e\cdot x}=g_{a}(\xi)g_{a|_e}(\mu),
\end{equation}
 and therefore $g_{a}$ satisfies the self-similar condition \eqref{selfsimilar groupoid defn}. 

It now suffices to prove that inverses and the composition of the generators of $G_A$ also satisfy \eqref{selfsimilar groupoid defn}.
Consider $g_a^{-1}\in G_A$. Let $\eta \in \c(g_a)\Lambda^{e_i}$ for some $i$, and $\nu\in s(\eta)\Lambda^p$. Let $\xi\in\D(g_a)\Lambda^{e_i}$ and $\mu\in s(\xi)\Lambda$ such that $g_a^{-1}(\eta\nu)=\xi\mu$. Then by \eqref{restriction-thm} we have $\eta \nu =g_a(\xi\mu)=g_a(\xi)g_{a|_e}(\mu)$. The factorisation property implies that $\eta =g_a(\xi)$ and $\nu =g_{a|_e}(\mu)$. Thus we have $\xi=g_a^{-1}(\eta)$ and $\mu=g_{a|_e}^{-1}(\nu)$, and so
\[g_a^{-1}(\eta\nu)=g_a^{-1}(\eta)g_{a|_e}^{-1}(\nu). \]
Hence $g_a^{-1}$ satisfies \eqref{selfsimilar groupoid defn}.

Now let $g_b,g_a\in G_A$ be composable. Let $\eta \in \D(g_a)\Lambda^{e_i}$ for some $i$, and $\nu\in s(\eta)\Lambda$. Since $g_a$ act self-similarly, there exists a unique $g_{a'}\in G_A$ such that $g_a(\xi\mu)=g_a(\xi)g_{a'}(\mu)$. Since $g_a(\xi)$ is an edge in $\Lambda$ and since $g_b$ acts self-similarly, there is a unique element $g_{b'}\in G_A$ such that $(g_bg_a)(\xi\mu)=g_b(g_a(\xi))g_{b'}\big(g_{a'}(\mu)\big)=(g_bg_a)(\xi)\big(g_{b'}g_{a'}\big)(\mu)$. Thus $g_bg_a$ acts self-similarly.
\end{proof}

\begin{example}\label{ex1}
Let $\Gamma$ be a $1$-coaligned $(k+1)$-graph, in the sense that for all $1\leq i\neq j\leq k+1$ and $(\lambda,\mu)\in \Gamma^{e_i}\times \Gamma^{e_j}$ with $s(\lambda)=s(\mu)$, there exists a unique pair $(\eta,\zeta)\in \Gamma^{e_j}\times \Gamma^{e_i}$ such that $\eta\lambda=\zeta\mu$. Fix $1\leq i\leq k+1$. Let $A:=\Gamma^{e_i}\cup \Gamma^0$, $E_\Gamma$ be the skeleton of $\Gamma$ and $\CC_\Gamma$ the complete collection of squares associated to $\Gamma$ (see \cite[Definition~4.1]{HRSW}). Let $E:=(E_\Gamma^0,E_\Gamma^1\setminus c^{-1}(c_i), r,s)$, and $\CC$ be the collection of squares in $\CC_\Lambda$ which do not involve edges of colour $c_i$. Define $r_A, s_A:A\to E^0$ as range and source maps of $\Gamma$.
For each $v\in A\cap \Gamma^0$ and $e\in E^1$ with $r(e)=v$, define $v\cdot e= e$ and $v|_e=s(e)$. 
 For each $a\in A$ and $e\in E^1$ with $r(e)=s_A(a)$, since $a, e$ are two edges of $\Gamma$ with different degrees, there is a unique $b\in A$ and $f\in E^1\cap c^{-1}(c(e))$ such that 
$ae=fb$. We define $a|_e:=b$ and $a\cdot e=f$, and we claim that $A$ is an $(E,\CC)$-automaton. Property \eqref{item: bijection} follows from the factorisation property and the $1$-coalignedness of $\Gamma$. Properties \eqref{item: colour_preserve}--\eqref{item: range_and_source} hold true by construction. To see \eqref{item: square_preserving_1} and \eqref{item: square_preserving_2}, let $a\in A$, $e\in s_A(a)E^1$, and $ef\sim f'e'$. Then
\[aef=(a\cdot e)(a|_e)f=(a\cdot e)(a|_e\cdot f)((a|_e)|_f)\]
and
\[af'e'=(a\cdot f')(a|_{f'})e'=(a\cdot f')(a|_{f'}\cdot e')((a|_{f'})|_{e'}).\]
The factorisation property now gives \eqref{item: square_preserving_1} and \eqref{item: square_preserving_2}.
\end{example}

\begin{remark}
Suppose that $E$ is a $k$-coloured graph and $\CC$ is an associative complete collection of squares in $E$. If $A$ is an $(E,\CC)$-automaton such that the map \eqref{autmap} is a bijection, then by giving an additional colour to the elements of 
$A$ we obtain a $(k+1)$-coloured graph $F$. Since the map \eqref{autmap} is a bijection, for each pair $(a,e)\in A \tensor[_{s_A}]{\times}{_{r_E}} E^1 $, $ae\sim (a\cdot e)a|_e$ is a square in $F$. Adding these squares to $\CC$, we get a complete collection of squares for $F$ that is associative by \eqref{item: square_preserving_1} and \eqref{item: square_preserving_2}. We note that for a general $(E,\CC)$-automaton $A$, one may not be able to find a complete and associative collection of squares for the new graph.
\end{remark}

We next give two examples $(E,\CC)$-automata for $E$ a $2$-coloured graph.

\begin{example}\label{ex: single vertex}
Consider the $2$-coloured graph $E$ in Figure \ref{fig:single vertex} with the complete collection of squares 
\[
\CC: \quad e_1 f_1\sim f_1 e_1, \,\, e_1 f_2\sim f_2 e_1,\,\, e_2 f_1\sim f_1 e_3,\,\, e_2 f_2\sim f_2 e_2,\,\, e_3 f_1\sim f_1 e_2 \,\text{ and }\,  e_3f_2\sim f_2 e_3.
\]
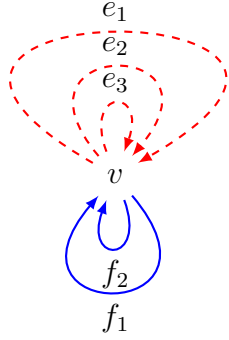
\begin{figure}[h]
\vspace{-0.75cm}
\begin{tikzpicture}
\node at (0,0) {};
\node[vertex] (vertexv) at (0,0)   {$v$}
	edge [->,>=latex,out=150,in=30,loop,thick,looseness=20,red,dashed] node[above]{\color{black} $e_1$} (vertexv)
	edge [->,>=latex,out=130,in=50,loop,thick,looseness=12,red,dashed] node[above]{\color{black} $e_2$} (vertexv)
	edge [->,>=latex,out=110,in=70,loop,thick,looseness=10,red,dashed] node[above]{\color{black} $e_3$} (vertexv)
	edge [->,>=latex,out=-70,in=250,loop,thick,looseness=10,blue] node[below]{\color{black} $f_2$} (vertexv)
	edge [->,>=latex,out=-50,in=230,loop,thick,looseness=13,blue] node[below]{\color{black} $f_1$} (vertexv);
\end{tikzpicture}
\vspace{-0.5cm}
\caption{The 2-graph $\Lambda$ of Example \ref{ex: single vertex}.}
\label{fig:single vertex}
\end{figure}

Using \eqref{k-graph from couloured} we obtain a $2$-graph $\Lambda:=\Lambda_{(E,\CC)}$. To define an $(E,\CC)$-automaton, let $A:=\{a\} \cup E^0=\{a,v\}$ with $s_A(a)=v=r_A(a)$. We define
\begin{align*}
a \cdot e_1=e_1, \quad a|_{e_1}=v; & \qquad a \cdot f_1=f_1, \quad a|_{f_1}=a;\\
a \cdot e_2=e_3, \quad a|_{e_2}=v; & \qquad a \cdot f_2=f_2, \quad a|_{f_2}=a;\\
a \cdot e_3=e_2, \quad a|_{e_3}=v, &
\end{align*}
and note that the restrictions on $E^0$ are specified in \eqref{item: range_and_source} of Definition \eqref{defn: (E,C)-automaton}. To verify \eqref{item: bijection} and \eqref{item: colour_preserve}, we immediately read off that $a$ is a colour preserving bijection on $E^1$. Both \eqref{second property of automaton} and \eqref{item: range_and_source} also follow immediately since $E^0$ consists of a single vertex. We check \eqref{item: square_preserving_1} and \eqref{item: square_preserving_2} for the commuting squares $\CC$:
\begin{align*}
&(a \cdot e_1)(a|_{e_1} \cdot f_1)=e_1 f_1\sim f_1 e_1= (a \cdot f_1)(a|_{f_1} \cdot e_1)&&(a|_{e_1})_{f_1}=v=(a|_{f_1})|_{e_1};\\
&(a \cdot e_1)(a|_{e_1} \cdot f_2)=e_1f_2\sim f_2 e_1= (a \cdot f_2)(a|_{f_2} \cdot e_1)&&(a|_{e_1})_{f_2}=v=(a|_{f_2})|_{e_1};\\
& (a \cdot e_2)(a|_{e_2} \cdot f_1)=e_3f_1\sim f_1 e_2= (a \cdot f_1)(a|_{f_1} \cdot e_3)&&(a|_{e_3})_{f_1}=v=(a|_{f_1})|_{e_2};\\
&(a \cdot e_2)(a|_{e_2} \cdot f_2)=e_3f_2\sim f_2 e_3= (a \cdot f_2)(a|_{f_2} \cdot e_2)&&(a|_{e_3})_{f_2}=v=(a|_{f_2})|_{e_3};\\
&(a \cdot e_3)(a|_{e_3} \cdot f_1)=e_2f_1\sim f_1 e_3= (a \cdot f_1)(a|_{f_1} \cdot e_2)&&( a|_{e_2})_{f_1}=v=(a|_{f_1})|_{e_3};\\
&(a \cdot e_3)(a|_{e_3} \cdot f_2)=e_2f_2\sim f_2 e_2= (a \cdot f_2)(a|_{f_2} \cdot e_3)&& (a|_{e_2})_{f_2}=v=(a|_{f_2})|_{e_2}.
\end{align*}
Theorem~\ref{def:GA} gives us a self-similar groupoid action $(G_A,\Lambda)$. Notice that since there is only one vertex, all elements of $G_A$ are composable. Hence $G_A$ is a group. 
\end{example}

The next example is inspired by the basilica group defined in \cite{nek-book}.

\begin{example}\label{ex: bascilica2}
Consider the $2$-coloured graph $E$ in Figure \ref{fig:two vertex} with the complete collection of squares 
\[
\CC: \quad
e_1 f_1\sim f_3 e_1, \,\, e_1 f_2\sim f_4 e_1,\,\, e_2 f_3\sim f_1 e_2,\,\, e_2 f_4\sim f_2 e_2,\,\, e_3 f_1\sim f_1 e_3 \,\text{ and }\,  e_3f_2\sim f_2 e_3.
\]
Using \eqref{k-graph from couloured} we obtain a $2$-graph $\Lambda:=\Lambda_{(E,\CC)}$.
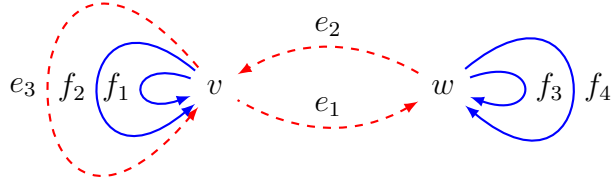
\begin{figure}[h]
\vspace{-2cm}
\begin{tikzpicture}
\node[vertex] (vertexv) at (0,0)   {$v$}
	edge [->,>=latex,out=130,in=230,loop,thick,looseness=20,red,dashed] node[left]{\color{black} $e_3$} (vertexv)
	edge [->,>=latex,out=160,in=200,loop,thick,looseness=10,blue] node[left]{\color{black} $f_1$} (vertexv)
	edge [->,>=latex,out=140,in=220,loop,thick,looseness=13,blue] node[left]{\color{black} $f_2$} (vertexv);
\node[vertex] (vertexw) at (3,0)   {$w$}
	edge [->,>=latex,out=150,in=30,thick,red,dashed] node[above]{\color{black} $e_2$} (vertexv)
	edge [<-,>=latex,out=210,in=-30,thick,red,dashed] node[above]{\color{black} $e_1$} (vertexv)
	edge [->,>=latex,out=20,in=-20,loop,thick,looseness=10,blue] node[right]{\color{black} $f_3$} (vertexw)
	edge [->,>=latex,out=40,in=-40,loop,thick,looseness=13,blue] node[right]{\color{black} $f_4$} (vertexw);
\end{tikzpicture}
\vspace{-2cm}
\caption{The 2-graph $\Lambda$ of Example \ref{ex: bascilica2}.}
\label{fig:two vertex}
\end{figure}

To define an $(E,\CC)$-automaton, let $A:=\{a_v,a_w,b_v,b_w\} \cup E^0$  where each element has range and source equal to its subscript. We define
the action and restriction maps by:
\begin{align*}
a_v \cdot e_2=e_2,& \quad a_v|_{e_2}=a_w&&b_v \cdot e_2=e_2,&&  b_v|_{e_2}=b_w \\
a_v \cdot e_3=e_3,& \quad a_v|_{e_3}=a_v&& b_v \cdot e_3=e_3,&&  b_v|_{e_3}=b_v \\
a_v \cdot f_1=f_2,& \quad a_v|_{f_1}=b_v && b_v \cdot f_1=f_1,&&  b_v|_{f_1}=a_v \\
a_v \cdot f_2=f_1,& \quad a_v|_{f_2}=v && b_v \cdot f_2=f_2,&&  b_v|_{f_2}=v \\
a_w \cdot e_1=e_1,&\quad a_w|_{e_1}=a_v&& b_w \cdot e_1=e_1,&&  b_w|_{e_1}=b_v  \\
a_w \cdot f_3=f_4,& \quad a_w|_{f_3}=b_w&& b_w \cdot f_3=f_3,&&  b_w|_{f_3}=a_w \\
a_w \cdot f_4=f_3,& \quad a_w|_{f_4}=w&& b_w \cdot f_4=f_4,&&  b_w|_{f_4}=w,
\end{align*}
and note that the restrictions on $E^0$ are specified in \eqref{item: range_and_source} of Definition~\ref{defn: (E,C)-automaton}. Since  elements of $A$ are colour preserving bijections on $E^1$,  \eqref{item: bijection} and \eqref{item: colour_preserve} are immediate.  Both \eqref{second property of automaton} and \eqref{item: range_and_source} are easily checked using the restrictions maps above. We check \eqref{item: square_preserving_1}   and \eqref{item: square_preserving_2} for the commuting squares $\CC$:
\begin{align*}
&(a_w \cdot e_1)(a_w|_{e_1} \cdot f_1)=e_1 f_2\sim f_4 e_1= (a_w \cdot f_3)(a_w|_{f_3} \cdot e_1)  && a_w|_{e_1 f_1}=b_v=a_w|_{f_3 e_1}\\
&(a_w \cdot e_1)(a_w|_{e_1} \cdot f_2)=e_1f_1\sim f_3 e_1= (a_w \cdot f_4)(a_w|_{f_4} \cdot e_1)  && a_w|_{e_1 f_2}=v=a_w|_{f_4 e_1}\\
&(a_v \cdot e_2)(a_v|_{e_2} \cdot f_3)=e_2f_4\sim f_2 e_2= (a_v \cdot f_1)(a_v|_{f_1} \cdot e_2) && a_v|_{e_2 f_3}=b_w=a_v|_{f_1 e_2}\\
&(a_v \cdot e_2)(a_v|_{e_2} \cdot f_4)=e_2f_3\sim f_1 e_2= (a_v \cdot f_2)(a_v|_{f_2} \cdot e_2) && a_v|_{e_2 f_4}=w=a_v|_{f_2 e_2} \\
&(a_v \cdot e_3)(a_v|_{e_3} \cdot f_1)=e_3f_2\sim f_2 e_3= (a_v \cdot f_1)(a_v|_{f_1} \cdot e_3) && a_v|_{e_3 f_1}=b_v=a_v|_{f_1 e_3} \\
&(a_v \cdot e_3)(a_v|_{e_3} \cdot f_2)=e_3f_1\sim f_1 e_3= (a_v \cdot f_2)(a_v|_{f_2} \cdot e_3) && a_v|_{e_3 f_2}=v=a_v|_{f_2 e_3}\\
&&\\
&(b_w \cdot e_1)(b_w|_{e_1} \cdot f_1)=e_1 f_1\sim f_3 e_1= (b_w \cdot f_3)(b_w|_{f_3} \cdot e_1) && b_w|_{e_1 f_1}=a_v=b_w|_{f_3 e_1} \\
&(b_w \cdot e_1)(b_w|_{e_1} \cdot f_2)=e_1f_2\sim f_4 e_1= (b_w \cdot f_4)(b_w|_{f_4} \cdot e_1) && b_w|_{e_1 f_2}=v=b_w|_{f_4 e_1} \\
&(b_v \cdot e_2)(b_v|_{e_2} \cdot f_3)=e_2f_3\sim f_1 e_2= (b_v \cdot f_1)(b_v|_{f_1} \cdot e_2) &&  b_v|_{e_2 f_3}=a_w=b_v|_{f_1 e_2} \\
&(b_v \cdot e_2)(b_v|_{e_2} \cdot f_4)=e_2f_4\sim f_2 e_2= (b_v \cdot f_2)(b_v|_{f_2} \cdot e_2) &&  b_v|_{e_2 f_4}=w=b_v|_{f_2 e_2}  \\
&(b_v \cdot e_3)(b_v|_{e_3} \cdot f_1)=e_3f_1\sim f_1 e_3= (b_v \cdot f_1)(b_v|_{f_1} \cdot e_3) && b_v|_{e_3 f_1}=a_v=b_v|_{f_1 e_3} \\
&(b_v \cdot e_3)(b_v|_{e_3} \cdot f_2)=e_3f_2\sim f_2 e_3= (b_v \cdot f_2)(b_v|_{f_2} \cdot e_3) && b_v|_{e_3 f_2}=v=b_v|_{f_2 e_3}
\end{align*}
Theorem~\ref{def:GA} gives us a self-similar groupoid action $(G_A,\Lambda)$. 
\end{example}
\begin{remark}
In Example~\ref{ex: bascilica2}, if for $g\in\{a_v,a_w,b_v,b_w\}$, we define $g\cdot e:=g_{r(e)} \cdot e$ and $g|_{e}:=g_{r(e)}|_{e}$ then we obtain a group rather than a groupoid. The group we obtain has exactly the same relations as the basilica group. 
\end{remark}

\section{Product systems from self-similar actions}
Our goal is to associate $C^*$-algebras to self-similar groupoid actions on $k$-graphs. In this section we construct a product system of right Hilbert bimodules from a self-similar groupoid action on a $k$-graph, and then in Section~\ref{universal-algebras}, we will describe the $C^*$-algebras associated to this product system through generators and relations. 

	Let $E$ be a finite directed graph, and $(G,E)$ be a self-similar groupoid action in the sense of \cite{lrrw}.  Let $X(E)$ be the corresponding graph bimodule as in \cite[Chapter~8]{R_book}. Recall from \cite[pages~282--284]{lrrw} that
	$M(G,E):=X(E)\otimes_{C(E^0)}C^*(G)$ is a right Hilbert $C^*(G)$-bimodule with the module actions
\begin{equation}\label{eq: rightaction}
(x\otimes a)\cdot i_g=(x\otimes ai_g),
\end{equation}
\begin{equation}\label{eq: rightaction1}
i_g\cdot(\delta_\lambda \otimes a)=\begin{cases}
(\delta_{g\cdot \lambda}\otimes i_{g|_\lambda} a) \notag&\text{if $r(\lambda)=\D(g)$}\\
	0\notag&\text{otherwise.}\\
	\end{cases}
\end{equation}
and inner product
\begin{equation}\label{eq: rightaction2}
\langle x\otimes a,y\otimes b\rangle= a^*\langle x,y \rangle b.
\end{equation}
The bimodule $M(G,E)$ is spanned by $\{\delta_\lambda\otimes i_g:g\in G , \lambda\in E^1, s(\lambda)=\c(g)\}$, and the elements $\{\delta_\lambda\otimes 1:\lambda\in E^1\}$ form a Parseval frame for $M(G,E)$ with the reconstruction formula
\begin{equation}\label{recons-for}
m=\sum_{\lambda\in E^1}\delta_\lambda\otimes\langle \delta_\lambda\otimes 1, m\rangle,\quad \text { for } m\in M(G,E).
\end{equation}

We now consider a self-similar groupoid action on a $k$-graph, and we generalise the construction of $M(G,E)$ to a product system of right Hilbert bimodules.

\begin{prop}\label{prop:product system}
	Let $\Lambda$ be a finite $k$-graph and suppose that $(G,\Lambda)$ is a self-similar groupoid action. For each nonzero $p\in \N^k$, $\Lambda_p:=(\Lambda^0,\Lambda^p,r|_{\Lambda^p},s|_{\Lambda^p})$ is a directed graph, and $(G,\Lambda_p)$ is a self-similar groupoid action on a directed graph. With $M_0:=C^*(G)$, and $M_p:=M(G,\Lambda_p)$ the right Hilbert $C^*(G)$-bimodule associated to $(G,\Lambda_p)$, $M:=\bigsqcup_{p\in \N^k}M_p$ is a product system with multiplication characterised by
\begin{align}	\label{multiplicationformula1}
	(\delta_\lambda\otimes i_g)(\delta_\mu\otimes i_h)=\begin{cases}
\delta_{\lambda (g\cdot \mu)}\otimes i_{(g|_\mu)h}&\text{if $r(\mu)=\D(g)$}\\
	0&\text{otherwise,}\\
	\end{cases}
	\end{align}	
for nonzero $p,q\in \N^k$, $\delta_\lambda\otimes i_g\in M_p$, $\delta_\mu\otimes i_h\in M_q$. Moreover, each $M_p$ is essential, the left action of $M_0$ on each $M_p$ is by compact operators, and $M$ is compactly aligned.
\end{prop}

We first note that it follows immediately from the definitions that for each nonzero $p\in\N^k$, $\Lambda_p:=(\Lambda^0,\Lambda^p,r|_{\Lambda^p},s|_{\Lambda^p})$ is a directed graph, and $(G,\Lambda_p)$ is a self-similar groupoid action on a directed graph in the sense of \cite{LRRW}. 

We set $M_0:=C^*(G)$, and $M_p:=M(G,\Lambda_p)$  for each nonzero $p\in\N^k$. Whenever we write $\delta_\lambda\otimes i_g\in M_p$, we assume that $s(\lambda)=\c(g)$.

\begin{lemma}\label{lemma:product system iso}	
Let $\Lambda$ be a finite $k$-graph and suppose that $(G,\Lambda)$ is a self-similar groupoid action. For each nonzero $p,q\in \N^k$, there is an isomorphism $\pi_{p,q}$ from $M_p\otimes_{C^*(G)} M_q$ onto $M_{p+q}$ such that
	\begin{align}\label{equ:isomor}
	\pi_{p,q}\big((\delta_\lambda\otimes i_g)\otimes(\delta_\mu\otimes i_h)\big)=\begin{cases}
\delta_{\lambda (g\cdot \mu)}\otimes i_{(g|_\mu)h}&\text{if $r(\mu)=\D(g)$}\\
	0&\text{otherwise,}\\
	\end{cases}
	\end{align}
for all $g,h\in G$, $\lambda\in \Lambda^p$, and $\mu\in\Lambda^q$ with $s(\lambda)=\c(g)$ and $s(\mu)=\c(h)$. 
\end{lemma}
\begin{proof}
	Since the elements $\{\delta_\lambda\otimes i_g:\lambda\in \Lambda^p, g\in G , s(\lambda)=\c(g)\}$ span $ M_p$, the formula
	\begin{align}\label{equ:preiso}
	\pi\big(\delta_\lambda\otimes i_g,\delta_\mu\otimes i_h\big)=\begin{cases}
\delta_{\lambda (g\cdot \mu)}\otimes i_{(g|_\mu)h}&\text{if $r(\mu)=\D(g)$}\\
	0&\text{otherwise.}\\
	\end{cases}
	\end{align}	
gives a map $\pi: M_p\times M_q\rightarrow M_{p+q}$. Clearly $\pi$ is bilinear. We also have that $\pi$ is  surjective. Indeed, since for any $\nu\in \Lambda^{p+q}$ and $h\in G$ with $s(\nu)=\c(h)$, we have $\pi( \delta_{\nu(0,p)}\otimes i_{\nu(p)},\delta_{\nu(p,p+q)}\otimes i_h)=\delta_\nu\otimes i_h$. By the universal property of the algebraic tensor product $M_p\odot M_q$, $\pi$ descends to a unique surjective linear map $\tilde{\pi}:M_p\odot M_q\rightarrow M_{p+q}$ such that $\tilde{\pi}((\delta_\lambda\otimes i_g)\odot(\delta_\mu\otimes i_h))=\pi\big(\delta_\lambda\otimes i_g,\delta_\mu\otimes i_h\big)$. Since for each $f\in G$ with $\c(f)=\D(g)$ we have
	$ \tilde{\pi}((\delta_\lambda\otimes i_g)\cdot i_f\odot (\delta_\mu\otimes i_h))=\tilde{\pi}((\delta_\lambda\otimes i_g)\odot i_f\cdot (\delta_\mu\otimes i_h))$,
	 $\tilde{\pi}$ induces a surjective linear map $\pi_{p,q}:M_p\odot_{C^*(G)} M_q\rightarrow M_{p+q}$ that satisfies \eqref{equ:isomor}.	
		
To see that $\pi_{p,q}$ preserves left actions, let $ (\delta_\lambda\otimes i_g)\otimes (\delta_\mu\otimes a)\in M_p\odot_{C^*(G)} M_q$ and $i_f\in C^*(G)$. Since 
	$i_f\cdot (\delta_\lambda\otimes i_g)=\delta_{\D(f),r(\lambda)}(\delta_{f\cdot \lambda}\otimes i_{f|_\lambda g})$, the formula \eqref{equ:isomor} implies that 
	\begin{align*}
\pi_{p,q}\Big(i_f\cdot \big(& (\delta_\lambda\otimes i_g)\otimes(\delta_\mu\otimes a)\big)\Big)=\pi_{p,q}\Big(\big(i_f\cdot (\delta_\lambda\otimes i_g)\big)\otimes(\delta_\mu\otimes a)\Big)\\
&=\begin{cases}
\delta_{(f\cdot \lambda)((f|_\lambda g)\cdot \mu)}\otimes i_{(f|_\lambda g)|_\mu}a&\text{if $r(\lambda)=\D(f)$ and $\D(g)=r(\mu)$}\\
	0\notag&\text{otherwise.}\\
	\end{cases}
	\end{align*}
We also have
	\begin{align*}
i_f\cdot \pi_{p,q}\big(&(\delta_\lambda\otimes i_g)\otimes(\delta_\mu\otimes a)\big)=\delta_{\D(g),r(\mu)}i_f\cdot(\delta_{\lambda(g\cdot \mu)}\otimes i_{g|_\mu }a)\\
&=\begin{cases}
\delta_{f\cdot (\lambda(g\cdot \mu))}\otimes i_{f|_{\lambda ( g\cdot \mu)}g|_\mu}a&\text{if $r(\mu)=\D(g)$ and $\D(f)=r(\lambda)$ }\\
	0\notag&\text{otherwise.}\\
	\end{cases}\\
	\end{align*}
	Now, (\ref{S1}) and (\ref{S6}) from Lemma~\ref{lem:properties} imply that $f\cdot (\lambda(g\cdot \mu))=(f\cdot \lambda)((f|_\lambda) g)\cdot \mu$ and $((f|_\lambda )g)|_\mu=f|_{\lambda (g\cdot \mu)}g|_\mu$, and it follows that $\pi_{p,q}$ preserves left actions.

A straightforward calculation shows that 
\[
\pi_{p,q}\big((\delta_\lambda\otimes i_g)\otimes(\delta_\mu\otimes a)\big)\cdot i_f = \delta_{\D(g),r(\mu)}\delta_{\lambda(g\cdot \mu)}\otimes i_{g|_ \mu}a i_ f= \pi_{p,q}\big((\delta_\lambda\otimes i_g)\otimes(\delta_\mu\otimes a)\cdot i_f\big),
\]
and hence $\pi_{p,q}$ preserves right actions.
		
	To see that $\pi_{p,q}$ preserves inner products, let $(\delta_\lambda\otimes i_g)\otimes (\delta_\mu\otimes a), (\delta_{\nu}\otimes i_{h})\otimes (\delta_{\tau}\otimes b)\in M_p\odot_{C^*(G)} M_q$. We have
	\begin{align*}
	\Big\langle (\delta_\lambda\otimes i_g)\notag&\otimes (\delta_\mu\otimes a),(\delta_{\nu}\otimes i_{h})\otimes (\delta_{\tau}\otimes b)\Big\rangle\\
	\notag&=\Big\langle \delta_\mu\otimes a,\big\langle\delta_\lambda\otimes i_g,\delta_{\nu}\otimes i_{h}\big\rangle \cdot (\delta_{\tau}\otimes b)\Big\rangle\\
		\notag&=\begin{cases}
	\big\langle \delta_\mu\otimes a,i_{g^{-1}h} \cdot (\delta_{\tau}\otimes b)\big\rangle&\text{if $\lambda =\nu$ and $\c(h)=\c(g)$}\\
	0&\text{otherwise.}\\
	\end{cases}\\
	\notag&=\begin{cases}
	\big\langle \delta_\mu\otimes a,\delta_{(g^{-1}h)\cdot\tau}\otimes i_{(g^{-1}h)|_{\tau}}b\big\rangle&\text{if $\lambda =\nu$ and $\c(h)=\c(g)$}\\
	0&\text{otherwise.}\\
	\end{cases}\\
	&=\begin{cases}
	a^*i_{s((g^{-1}h)\cdot\tau)}i_{(g^{-1}h)|_{\tau}}b&\text{if $\lambda =\nu$, $\c(h)=\c(g)$ and $(g^{-1}h)\cdot\tau=\mu$}\\
	0&\text{otherwise.}\\
	\end{cases}
	\end{align*}
	Since $s((g^{-1}h)\cdot\tau)=\c((g^{-1}h)|_{\tau})$, we have $i_{s((g^{-1}h)\cdot\tau)}i_{(g^{-1}h)|_{\tau}}=i_{(g^{-1}h)|_{\tau}}$, and hence
	\begin{align}\label{eq:inerright}
	\big\langle (\delta_\lambda\otimes i_g)\otimes (\delta_\mu\otimes a),\notag&(\delta_{\nu}\otimes i_{h})\otimes (\delta_{\tau}\otimes b)\big\rangle\\
	&=\begin{cases}
	a^*i_{(g^{-1}h)|_{\tau}}b&\text{if $\lambda =\nu$, $\c(h)=\c(g)$ and $(g^{-1}h)\cdot\tau=\mu$}\\
	0&\text{otherwise.}\\
	\end{cases}
	\end{align}
We also have
\begin{align}\label{eq:inerleft}
	\big\langle\pi_{p,q} \big((\delta_\lambda\otimes i_g)\notag&\otimes (\delta_\mu\otimes a)\big),\pi_{p,q}\big((\delta_{\nu}\otimes i_{h})\otimes (\delta_{\tau}\otimes b)\big)\big\rangle\\
	\notag&=\big\langle \delta_{\lambda (g\cdot\mu)}\otimes i_{(g|_\mu)}a,\delta_{\nu (h\cdot\tau)}\otimes i_{(h|_{\tau})}b\big\rangle\\
\notag&=\begin{cases}
	a^* i_{(g|_\mu)^{-1}}i_{(h|_{\tau})}b&\text{if $\lambda (g\cdot\mu)=\nu (h\cdot\tau)$}\\
	0&\text{otherwise.}\\
	\end{cases}\\
		&=\begin{cases}
	a^* i_{(g|_\mu)^{-1}(h|_{\tau})}b&\text{if $\lambda (g\cdot\mu)=\nu (h\cdot\tau)$ and $\c(g|_\mu)=\c(h|_{\tau})$}\\
	0&\text{otherwise.}\\
	\end{cases}
	\end{align}
	
	If $\lambda =\nu$, $\c(h)=\c(g)$ and $(g^{-1}h)\cdot\tau=\mu$, then $ \lambda(g\cdot\mu)=\nu(g\cdot( (g^{-1}h)\cdot\tau))=\nu(h\cdot \tau)$ and $\c(g|_\mu)=s(g\cdot \mu)=s(h\cdot \tau)=\c(h|_{\tau})$. Conversely if $\lambda (g\cdot\mu)=\nu(h\cdot\tau)$, then the factorisation property implies that $\lambda=\nu$ and $h\cdot\tau=g\cdot \mu$. Then $\c(g)=r(g\cdot \mu)=r(h\cdot\tau)=\c(h)$ and $(g^{-1}h)\cdot\tau=g^{-1}(h\cdot\tau)=g^{-1}(g\cdot \mu)=\mu$. Now, using (\ref{S7})	we see that
$(g|_\mu)^{-1}=(g|_{(g^{-1}h)\cdot\tau})^{-1}=(g|_{g^{-1}(h\cdot\tau)})^{-1}=g^{-1}|_{h\cdot \tau}.$ So $(g|_\mu)^{-1}(h|_{\tau})=g^{-1}|_{h\cdot \tau}(h|_{\tau})=(g^{-1}h)|_{\tau}$.	
Hence \eqref{eq:inerright}	and \eqref{eq:inerleft}	are equal and therefore $\pi_{p,q}$ preserves inner products. 

It then follows that $\pi_{p,q}$ is an isometry on $M_p\odot_{C^*(G)} M_q$, and hence it extends to an isomorphism $\pi_{p,q}$ of $M_p\otimes_{C^*(G)}M_q$ onto $M_{p+q}$ which satisfies \eqref{equ:isomor}.
\end{proof}
\begin{remark}\label{remark:identification}
Let $M_0$ be the bimodule generated by $\{\delta_{\c(g)}\otimes i_g:g\in G\}$ as in Lemma~\ref{lemma:product system iso} and consider the standard bimodule $_{C^*(G)} C^*(G)_{ C^*(G)}$. The canonical map $(\delta_{\c(g)},i_g)\mapsto i_g:X(\Lambda^0)\odot_{C^*(G)}C^*(G)\rightarrow C^*(G)$ preserves the actions and the inner products and hence extends to a (right Hilbert bimodules) isomorphism from $M_0$ onto $_{C^*(G)} C^*(G)_{ C^*(G)}$. Now for $(\delta_{\c(g)}\otimes i_g)\in M_0$ and $(\delta_\lambda\otimes i_h)\in M_p$ with $\c(h)=s(\lambda)$ and $p\neq 0$, we have 
	\begin{align*}
(\delta_\lambda\otimes i_h)(\delta_{\c(g)}\otimes i_g)
	=\delta_{\D(h),\c(g)}(\delta_{\lambda (h\cdot \D(h))}\otimes i_{h|_{\D(h)}g})&=(\delta_\lambda\otimes i_h)\cdot i_g.
\end{align*}
and
\begin{align*}(\delta_{\c(g)}\otimes i_g)(\delta_\lambda\otimes i_h)
		&=i_g\cdot(\delta_\lambda\otimes i_h),
\end{align*}
and the formula \eqref{equ:isomor} for $\pi_{0,p}:M_0\otimes M_p\to M_0$ and $\pi_{0,p}:M_p\otimes M_0\to M_0$ coincide with the left and right actions of $C^*(G)$ on $M_p$ respectively. Thus, to ease the computations, from now on, whenever we need to view $C^*(G)$ as the standard bimodule, we will identify $C^*(G)$ with $M_0$ and use the formula \eqref{equ:isomor} for the actions. 
\end{remark}

We need the following lemma to prove the compact alignment of our product system. For each $p\in\N^k$, and $\delta_{\lambda}\otimes i_{g},\delta_{\mu}\otimes i_{h}\in M_p$, we define $\Theta_{\lambda,g,\mu,h}$ to be the rank-one operator
\[
\Theta_{\lambda,g,\mu,h} := \Theta_{(\delta_{\lambda}\otimes i_{g}), (\delta_{\mu}\otimes i_{h})} \in \KK(M_p).
\]

\begin{lemma}\label{lemma:compactly-alig}
Let $p,q\in \N^k$, $\delta_{\lambda_1}\otimes i_{g_1},\delta_{\lambda_2}\otimes i_{g_2}\in M_p$ and $\delta_{\mu_1}\otimes i_{h_1},\delta_{\mu_2}\otimes i_{h_2}\in M_q$. Then the product
\begin{align}\label{eq:comp-alig}
\big(\Theta_{\lambda_1,g_1,\lambda_2,g_2}\big)_p^{p\vee q}\big(\Theta_{\mu_1,h_1,\mu_2,h_2}\big)_q^{p\vee q}
\end{align}
is equal to
\begin{align}\label{eq:comp-alig2}
\sum_{(\eta,\zeta)\in\Lambda^{\min}(\lambda_2, \mu_1)}\Theta_{\lambda_1((g_1g_2^{-1})\cdot\eta)\,,\, (g_1g_2^{-1})|_\eta\,,\,\mu_2((h_2h_1^{-1})\cdot\zeta)\,,\,(h_2h_1^{-1})|_\zeta}
\end{align}
whenever $\D(h_1)=\D(h_2)=s(\mu_1)$ and $\D(g_1)=\D(g_2)=s(\lambda_2)$, and is zero otherwise.
\end{lemma}

\begin{proof}
 We follow the technique of \cite[Theorem~5.4]{rs}. We start by evaluating the product in \eqref{eq:comp-alig} on an element $m=x\otimes a\in M_{p\vee q}$, where $x\in C_c(\Lambda^p)$. Then $x=\sum_{\nu\in \Lambda^{p\vee q}} x(\nu)\delta_\nu$, and we have 
\begin{align}\label{eq:comp-alig-left1}
\big(\Theta_{\lambda_1,g_1,\lambda_2,g_2}\big)_p^{p\vee q}\notag&\big(\Theta_{\mu_1,h_1,\mu_2,h_2}\big)_q^{p\vee q}(x\otimes a)\\
&=\sum_{\nu\in \Lambda^{p\vee q}}x(\nu)\big(\Theta_{\lambda_1,g_1,\lambda_2,g_2}\big)_p^{p\vee q}\big(\Theta_{\mu_1,h_1,\mu_2,h_2}\big)_q^{p\vee q}(\delta_\nu\otimes a).
\end{align}
 Fix $\nu\in \Lambda^{p\vee q}$, and factor $\nu$ as $\nu_1\nu_2$ with $\nu_1\in \Lambda^q$. Then $\delta_\nu\otimes a= (\delta_{\nu_1}\otimes i_{r(\nu_2)})(\delta_{\nu_2}\otimes a)$, and we have
\begin{align}\label{eq:comp-alig-left-inner}
\notag\big(\Theta_{\mu_1,h_1, \mu_2,h_2}\big)_q^{p\vee q}(\delta_{\nu}\otimes a) &=
\Theta_{\mu_1,h_1,\mu_2,h_2}(\delta_{\nu_1}\otimes i_{r(\nu_2)})(\delta_{\nu_2}\otimes a)\\
\notag&=\Big((\delta_{\mu_1}\otimes i_{h_1})\cdot \big\langle\delta_{\mu_2}\otimes i_{h_2}, \delta_{\nu_1}\otimes i_{r(\nu_2)}\big\rangle\Big)(\delta_{\nu_2}\otimes a)\\
\notag&=\begin{cases}
	\big(\delta_{\mu_1}\otimes i_{h_1h_2^{-1}\id _{r(\nu_2)}}\big)(\delta_{\nu_2}\otimes a)&\text{if $\nu_1=\mu_2$, $\D(h_1)=\D(h_2)$}\\
	0\notag&\text{otherwise}\\
	\end{cases}\\
	&=\begin{cases}
	\delta_{\mu_1(h_1h_2^{-1}\cdot \nu_2)}\otimes i_{(h_1h_2^{-1})|_{\nu_2}}a&\text{if $\nu_1=\mu_2, $ $\D(h_1)=\D(h_2)$}\\
	0&\text{otherwise.}\\
	\end{cases}
\end{align}
Now suppose that $\nu_1=\mu_2, $ $\D(h_1)=\D(h_2)$ and let $\beta_\nu:=\mu_1(h_1h_2^{-1}\cdot \nu_2)$ . Applying the computation of \eqref{eq:comp-alig-left-inner} shows that
 \begin{align*}
&\big(\Theta_{\lambda_1,g_1,\lambda_2,g_2}\big)_p^{p\vee q}(\delta_{\beta_\nu}\otimes i_{(h_1h_2^{-1})|_{\nu_2}}a)\\
	\notag&=\begin{cases}
	\delta_{\lambda_1((g_1g_2^{-1})\cdot \beta_\nu(p,p\vee q))}\otimes i_{(g_1 g_2^{-1})|_{ \beta_\nu(p,p\vee q)}(h_1h_2^{-1})|_{\nu_2}}a&\text{if $\beta_\nu(0,p)=\lambda_2$, $ \D(g_1)=\D(g_2)$}\\
	0&\text{otherwise.}\\
	\end{cases}\\
	\end{align*}
Since $\nu_2=(h_2 h_1^{-1})\cdot\beta_\nu(q,p\vee q)$, identity (\ref{S7}) from Lemma~\ref{lem:properties} implies that $(h_1h_2^{-1})|_{\nu_2}=(h_2h_1^{-1})|_{\beta_\nu(q,p\vee q)}^{-1}$, and hence we have
 \begin{align}\label{equ-fin}
&\big(\Theta_{\lambda_1,g_1,\lambda_2,g_2}\big)_p^{p\vee q}(\delta_{\beta_\nu}\otimes i_{(h_1h_2^{-1})|_{\nu_2}}a)\\
	\notag&=\begin{cases}
	\delta_{\lambda_1((g_1g_2^{-1})\cdot \beta_\nu(p,p\vee q))}\otimes i_{(g_1 g_2^{-1})|_{ \beta_\nu(p,p\vee q)}(h_2h_1^{-1})|_{\beta_\nu(q,p\vee q)}^{-1}}a&\!\!\text{if $\beta_\nu(0,p)=\lambda_2$, $ \D(g_1)\!=\!\D(g_2)$}\\
	0&\!\!\text{otherwise.}\\
	\end{cases}
\end{align}
The above computations show that the product in \eqref{eq:comp-alig} is zero unless $\D(h_1)=\D(h_2)=s(\mu_1)$ and $\D(g_1)=\D(g_2)=s(\lambda_2)$, which is the second assertion of the lemma. 

Assume $\D(h_1)=\D(h_2)=s(\mu_1)$ and $\D(g_1)=\D(g_2)=s(\lambda_2)$. To see that the expressions in \eqref{eq:comp-alig} and \eqref{eq:comp-alig2} agree, we first we first observe that the map $\nu\mapsto \big(\beta_\nu(p,p\vee q),\beta_\nu(q,p\vee q)\big)$ is an isomorphism from $\mu_2\Lambda^{(p\vee q)-q}$ onto $\Lambda^{\min}(\lambda_2,\mu_1)$, with inverse $(\eta,\zeta)\mapsto \mu_2(h_2 h_1^{-1})\cdot\zeta$. Now, it follows from \eqref{eq:comp-alig-left1} and \eqref{equ-fin} that
\begin{align}\label{eq:comp-alig-left}
\big(\Theta_{\lambda_1,g_1,\lambda_2,g_2}\big)_p^{p\vee q}\notag&\big(\Theta_{\mu_1,h_1,\mu_2,h_2}\big)_q^{p\vee q}(x\otimes a)\\
&=\sum_{(\eta,\zeta)\in \Lambda^{\min}(\lambda_2,\mu_1)}x\big(\mu_2(h_2 h_1^{-1})\cdot\zeta\big)\delta_{\lambda_1((g_1g_2^{-1})\cdot \eta)}\otimes i_{(g_1 g_2^{-1})|_{ \eta}(h_2h_1^{-1})|_\zeta^{-1}}a.
\end{align}
For \eqref{eq:comp-alig2}, using $x=\sum_{\nu\in \Lambda^{p\vee q}} x(\nu)\delta_\nu$, we have 
\begin{align}\label{kholaseh}
\notag&\sum_{(\eta,\zeta)\in\Lambda^{\min}(\lambda_2, \mu_1)}\Theta_{\lambda_1(g_1g_2^{-1})\cdot\eta\,,\, (g_1g_2^{-1})|_\eta\,,\,\mu_2(h_2h_1^{-1})\cdot\zeta\,,\,(h_2h_1^{-1})|_\zeta^{-1}}(x\otimes a)\\
\notag&=\sum_{(\eta,\zeta)\in\Lambda^{\min}(\lambda_2, \mu_1)}\sum_{\nu\in \Lambda^{p\vee q}}x(\nu)\Theta_{\lambda_1(g_1g_2^{-1})\cdot\eta\,,\, (g_1g_2^{-1})|_\eta\,,\,\mu_2(h_2h_1^{-1})\cdot\zeta\,,\,(h_2h_1^{-1})|_\zeta^{-1}}
(\delta_\nu\otimes a)\\
&=\sum_{(\eta,\zeta)\in\Lambda^{\min}(\lambda_2, \mu_1)}\sum_{\nu\in \Lambda^{p\vee q}}x(\nu)\Big(\big(\delta_{\lambda_1(g_1g_2^{-1})\cdot\eta}\otimes i_{ (g_1g_2^{-1})|_\eta}\big)\cdot \big \langle\delta_{\mu_2(h_2h_1^{-1})\cdot\zeta}\otimes i_{(h_2h_1^{-1})|_\zeta^{-1}}, \delta_\nu\otimes a\big\rangle \Big)
\end{align} 
Evaluating the inner product and applying the right action shows that this expression is the right-hand side of \eqref{eq:comp-alig-left}, and we see that the expressions in \eqref{eq:comp-alig} and \eqref{eq:comp-alig2} agree.
\end{proof}

\begin{proof}[Proof of Proposition~\ref{prop:product system}]
The arguments in the proof of Lemma~\ref{lemma:product system iso} show that there is a well-defined multiplication on $M:=\bigsqcup_{p\in \N^k}M_p$ characterised by
\begin{align}	\label{multiplicationformula2}
	(\delta_\lambda\otimes i_g)(\delta_\mu\otimes i_h)=\begin{cases}
\delta_{\lambda (g\cdot \mu)}\otimes i_{(g|_\mu)h}&\text{if $r(\mu)=\D(g)$}\\
	0&\text{otherwise,}\\
	\end{cases}
	\end{align}	
for nonzero $p,q\in \N^k$, $\delta_\lambda\otimes i_g\in M_p$, $\delta_\mu\otimes i_h\in M_q$. Straightforward computations using the properties in Lemma~\ref{lem:properties} show that this multiplication is associative.  Condition (\ref{P1}) holds by construction. Conditions (\ref{P2}) and (\ref{P3}) follow from Lemma~\ref{lemma:product system iso} and Remark~\ref{remark:identification}, respectively. Hence $M$ a product system. 

To see that each each bimodule $M_p$ is essential, it suffices to show that the identity element $\sum_{v\in \Lambda^0}i_v$ in $G$ acts trivially on each elementary tensor $\delta_\lambda\otimes i_g$. This follows from 
\[
\big(\sum_{v\in \Lambda^0}i_v\big)\cdot (\delta_\lambda \otimes i_g)=\delta_{r(\lambda)\cdot \lambda}\otimes i_{r(\lambda)|_\lambda} i_g=\delta_\lambda \otimes i_g.
\]
A straightforward computation using the reconstruction formula~\eqref{recons-for} shows that for each $a\in C^*(G)$ and $p\in \N^k$, we have 
\begin{equation}\label{cuntz-proof}
a=\sum_{\lambda\in \Lambda^p}\Theta_{(\delta_\lambda\otimes 1), a^*(\delta_\lambda\otimes 1)}.
\end{equation} 
Hence $M_0$ acts by compact operators on each fiber $M_p$. Finally, Lemma~\ref{lemma:compactly-alig} implies that $M$ is compactly aligned.
\end{proof}

\section{$C^*$-algebras of self-similar groupoid actions on $k$-graphs}\label{universal-algebras}

Let $(G,\Lambda)$ be a self-similar groupoid action on a $k$-graph, and $M(G,\Lambda)$ be the product system of Proposition~\ref{prop:product system}. Fowler's theory of product systems \cite{F} gives us two $C^*$- algebras: the Nica--Toeplitz algebra $\TT(G,\Lambda):=\NT\big(M(G,\Lambda)\big)$, and the Cuntz--Pimsner algebra $\OO(G,\Lambda):=\OO\big(M(G,\Lambda)\big)$. In this section we show that these universal $C^*$-algebras are generated by families of unitaries and partial isometries satisfying relations that satisfy the self-similarity equation \eqref{selfsimilar groupoid defn}.

\begin{prop}\label{prop:universal}
Let $\Lambda$ be a finite $k$-graph, and $(G,\Lambda)$ be a self-similar groupoid action. 
Define $u:G\rightarrow \TT(G,\Lambda)$ and $t:\Lambda\rightarrow \TT(G,\Lambda)$ by 
\[
u_g:=\psi_0(i_g), \quad t_v:=\psi_0(i_v), \quad \text{and} \quad t_\lambda:=\psi_{p}(\delta_\lambda\otimes i_{s(\lambda)}),
\] 
for all $g\in G$, $v\in \Lambda^0$, and $\lambda\in \Lambda^p$ with $p\not=0$. Then the $C^*$-algebra $\TT(G,\Lambda)$ is generated by $u(G)\cup t(\Lambda)$, and we have
	\begin{enumerate}[label={(\arabic*)},ref={\arabic*}]
\item \label{us1}$u$ is a unitary representation of $G$ with $u_v=t_v$ for $v\in \Lambda^0$;
\item \label{us2}the set $\{t_\lambda:\lambda\in \Lambda\}$ is a Toeplitz--Cuntz--Krieger $\Lambda$-family in $\TT(G,\Lambda)$, with $\sum_{v\in \Lambda^0}t_v$ the identity of $\TT(G,\Lambda)$; and
\item \label{us3}for $g\in G$ and $\lambda\in \Lambda$, we have
\begin{align}\label{UT}
u_gt_\lambda=\begin{cases}
t_{g\cdot \lambda}u_{g|_\lambda}&\text{if $\D(g)=r(\lambda)$}\\
0&\text{otherwise.}
\end{cases}
\end{align}
\end{enumerate}
\end{prop}

Before starting the proof, we note that taking adjoints in \eqref{UT}, and then applying (\ref{S7}) give us the following useful formula:
\begin{align}\label{UT*}
t_\lambda^*u_g=\begin{cases}
u_{g|_{g^{-1}\cdot \lambda}}t_{g^{-1}\cdot \lambda}^*&\text{if $\c(g)=r(\lambda)$}\\
0&\text{otherwise.}
\end{cases}
\end{align}

\begin{proof}[Proof of Proposition~\ref{prop:universal}]
 We know that $\TT(G,\Lambda)$ is generated by $\{u_g\}\cup\{t_\lambda\}$ because $\psi(\delta_\lambda\otimes i_g)=t_\lambda u_g$, and the span of elementary tensors $\delta_\lambda\otimes i_g$ are dense in each $M_p$. Part~\eqref{us1} follows from the fact that $i:G\to C^*(G)$ is a unitary representation and $\psi_0$ is a homomorphism. We next verify part \eqref{us3}. Note that
\begin{align*}
u_g t_\lambda&= \psi_0(i_g)\psi_{d(\lambda)}(\delta_\lambda\otimes i_{s(\lambda)})=\psi_{d(\lambda)}\big(i_g\cdot (\delta_\lambda\otimes i_{s(\lambda)})\big)
=\delta_{\D(g),r(\lambda)}\psi_{d(\lambda)}(\delta_{g\cdot \lambda}\otimes i_{g|_\lambda} i_{s(\lambda)}).
\end{align*}
Now, since $ i_{g|_\lambda} i_{s(\lambda)}= i_{g|_\lambda}=
 i_{\c(g|_\lambda)} i_{g|_\lambda}= i_{s(g\cdot\lambda)} i_{g|_\lambda}$, we have
\begin{align*}
u_g t_\lambda&=
\delta_{\D(g),r(\lambda)}\psi_{d(\lambda)}(\delta_{g\cdot \lambda}\otimes i_{s(g\cdot\lambda)} i_{g|_\lambda})= \delta_{\D(g),r(\lambda)}t_{g \cdot \lambda}u_{g|_\lambda}.
\end{align*}

For part~\eqref{us2}, first observe that the point masses $\{\delta_v:v\in \Lambda^0\}$ are mutually orthogonal projections in $C(\Lambda^0)$, and hence the $\{t_v:v\in\Lambda^0\}$ are mutually orthogonal projections in $\TT(G,\Lambda)$. Therefore $\sum_{v\in \Lambda^0}t_v$ is a projection. Since each fibre $M_p$ is essential, the homomorphism $\psi_0:C^*(G)\to \TT(G,\Lambda)$ is nondegenerate, and we have
$1=\psi_0(1)=\psi_0\big(\sum_{v\in \Lambda^0}i_v\big)=\sum_{v\in \Lambda^0} t_v$.

It remains to show that $\{t_\lambda:\lambda\in \Lambda\}$ satisfies (\ref{TCK1})--(\ref{TCK3}). To prove this, we identify each $i_g$ with $\delta_{\c(g)}\otimes i_g$ and use the multiplication formula \eqref{equ:isomor} for the actions (this is possible by Remark~\ref{remark:identification}). To see that (\ref{TCK1}) holds, take $\lambda,\mu \in \Lambda^p$. Applying (\ref{T2}) gives
\begin{align*}
t_\lambda^*t_\mu=\psi_p(\delta_\lambda\otimes i_{s(\lambda)})^*\psi_p(\delta_\mu\otimes i_{s(\mu)})=\psi_0\big(\langle \delta_\lambda\otimes i_{s(\lambda)},\delta_\mu\otimes i_{s(\mu)}\rangle\big)=\delta_{\lambda,\mu}\psi_0(i_{s(\lambda)})=\delta_{\lambda,\mu}t_{s(\lambda)}.
\end{align*}
Therefore $t_\nu^*t_\nu=t_{s(\nu)}$, and we see that each $t_\nu$ is a partial isometry. Condition (\ref{TCK1}) also follows. 

To see that (\ref{TCK2}) holds, take $\lambda\in \Lambda^p$ and $\mu\in \Lambda^q$ with $s(\lambda)=r(\mu)$. We have
\begin{align*}
t_\lambda t_\mu&=\psi_p(\delta_\lambda\otimes i_{s(\lambda)})\psi_p(\delta_\mu\otimes i_{s(\mu)})=\psi_{p+q}\big( (\delta_\lambda\otimes i_{s(\lambda)})(\delta_\mu\otimes i_{s(\mu)})\big)\\
&=\psi_{p+q}\big( \delta_{\lambda(s(\lambda)\cdot \mu)}\otimes i_{s(\lambda)|_\mu }i_{s(\mu)}\big)=\psi_{p+q}\big( \delta_{\lambda \mu}\otimes i_{s(\mu)}\big)=t_{\lambda\mu}.
\end{align*}

For (\ref{TCK3}), we use the Nica covariance of $\psi$ to get 
\begin{align*}
t_\lambda^*t_\mu&=t_\lambda^*t_\lambda t_\lambda^* t_\mu t_\mu^* t_\mu\\
&=\psi_{p}(\delta_\lambda\otimes i_{s(\lambda)})^*\psi_{p}(\delta_\lambda\otimes i_{s(\lambda)})\psi_{p}(\delta_\lambda\otimes i_{s(\lambda)})^*\psi_{q}(\delta_\mu\otimes i_{s(\mu)})
\psi_{q}(\delta_\mu\otimes i_{s(\mu)})^*\psi_{q}(\delta_\mu\otimes i_{s(\mu)})\\
&=\psi_{p}(\delta_\lambda\otimes i_{s(\lambda)})^*\psi^{(p)}\big(\Theta_{(\delta_\lambda\otimes i_{s(\lambda)}),(\delta_\lambda\otimes i_{s(\lambda)})}\big)\psi^{(q)}\big(\Theta_{(\delta_\lambda\otimes i_{s(\lambda)}),(\delta_\mu\otimes i_{s(\mu)})}\big)
\psi_{q}(\delta_\mu\otimes i_{s(\mu)})\\
&=\psi_{p}(\delta_\lambda\otimes i_{s(\lambda)})^*\psi^{(p\vee q)}\Big(\big(\Theta_{(\delta_\lambda\otimes i_{s(\lambda)}),(\delta_\lambda\otimes i_{s(\lambda)})}\big)_p^{p\vee q}\big(\Theta_{(\delta_\mu\otimes i_{s(\mu)}),(\delta_\mu\otimes i_{s(\mu)})}\big)_q^{p\vee q}\Big)
\psi_{q}(\delta_\mu\otimes i_{s(\mu)}).
 \end{align*}
 Now, Lemma~\ref{lemma:compactly-alig} implies that 
\begin{align*}
\big(\Theta_{(\delta_\lambda\otimes i_{s(\lambda)}),(\delta_\lambda\otimes i_{s(\lambda)})}\big)_p^{p\vee q}\big(\Theta_{(\delta_\mu\otimes i_{s(\mu)}),(\delta_\mu\otimes i_{s(\mu)})}\big)_q^{p\vee q}
&=\sum_{(\eta,\zeta)\in \Lambda^{\min}(\lambda,\mu)}\Theta_{(\delta_{\lambda\eta}\otimes i_{s(\eta)}),(\delta_{\mu\zeta}\otimes i_{s(\zeta)})}.
\end{align*}
 Therefore 
\begin{align*}
t_\lambda^*t_\mu=&\sum_{(\eta,\zeta)\in \Lambda^{\min}(\lambda,\mu)}\psi_{p}(\delta_\lambda\otimes i_{s(\lambda)})^*\psi_{p\vee q}(\delta_{\lambda\eta}\otimes i_{s(\eta)})\psi_{p\vee q}(\delta_{\mu\zeta}\otimes i_{s(\zeta)})^*
\psi_{q}(\delta_\mu\otimes i_{s(\mu)}).
 \end{align*} 
Since $\psi_{p\vee q}(\delta_{\lambda\eta}\otimes i_{s(\eta)})=\psi_{p}(\delta_{\lambda}\otimes i_{s(\lambda)})\psi_{(p\vee q)-p}(\delta_\eta \otimes i_{s(\eta)})$ and $\psi_{p\vee q}(\delta_{\mu\zeta}\otimes i_{s(\zeta)})^*=\psi_{(p\vee q)-q}(\delta_\zeta\otimes i_{s(\zeta)})^*\psi_q(\delta_\mu\otimes i_{s(\mu)})^*$, we have
\begin{align*}
t_\lambda^*t_\mu
&=\sum_{(\eta,\zeta)\in \Lambda^{\min}(\lambda,\mu)}\psi_{0}( i_{s(\lambda)})\psi_{(p\vee q)-p}(\delta_{\eta}\otimes i_{s(\eta)})\psi_{(p\vee q)-q}(\delta_{\zeta}\otimes i_{s(\zeta)})^*
\psi_{0}(i_{s(\mu)})\\
&=\sum_{(\eta,\zeta)\in \Lambda^{\min}(\lambda,\mu)}u_{s(\lambda)}t_\eta t_{\zeta}^*u_{s(\mu)}^*\\
&=\sum_{(\eta,\zeta)\in \Lambda^{\min}(\lambda,\mu)}t_\eta t_{\zeta}^*.
 \end{align*}
Hence (\ref{TCK3}) holds.
\end{proof}

\begin{prop}\label{prop:uniforT}
Let $\Lambda$ be a finite $k$-graph, and $(G,\Lambda)$ be a self-similar groupoid action. Then $\big(\TT(G,\Lambda), (u,t)\big)$ is universal for families $\{U_g:g\in G\}\cup\{T_\lambda:\lambda\in\Lambda\}$ satisfying the relations \textnormal{(\ref{us1})}--\textnormal{(\ref{us3})} of Proposition~\ref{prop:universal}.
\end{prop}

\begin{proof}
Suppose $B$ is a $C^*$-algebra, and $U:G\rightarrow B$ and $T:\Lambda\rightarrow B $ satisfy the relations \eqref{us1}--\eqref{us3} of Proposition~\ref{prop:universal}. Then the universal property of $C^*(G)$ induces a homomorphism $\rho_0:C^*(G)\to B$ that satisfies $\rho_0(i_g)=U_g$.
For each nonzero $p\in \N^k,$ define $\rho_p:M_p\to B$ by
\begin{align}\label{omega0}
\rho_p(m)=\sum_{\lambda\in \Lambda^p}T_\lambda\rho_0(\langle \delta_\lambda\otimes1,m\rangle).
\end{align}
We claim that $\rho$ is a Nica covariant Toeplitz representation of $M(G,\Lambda)$. Condition (\ref{T1}) is clear, and (\ref{T2}) follows from the arguments in the proof of \cite[Proposition~4.4]{LRRW}. We can also use the proof of \cite[Proposition~4.4]{LRRW} to verify (\ref{T3}) in the case where the multiplication comes from the left or right action. To complete the proof of (\ref{T3}), let $p,q\in \N^k$ be nonzero, and take $m=\delta_\lambda\otimes i_{g}\in M_p$ and $n=\delta_\mu\otimes i_{h}\in M_q$. Then we have
\begin{align*}
\rho_p(m)\rho_q(n)&=\sum_{\xi\in \Lambda^p}T_\xi \rho_0(\langle \delta_\xi\otimes 1,\delta_\lambda\otimes i_{g}\rangle)\sum_{\eta\in \Lambda^q}T_\eta\rho_0(\langle \delta_\eta\otimes 1,\delta_\mu\otimes i_{h}\rangle)\\
&=T_\lambda U_{g}T_\mu U_{h}\\
&=\delta_{r(\mu),\D(g)}T_\lambda T_{g\cdot \mu}U_{g|_\mu}U_h\\
&=\delta_{r(\mu),\D(g)}T_{\lambda (g\cdot \mu)}U_{g|_\mu}U_h.
\end{align*}
On the other hand
\begin{align*}
\rho_{p+q}(mn)&=\delta_{r(\mu),\D(g)}\rho(\delta_{\lambda(g\cdot \mu)}\otimes i_{g|_\lambda}i_h)\\
&=\delta_{r(\mu),\D(g)} \sum_{\nu\in \Lambda^{p+q}}T_\nu \rho_0(\langle \delta_\nu\otimes 1,\delta_{\lambda(g\cdot \mu)}\otimes i_{g|_\mu}i_h\rangle)\\
&=\delta_{r(\mu),\D(g)} T_{\lambda(g\cdot \mu) }\rho_0( i_{g|_\mu}i_h)\\
&=\delta_{r(\mu),\D(g)}T_{\lambda (g\cdot \mu)}U_{g|_\mu}U_h,
\end{align*}
giving (\ref{T3}).

To see that $\rho$ is Nica covariant, take $(\delta_{\lambda_1}\otimes i_{g_1}), (\delta_{\lambda_2}\otimes i_{g_2})\in M_p$ and $ (\delta_{\mu_1}\otimes i_{h_1}),(\delta_{\mu_2}\otimes i_{h_2})\in M_q$, and let $K=\Theta_{(\delta_{\lambda_1}\otimes i_{g_1}), (\delta_{\lambda_2}\otimes i_{g_2})}$ and $S=\Theta_{(\delta_{\mu_1}\otimes i_{h_1}),(\delta_{\mu_2}\otimes i_{h_2})}$. We have
\begin{align*}
\rho^{(p)}(K)&\rho^{(q)}(S)\\
&=\rho_p(\delta_{\lambda_1}\otimes i_{g_1})\rho_p (\delta_{\lambda_2}\otimes i_{g_2})^*\rho_q(\delta_{\mu_1}\otimes i_{h_1})\rho_q(\delta_{\mu_2}\otimes i_{h_2})^*\\
&=T_{\lambda_1}U_{g_1}(T_{\lambda_2}U_{g_2})^*T_{\mu_1}U_{h_1}(T_{\mu_2}U_{h_2})^*\\
&=\delta_{\D(g_1),\D(g_2)}\delta_{\D(h_1),\D(h_2)}T_{\lambda_1}U_{g_1g_2^{-1}} T_{\lambda_2}^*T_{\mu_1}U_{h_1h_2^{-1}}T_{\mu_2}^*\\
&=\delta_{\D(g_1),\D(g_2)}\delta_{\D(h_1),\D(h_2)}T_{\lambda_1}U_{g_1g_2^{-1}}\Big(\sum_{(\eta,\zeta)\in \Lambda^{\min}(\lambda_2,\mu_1)} T_{\eta}T_{\zeta}^*\Big)U_{h_1h_2^{-1}}T_{\mu_2}^*\\
&=\delta_{\D(g_1),\D(g_2)}\delta_{\D(h_1),\D(h_2)}\sum_{(\eta,\zeta)\in \Lambda^{\min}(\lambda_2,\mu_1)} T_{\lambda_1}U_{g_1g_2^{-1}}T_{\eta}\big(T_{\mu_2}U_{h_2h_1^{-1}}T_{\zeta}\big)^*.
\end{align*}
Applying \eqref{UT} now gives 
	\[\rho^{(p)}(K)\rho^{(q)}(S)=\sum_{(\eta,\zeta)\in \Lambda^{\min}(\lambda_2,\mu_1)}T_{\lambda_1((g_1g_2^{-1})\cdot\eta)}U_{ (g_1g_2^{-1})|_{\eta}}\big(T_{\mu_2((h_2h_1^{-1})\cdot\zeta)}U_{(h_2h_1^{-1})|_{\zeta}}\big)^*,\]
whenever $\D(g_1)=\D(g_2)=s(\lambda_1)$ and $\D(h_1)=\D(h_2)=s(\mu_2)$, and zero otherwise.
On the other hand, Lemma~\ref{lemma:compactly-alig} implies that $\rho^{(p\vee q)}(K_p^{p\vee q}S_p^{p\vee q})$ is only nonzero if $\D(g_1)=\D(g_2)=s(\lambda_1)$ and $\D(h_1)=\D(h_2)=s(\mu_2)$, and is given by
\begin{align*}
\rho^{(p\vee q)}(K_p^{p\vee q}S_p^{p\vee q})
&=\rho^{(p\vee q)}\Big(\sum_{(\eta,\zeta)\in \Lambda^{\min}(\lambda_2,\mu_1)}\Theta_{\big(\delta_{\lambda_1((g_1g_2^{-1})\cdot\eta)}\otimes i_{(g_1g_2^{-1})|_{\eta}}\big),\big(\delta_{\mu_2((h_2h_1^{-1})\cdot\zeta)}\otimes i_{(h_2h_1^{-1})|_{\zeta}}\big)}\Big)\\
&=\sum_{(\eta,\zeta)\in \Lambda^{\min}(\lambda_2,\mu_1)}\rho\big(\delta_{\lambda_1((g_1g_2^{-1})\cdot\eta)}\otimes i_{(g_1g_2^{-1})|_{\eta}}\big)\rho\big(\delta_{\mu_2((h_2h_1^{-1})\cdot\zeta)}\otimes i_{(h_2h_1^{-1})|_{\zeta}}\big)^*\\
&=\sum_{(\eta,\zeta)\in \Lambda^{\min}(\lambda_2,\mu_1)}T_{\lambda_1((g_1g_2^{-1})\cdot\eta)}U_{ (g_1g_2^{-1})|_{\eta}}\big(T_{\mu_2((h_2h_1^{-1})\cdot\zeta)}U_{(h_2h_1^{-1})|_{\zeta}}\big)^*
\end{align*}
and we have proven the Nica covariance.

The universal property of $\TT(G,\Lambda)$ induces a homomorphism $\rho_*:\TT(G,\Lambda)\to B$ such that $\rho=\rho_*\circ \psi$. It remains to check that $\rho_*$ maps $(u,t)$ to $(U,T)$. For each $g\in G$, we have
\[
\rho_*(u_g)=\rho_*(\psi_0(i_g))=\rho(i_g)=U_g.
\]
Also for $\lambda\in \Lambda^p$, we have
\begin{align*}
\rho_*(t_\lambda)&=\rho_*(\psi_p(\delta_\lambda\otimes 1))=\rho(\delta_\lambda\otimes 1)=T_\lambda U_{s(\lambda)}=T_\lambda.\qedhere
\end{align*}
\end{proof}

\begin{prop}\label{Toeplitz_spanning}
Let $(G,\Lambda)$ be a self-similar groupoid action, and let $(u, t)$ be as in Proposition~\ref{prop:universal}. Then
\begin{enumerate}
\item \label{Toeplitz_spanning-1}
$
\TT(G,\Lambda)=\clsp\{t_\lambda u_g t_\mu^* : \lambda,\mu \in \Lambda, g \in G \text{ and } s(\lambda)=g\cdot s(\mu)\}
$; and
\item \label{Toeplitz_spanning-2}$\OO(G,\Lambda)$ is the quotient of $\TT(G,\Lambda)$ by the ideal 
\begin{align}\label{Cuntz-quotient}
\Big\langle t_v-\sum_{\lambda\in v\Lambda^p}t_\lambda t_\lambda^*:v\in \Lambda^0, p\in \N^k \Big\rangle.
\end{align}
\end{enumerate}
\end{prop}
\begin{proof}
For \eqref{Toeplitz_spanning-1}, note that since
$t_\lambda u_g t_\mu^*= t_\lambda t_{s(\lambda)}u_g t_{s(\mu)}t_\mu^*$, the equation $t_\lambda u_g t_\mu^*\not=0$ implies that $t_{s(\lambda)}u_g t_{s(\mu)}=t_{s(\lambda)}t_{g\cdot s(\mu)}\not=0$.
Thus the condition $s(\lambda)=g\cdot s(\mu)$ rules out zero elements. Now let $\lambda, \mu,\nu, \xi \in \Lambda$ and $g,h \in S$ such that $s(\lambda)=g \cdot s(\mu)$ and $s(\nu)=h\cdot s(\xi)$. Then 
\begin{align*}
(t_\lambda u_g t_\mu^*)(t_\nu u_ht_\xi^*)&=t_\lambda u_g \big(\sum_{(\eta,\zeta)\in \Lambda^{\min}(\mu,\nu)}t_\eta t_\zeta^*\big) u_ht_\xi^*\\
&=\delta_{\D(g),s(\mu)}\delta_{\c(h),s(\nu)}\sum_{(\eta,\zeta)\in \Lambda^{\min}(\mu,\nu)}t_{\lambda(g\cdot \eta)}u_{(g|_\eta )(h|_{h^{-1}\cdot \zeta})}t_{\xi(h^{-1}\cdot \zeta)}^*.
\end{align*}
Now, applying (\ref{S4}) gives
\begin{align*}
s(\lambda(g\cdot \eta)) &= s(g\cdot \eta)= g|_{\eta}\cdot s(\eta)= g|_{\eta}\cdot s(\zeta)= g|_{\eta} h|_{h^{-1}\cdot \zeta}\cdot s(h^{-1}\cdot\zeta)\\
&= g|_{\eta} h|_{h^{-1}\cdot \zeta}\cdot s((\xi h^{-1}\cdot\zeta)),
\end{align*}
and it follows that $\operatorname{span}\{t_\lambda u_g t_\mu^* : \lambda,\mu \in \Lambda, g \in G \text{ and } s(\lambda)=g\cdot s(\mu)\}$ is a subalgebra. Since it contains the generators of $\TT(G,\Lambda)$, part~ \eqref{Toeplitz_spanning-1} follows.

For \eqref{Toeplitz_spanning-2}, since the left action of $C^*(G)$ on each fiber is by compact operators, \cite[Proposition~5.4]{F} says that Cuntz--Pimsner covariance implies Nica covariance. Therefore $\OO(G,\Lambda)$ is the quotient of $\TT(G,\Lambda)$ by the ideal from \eqref{C-P-Q}. For $p\in \N^k$ recall that $\varphi_p:C^*(G)\to \LL(M_p)$ is the homomorphism defining the left action. For each $v\in \Lambda^0$, \eqref{cuntz-proof} implies that
\[\psi^{(p)}(\varphi_p(i_v))=\sum_{\lambda\in \Lambda^0}\psi_p(\delta_\lambda\otimes 1)\psi_p( i_v(\delta_\lambda\otimes 1))^*=\big(\sum_{\lambda\in \Lambda^0}t_\lambda t_\lambda^*\big)t_v.\]
It follows that the ideal from \eqref{C-P-Q} is the ideal from \eqref{Cuntz-quotient}, and so the result follows.
\end{proof}

\begin{remark}\label{rem:gensofO}
We also label the generators of $\OO(G,\Lambda)$ by $\{u_g:g\in G\}$ and $\{t_\lambda: \lambda\in \Lambda\}$. 
\end{remark}

\begin{example}
Recall the example discussed in Example~\ref{eg:LYexample}. One can check that $U_{(f,v)}:=u_f s_v$ defines a unitary representation of $F\times \Lambda^0$ in $\OO_{F,\Lambda}$ of \cite[Definition~3.8]{LY0}, and the family $\{U_{f,v}\}$ together with the universal Cuntz-Krieger $\Lambda$-family satisfy the relations \eqref{us1}--\eqref{us3} of Proposition~\ref{prop:universal}. The induced homomorphism $\pi:\OO\big(F\times \Lambda^0, \Lambda\big) \to \OO_{F,\Lambda}$ is an isomorphism with the inverse map given by $\pi^{-1}(u_f)= \sum_{v\in \Lambda^0}u_{(f,v)}$ and $\pi^{-1}(s_\lambda)= t_\lambda$.
\end{example}

\section{An algebraic characterisation of KMS states of $\TT(G,\Lambda)$}\label{sec:algchar}

Let $\Lambda$ be a finite $k$-graph, $(G,\Lambda)$ be a self-similar groupoid action and $r\in (0,\infty)^k$. By the universal property of $\TT(G,\Lambda)$, 
there is a strongly continuous action $\gamma: \T^k \to \Aut\TT(G,\Lambda)$ such that $\gamma_z(t_\lambda u_g t_\mu^*)=z^{r\cdot (d(\lambda)-d(\mu))}t_\lambda u_g t_\mu^*$ where $r\cdot p=\sum_{i=1}^kr_ip_i$ for all $p\in \N^k$.  The action $\gamma$ gives rise to dynamics $\sigma: \R\to \Aut\TT(G,\Lambda)$  given by $\sigma_t=\gamma_{ e^{irt}}$. So we have
\begin{equation}\label{action on T}
\sigma_t(t_\lambda u_g t_\mu^*)=e^{itr\cdot (d(\lambda)-d(\mu))}t_\lambda u_g t_\mu^*.
\end{equation}
This action fixes the elements generating the ideal in \eqref{Cuntz-quotient}, and hence it descends to an action of $\OO(G,\Lambda)$.

We recall that for a $C^*$-algebra $A$ and an action $\sigma: \R \to \Aut(A)$, an
element $a\in A$ is \textit{analytic} if the map $t\mapsto \sigma_t(a)$ extends to 
 an analytic function $z\mapsto\alpha_z(a)$ on the complex plane $\C$. For $\beta\in \R$, a state $\phi$ of $A$ is called \text{KMS$_\beta$}-state for $(A,\sigma)$ if it satisfies the \textit{KMS condition}
 \[\phi(ab)=\phi(b\sigma_{i\beta}(a)) \text { for all analytic elements $a,b$}. \]
We note that it suffices to check this condition on a set of analytic elements that span a dense subalgebra of $A$. 

In this section and Section~\ref{sec:largetemp} we study KMS states of the dynamics $\big(\TT(G,\Lambda),\sigma\big)$. We first prove an algebraic characterisation formula for KMS states. Recall that $\psi_0:C^*(G)\to \TT(G,\Lambda)$ denotes the homomorphism satisfying $\psi_0(i_g)=u_g$ for all $g\in G$. 

\begin{prop}\label{prop:algebraic-chara}
Let $\Lambda$ be a finite $k$-graph with no sources. For $1 \leq i\leq k$, let $B_i$ be the matrix with entries $B_i(v, w) = |v\Lambda^{e_i}w|$ and let 
 $\rho(B_i)$ be the spectral radius of $B_i$.
Suppose that $(G,\Lambda)$ is a self-similar groupoid action. Let $r\in (0,\infty)^k$ and let $\sigma: \R \to \Aut\TT(G,\Lambda)$ be the dynamics given by \eqref{action on T}. Suppose that $\beta\in[0,\infty)$ and $\phi$ is a state on $\TT(G,\Lambda)$. 
\begin{enumerate}
\item\label{prop:algebraic-chara-1}If $\phi\circ \psi_0$ is a trace on $C^*(G)$, and 
\begin{equation}\label{char:span}
\phi(t_\lambda u_g t_\mu^*)=
\delta_{\lambda,\mu} \delta_{\c(g),\D(g)}e^{-\beta r\cdot d(\lambda)}\phi( u_g) \text{ for } g \in G \text{ and } \lambda, \mu\in \Lambda,
\end{equation}
then $\phi$ is a KMS$_{\beta}$-state for $\big(\TT(G,\Lambda),\sigma\big)$.
\item \label{prop:algebraic-chara-2}If $\phi$ is a KMS$_{\beta}$-state for $\big(\TT(G,\Lambda),\sigma\big)$, then $\phi\circ \psi_0$ is a trace on $C^*(G)$. If in addition $r$ has rationally independent coordinates, then $\phi$ satisfies \eqref{char:span}. 
\item \label{prop:algebraic-chara-3}If $\beta r_i>\ln \rho(B_i)$ for all $1\leq i\leq k$, then $\phi$ is a KMS$_{\beta}$-state for $\big(\TT(G,\Lambda),\sigma\big)$ if and only if $\phi$ satisfies \eqref{char:span} and $\phi\circ \psi_0$ is a trace on $C^*(G)$.
\end{enumerate}
\end{prop}

We start the proof of Proposition~\ref{prop:algebraic-chara} with some basic properties of KMS states of $\TT(G,\Lambda)$. Note that the same properties hold for KMS states of $(\OO(G,\Lambda),\sigma)$.

\begin{lemma}\label{lemma-kms-g-invariance}
Let $\Lambda$ be a finite $k$-graph with no sources, and $(G,\Lambda)$ be a self-similar groupoid action. Let $\beta\in \R$, and suppose that $\phi$ is a KMS$_{\beta}$-state for $(\TT(G,\Lambda),\sigma)$.
\begin{enumerate}
\item\label{kms-action-inv-1} $\phi(u_{\D(g)})=\phi(u_{\c(g)})$ for all $g\in G$.
\item \label{kms-action-inv-2}$\phi(u_{s(g\cdot \lambda)})=\phi(u_{s(\lambda)})$ for all $g\in G$ and $\lambda\in \D(g)\Lambda$.
\item\label{kms-action-inv-3} If $\D(g)\neq \c(g)$, then $\phi(u_g)=0$.
\item \label{kms-action-inv-4}For each $\lambda,\mu\in \Lambda$ and $g\in G$ with $s(\lambda)=g\cdot s(\mu)$, we have 
\begin{align}\label{ok1}
\big|\phi(t_\lambda u_g t_\mu^*)\big|\leq\phi(t_{\mu}t_{\mu}^*).
\end{align}
\end{enumerate}
\end{lemma}

\begin{proof}
Part~\eqref{kms-action-inv-1} follows because $\phi(u_{\D(g)})=\phi(u_g^*u_g)=\phi(u_gu_g^*)=\phi(u_{\c(g)})$. For \eqref{kms-action-inv-2}, we can apply \eqref{kms-action-inv-1} and (\ref{S3}) to get
$\phi(u_{s(g\cdot \lambda)})=\phi(u_{\c(g|_\lambda)})=\phi(u_{\D(g|_\lambda)})=\phi(u_{s(\lambda)})$.
For \eqref{kms-action-inv-3}, we apply the KMS condition to get
\[\phi(u_g)=\phi(u_{\c(g)} u_g u_{\D(g)}) = \phi( u_g u_{\D(g)}u_{\c(g)})=0.\]
For \eqref{kms-action-inv-4}, note that if $r\cdot d(\lambda)\neq r\cdot d(\mu)$, we can apply the KMS condition twice to get 
$\phi(t_\lambda u_g t_\mu^*)=e^{r\cdot d(\lambda)-r\cdot d(\mu)}\phi(t_\lambda u_g t_\mu^*)$, forcing $\phi(t_\lambda u_g t_\mu^*)=0$. If $r\cdot d(\lambda)= r\cdot d(\mu)$, then applying the Cauchy--Schwarz inequality and part~\eqref{kms-action-inv-2} gives
\begin{align*}
\big|\phi(t_\lambda u_g t_\mu^*)\big|^2&\leq\phi(t_{\mu}t_\mu^*)\phi(t_{\lambda}u_{\c(g)}t_\lambda^*)=\phi(t_{\mu}t_\mu^*)\phi(t_{\lambda}t_\lambda^*)
=\phi(u_{s(\mu)})\phi(u_{s(\lambda)})\\
&=\phi(u_{s(\mu)})^2=\phi(t_{\mu}t_{\mu}^*)^2.\qedhere
\end{align*}
\end{proof}

We will need the following analogue of \cite[Lemma~5.3]{aHLRS} in the proof of Proposition~\ref{prop:algebraic-chara}\eqref{prop:algebraic-chara-3}.

\begin{lemma}\label{lemma:approximation}
Let $\beta\in(0,\infty)$ and suppose that $\phi$ is a KMS$_{\beta}$-state for $(\TT(G,\Lambda),\sigma)$. Let $p:=(d(\lambda)\vee d(\mu))-d(\mu)$ and take $\lambda,\mu\in \Lambda$ and $g\in G$ with $s(\lambda)=g\cdot s(\mu)$. Then 
\begin{align}\label{approximation}
\phi(t_\lambda u_g t_\mu^*)=\sum_{\nu\in s(\mu)\Lambda^{lp}} \phi\big(t_{\lambda (g\cdot \nu)}u_{g|_\nu}t_{\mu\nu}^*\big)\quad \text{ for all } l\in \N.
\end{align}
\end{lemma}

\begin{proof}
We prove by induction on $l$. If $l=0$, then the sum in the right-hand side of \eqref{approximation} collapses to the summand $\nu=s(\mu)$ and the result follows from $g\cdot s(\mu)=s(\lambda)$ and $g|_{s(\mu)}=g|_{\D(g)}=g$. So we suppose that \eqref{approximation} holds for 
$l\in \N$. 

Suppose $\nu\in s(\mu)\Lambda^{lp}$. Since $t_{\lambda (g\cdot \nu)}$ is a partial isometry, applying the KMS condition gives
\[\phi\big(t_{\lambda (g\cdot \nu)}u_{g|_\nu}t_{\mu\nu}^*\big)=\phi\big(t_{\lambda (g\cdot \nu)}t_{\lambda (g\cdot \nu)}^*t_{\lambda (g\cdot \nu)}u_{g|_\nu}t_{\mu\nu}^*\big)=\phi\big(t_{\lambda (g\cdot \nu)}u_{g|_\nu}t_{\mu\nu}^*t_{\lambda (g\cdot \nu)}t_{\lambda (g\cdot \nu)}^*\big).\]
Applying (\ref{TCK3}) to $t_{\mu\nu}^*t_{\lambda (g\cdot \nu)}$ then gives
\begin{align}\label{equ-addi-00}
\phi\big(t_{\lambda (g\cdot \nu)}u_{g|_\nu}t_{\mu\nu}^*\big)&=\phi\Big(t_{\lambda (g\cdot \nu)}u_{g|_\nu}\Big(\sum_{(\eta,\zeta)\in \Lambda^{\min}(\mu\nu,\lambda(g\cdot \nu))}t_{\eta}t_{\zeta}^*\Big)t_{\lambda (g\cdot \nu)}^*\Big).
\end{align}
Since $r(\eta)=s(\nu)=\D(g|_\nu)$, 
the relation \eqref{UT} implies that 
 $u_{g|_\nu}t_{\eta}=t_{(g|_\nu)\cdot \eta}u_{g|_{\nu\eta}}$. Since $r((g|_\nu)\cdot \eta)=\c(g|_\nu)=s(g\cdot \nu)$, we have $t_{\lambda (g\cdot \nu)}u_{g|_\nu}t_{\eta}=t_{\lambda (g\cdot \nu)(g|_\nu)\cdot \eta}u_{g|_{\nu\eta}}$, which is equal $t_{\lambda (g\cdot (\nu\eta))}$ by (\ref{S1}). Putting this in \eqref{equ-addi-00} gives
\begin{align}\label{komaki}
\phi\big(t_{\lambda (g\cdot \nu)}u_{g|_\nu}t_{\mu\nu}^*\big)
&=\phi\Big(\sum_{(\eta,\zeta)\in \Lambda^{\min}(\mu\nu,\lambda(g\cdot \nu))}t_{\lambda (g\cdot (\nu\eta))}u_{g|_{\nu\eta}}t_{\lambda (g\cdot \nu)\zeta}^*\Big).
\end{align}

We now use the inductive hypothesis and \eqref{komaki} to write 
\begin{align}\label{j1 induction}
\phi(t_\lambda u_g t_\mu^*)\notag&=\sum_{\nu\in s(\mu)\Lambda^{pl}} \sum_{(\eta,\zeta)\in \Lambda^{\min}(\mu\nu,\lambda(g\cdot \nu))}\phi\big (t_{\lambda (g\cdot (\nu \eta))}u_{g|_{\nu\eta}}t_{\lambda (g\cdot \nu)\zeta}^*\big)\\
&=\sum_{\nu\in s(\mu)\Lambda^{pl}} \sum_{(\eta,\zeta)\in \Lambda^{\min}(\mu\nu,\lambda(g\cdot \nu))}\phi\big (t_{\lambda (g\cdot (\nu \eta))}u_{g|_{\nu\eta}}t_{\mu\nu\eta}^*\big).
\end{align}
We claim that \eqref{j1 induction} is equal to the expression 
\begin{equation}\label{j1 induction 2}
\sum_{\tau\in s(\mu)\Lambda^{p(l+1)}} \phi\big (t_{\lambda (g\cdot \tau)}u_{g|_{\tau}}t_{\mu\tau}^*\big).
\end{equation}
To see this, first fix $(\eta,\zeta)\in \Lambda^{\min}(\mu\nu,\lambda(g\cdot \nu))$. Since
\[d(\mu\nu)+d(\eta)=d(\mu\nu)\vee d(\lambda (g\cdot \nu))=(d(\mu)+d(\nu))\vee(d(\lambda) +d(\nu))=(d(\mu)\vee d(\lambda)) +d(\nu),\]
we have $d(\eta)=(d(\mu)\vee d(\lambda)) -d(\mu)=p$. Hence $d(\nu\eta)=(l+1)p$, and it follows that every summand in \eqref{j1 induction} appears in \eqref{j1 induction 2}. 

Now suppose $\tau\in s(\lambda)\Lambda^{(l+1)p}$, and $\phi(t_{\lambda (g\cdot \tau)}u_{g|_\nu}t_{\mu\tau}^*)\neq 0$. Then by the KMS condition we have
\[
\phi(t_{\lambda (g\cdot \tau)}u_{g|_\nu}t_{\mu\tau}^*)\neq 0 \implies \phi(t_{\mu\tau}^*t_{\lambda (g\cdot \tau)}u_{g|_\nu})\neq 0 \implies t_{\mu\tau}^*t_{\lambda (g\cdot \tau)} \implies \Lambda^{\min}(\mu\tau, \lambda(g\cdot \tau))\not=\emptyset.
\]
Let $(\alpha, \xi)\in\Lambda^{\min}(\mu\tau, \lambda(g\cdot \tau))$. Define
\[
\nu := \tau(0,lp),\quad\eta := \tau(lp,(l+1)p),\quad\text{and}\quad\zeta := (g\cdot \tau)(lp,lp+(d(\mu)\vee d(\lambda))-d(\lambda)).
\]
Then
\[
\mu\nu\eta=(\mu\tau\alpha)(0,(d(\mu)\vee d(\lambda))+lp)=(\lambda(g\cdot\tau)\xi)(0,(d(\mu)\vee d(\lambda))+lp)=\lambda(g\cdot \nu)\zeta.
\]
Hence $(\eta,\zeta)\in \Lambda^{\min}(\mu\nu,\lambda(g\cdot \nu))$, and $\phi\big (t_{\lambda (g\cdot \tau)}u_{g|_{\tau}}t_{\mu\tau}^*\big)=\phi\big (t_{\lambda (g\cdot (\nu \eta))}u_{g|_{\nu\eta}}t_{\mu\nu\eta}^*\big)$. It follows that the expressions \eqref{j1 induction} and \eqref{j1 induction 2} agree.
\end{proof}

\begin{proof} [Proof of Proposition~\ref{prop:algebraic-chara}]
For part~\eqref{prop:algebraic-chara-1}, suppose that $\phi\circ \psi_0$ is a trace and that $\phi$ satisfies \eqref{char:span}. To see that $\phi$ is a KMS$_\beta$-state, it suffices to take two spanning elements $b=t_\lambda u_g t_\mu^*$ and $c=t_\nu u_h t_\xi^*$ with $s(\lambda)=g\cdot s(\mu)$ and $s(\nu)=h\cdot s(\xi)$ and check the KMS condition
\begin{equation}\label{KMScondbc}
\phi(bc)=\phi(c\sigma_{i\beta}(b))=e^{-\beta r\cdot (d(\lambda)-d(\mu))}\phi(cb).
\end{equation}
We compute using (\ref{TCK3}) in the first equality, and \eqref{UT}, \eqref{UT*} in the second, to get
\[
bc=t_\lambda u_g \Big(\sum_{(\eta,\zeta)\in \Lambda^{\min}(\mu,\nu)}t_\eta t_\zeta^*\Big) u_h t_\xi^*=\sum_{(\eta,\zeta)\in \Lambda^{\min}(\mu,\nu)}t_{\lambda(g\cdot \eta)}u_{g|_\eta }u_{h|_{h^{-1}\cdot \zeta}}t_{\xi(h^{-1}\cdot \zeta)}^*.
\]
Applying (\ref{S3}) now gives 
\[
bc=\sum_{(\eta,\zeta)\in \Lambda^{\min}(\mu,\nu)}t_{\lambda(g\cdot \eta)}u_{g|_\eta h|_{h^{-1}\cdot \zeta}}t_{\xi(h^{-1}\cdot \zeta)}^*.
\]
We now apply $\phi$ to both sides, and it follows from \eqref{char:span} that
 \begin{align}\label{char:left}
\phi(bc)
&=\sum_{\substack{(\eta,\zeta)\in \Lambda^{\min}(\mu,\nu) \\ \lambda(g\cdot \eta)=\xi(h^{-1}\cdot \zeta)}}e^{-\beta r\cdot d(\lambda(g\cdot \eta))}\phi\big(u_{g|_\eta h|_{h^{-1}\cdot \zeta}}\big).
\end{align}
A similar argument gives 
\begin{align}\label{char:right}
\phi(cb)&=\sum_{\substack{(\alpha,\tau)\in \Lambda^{\min}(\xi,\lambda) \\ \nu(h\cdot \alpha)=\mu(g^{-1}\cdot \tau)}}e^{-\beta r\cdot d(\nu(h\cdot \alpha))}\phi\big(u_{h|_\alpha g|_{g^{-1}\cdot \tau}}\big).
\end{align}
Next suppose that $(\eta,\zeta)\in \Lambda^{\min}(\mu,\nu)$ with $ \lambda(g\cdot \eta)=\xi(h^{-1}\cdot \zeta)$. Then since $d(g\cdot \eta)\wedge d(h^{-1}\cdot \zeta)= d(\eta)\wedge d(\zeta)=0$ and $d(\lambda)+d(g\cdot \eta)=d(\xi)+d(h^{-1}\cdot \zeta)$, \cite[Lemma~3.2]{aHLRS} implies that $d(\lambda)+d(g\cdot \eta)=d(\lambda)\vee d(g\cdot \eta)$ and $(h^{-1}\cdot \zeta, g\cdot \eta)\in\Lambda^{\min}(\xi,\lambda)$. Since $\nu(h\cdot (h^{-1}\cdot \zeta))=\mu(g^{-1}\cdot (g\cdot \eta))$, we see that $(h^{-1}\cdot \zeta, g\cdot \eta)$ belongs to the index of the sum in \eqref{char:right}. Similarly, we can show that for each $(\alpha,\tau)$ in the index set of the sum in \eqref{char:right}, $( g^{-1}\cdot \tau,h\cdot \alpha)$ is in the the index of the sum in \eqref{char:left}. Thus the map $(\eta,\zeta)\mapsto (h^{-1}\cdot \zeta, g\cdot \eta)$ is a bijection between the two index sets. 

Now take $(\eta,\zeta)$ in the index set of \eqref{char:left} and let $(\alpha,\tau)= (h^{-1}\cdot \zeta, g\cdot \eta)$ be the corresponding element in the index of \eqref{char:right}. Then $h|_{h^{-1}\cdot \zeta}=h|_\alpha$ and $g|_{g^{-1}\cdot \tau}=g|_\eta$ and the tracial property of $\phi$ gives $\phi\big(u_{h|_\alpha g|_{g^{-1}\cdot \tau}}\big)=\phi\big(u_{g|_\eta h|_{h^{-1}\cdot \zeta}}\big)$. Moreover since $d(g^{-1}\cdot \tau)=d(\tau)=d(\eta)$, we have
\[e^{-\beta r\cdot (d(\lambda)-d(\mu))}e^{-\beta r\cdot d(\nu(g\cdot \alpha))}=e^{-\beta r\cdot (d(\lambda)-d(\mu))}e^{-\beta r\cdot d(\mu(g^{-1}\cdot \tau))}=e^{-\beta r\cdot (d(\lambda)+d(\eta))}=e^{-\beta r\cdot d(\lambda(g\cdot \eta))}.\]
Thus the summands in both sides of \eqref{KMScondbc} are equal, as required. 

For \eqref{prop:algebraic-chara-2}, since each $u_g$ is fixed by the action, the KMS condition implies that $\phi\circ \psi_0$ is a trace on $C^*(G)$. To see that \eqref{KMScondbc} holds, first assume that $\lambda\not=\mu$. If $d(\lambda)=d(\mu)$, then applying the KMS condition gives $\phi(t_\lambda u_g t_\mu^*)=e^{-\beta r\cdot d(\mu)}\phi(t_\mu^*t_\lambda u_g)=0$. If $d(\lambda)\not= d(\mu)$, then two applications of the KMS condition gives $\phi(t_\lambda u_g t_\mu^*)=e^{-\beta r\cdot (d(\mu)-d(\lambda))}\phi(t_\lambda u_g t_\mu^*)$. Since $r$ has rationally independent coordinates we have $e^{-\beta r(\cdot d(\mu)-d(\lambda))}\not= 0$, and hence $\phi(t_\lambda u_g t_\mu^*)=0$. Now assume $\c(g)\not=\D(g)$. Since $t_\lambda u_g t_\mu^*$ is nonzero only if $s(\mu)=\D(g)$ and $\c(g)=s(\lambda)$, we see that $t_\lambda u_g t_\mu^*\not=0 \implies \lambda\not=\mu$, and then by the argument above we know that $\phi(t_\lambda u_g t_\mu^*)=0$. We now assume that $\lambda=\mu$ and $\c(g)=\D(g)$. Then the KMS condition gives $\phi(t_\lambda u_g t_\mu^*)=e^{-\beta r\cdot d(\mu)}\phi(t_\mu^*t_\mu u_g)=e^{-\beta r\cdot d(\mu)}\phi(u_g)$. Hence \eqref{KMScondbc} holds.

We now need to prove \eqref{prop:algebraic-chara-3}.  By parts~\eqref{prop:algebraic-chara-1} and \eqref{prop:algebraic-chara-2}, we need only prove that if $\phi$ is a KMS$_\beta$-state, then $\phi$ satisfies \eqref{char:span}. We can use the arguments in the preceding paragraph to prove all cases except the implication $\lambda\not=\mu \implies \phi(t_\lambda u_g t_\mu^*)=0$. When $d(\lambda)= d(\mu)$, we can use the KMS condition as above to get $\phi(t_\lambda u_g t_\mu^*)=0$. So we suppose that $d(\lambda)\neq d(\mu)$. Then one of $p:=(d(\lambda)\vee d(\mu))-d(\mu)$ and $q:=(d(\lambda)\vee d(\mu))-d(\lambda)$ will be nonzero. We assume $p\neq 0$.  (Taking adjoints gives $q\neq 0$.) A computation using \eqref{approximation} and the Cauchy-Schwarz inequality gives
\begin{align*}
\big|\phi(t_\lambda u_g t_\mu^*)\big|&\leq \sum_{\nu\in s(\lambda)\Lambda^{pl}}\big| \phi\big(t_{\lambda\nu}u_{g|_{g^{-1}\cdot \nu}}t_{\mu(g^{-1}\cdot \nu)}^*\big)\big|\\
&\leq \sum_{\nu\in s(\lambda)\Lambda^{pl}}\sqrt{\phi\big(t_{\lambda\nu}u_{\c(g|_{g^{-1}\cdot \nu})}t_{\lambda\nu}^*\big)\phi\big(t_{\mu(g^{-1}\cdot \nu)}t_{\mu(g^{-1}\cdot \nu)}^*\big)}\\
&= \sum_{\nu\in s(\lambda)\Lambda^{pl}}\sqrt{\phi(t_{\lambda\nu}u_{s(\nu)}t_{\lambda\nu}^*)\phi\big(t_{\mu(g^{-1}\cdot \nu)}t_{\mu(g^{-1}\cdot \nu)}^*\big)}.
\end{align*}
Since $d(g^{-1}\cdot \nu)=d(\nu)$, applying the KMS condition gives us
\[
\big|\phi(t_\lambda u_g t_\mu^*)\big|\le \sum_{\nu\in s(\lambda)\Lambda^{pl}}\sqrt{e^{-\beta r\cdot(d(\lambda \nu)+d(\mu \nu))}\phi(u_{s(\nu)})\phi(u_{s(g^{-1}\cdot \nu)})}.
\]
We can now use Lemma~\ref{lemma-kms-g-invariance}\eqref{kms-action-inv-2} to see that each $\phi(u_{s(\nu)})=\phi(u_{s(g^{-1}\cdot \nu)})$, and hence
\[
\big|\phi(t_\lambda u_g t_\mu^*)\big| =e^{-\frac{\beta}{2} r\cdot(d(\lambda)+d(\mu))}\sum_{\nu\in s(\lambda)\Lambda^{pl}}e^{-\beta r\cdot d(\nu)}\phi(u_{s(\nu)}).
\]
The argument of the last paragraph in \cite[Theorem~5.1]{aHLRS} now shows that the right-hand side vanishes as $l\to \infty$.
\end{proof}

\section{KMS states of $\TT(G,\Lambda)$ for large inverse temperatures}\label{sec:largetemp}

We now state our main result for the KMS structure of the Toeplitz algebra of a self-similar groupoid action on a $k$-graph.

\begin{thm}\label{Thm:KMS_beta_tau}
Let $\Lambda$ be a finite $k$-graph with no sources, and $(G,\Lambda)$ a self-similar groupoid action. For $1 \leq i\leq k$, let $B_i$ be the matrix with entries $B_i(v, w) = |v\Lambda^{e_i}w|$ and let 
 $\rho(B_i)$ be the spectral radius of $B_i$. Take $r\in (0,\infty)^k$ and let $\sigma: \R \to \Aut\TT(G,\Lambda)$ be the dynamics of \eqref{action on T}.
Suppose that $\beta r_i> \ln\rho(B_i)$ for all $1\leq i\leq k$. 
 \begin{enumerate}[label={(\arabic*)},ref={\arabic*}]
 \item\label{Zconverges} If $\tau$ is tracial state on $C^*(G)$, then the series
\begin{equation}\label{suminthm}
\sum_{p\in\N^k} \sum_{\lambda \in \Lambda^p} e^{-\beta r\cdot p} \tau(i_{s(\lambda)})
\end{equation}
converges to a positive number $Z(\beta,\tau)$, and there is a KMS$_\beta$-state $\phi_{\tau}$ of $(\TT(G,\Lambda),\sigma)$ such that
\begin{equation}\label{KMS_tau_formula}
\phi_{\tau}(t_\lambda u_g t_\mu^*)=\delta_{\lambda,\mu}Z(\beta,\tau)^{-1} \sum_{p\geq d(\lambda)}e^{-\beta r\cdot p} \sum_{\{\nu \in s(\lambda)\Lambda^{p-d(\lambda)}: \, g \cdot \nu=\nu\}} \tau\big(i_{g|_\nu}\big) .
\end{equation}
\item\label{KMS_tau_affine} 
The map $\tau \mapsto \phi_{\tau}$ 
is an affine isomorphism of the simplex of tracial states of $C^*(G)$ onto the simplex of KMS$_\beta$-states of $(\TT(G,\Lambda),\sigma)$.
\end{enumerate}
\end{thm}

\begin{remark}\label{rem:refsofrproductsystems}
We will use direct methods to prove Theorem~\ref{Thm:KMS_beta_tau} by considering the generators-and-relations model of $\TT(G,\Lambda)$, as in Proposition~\ref{prop:uniforT}. We point out here that there are results for equilibrium states of product system $C^*$-algebras in the literature \cite{ALN,Kak} that could give alternative proofs.
\end{remark}

\begin{proof}[Proof of  Theorem~\ref{Thm:KMS_beta_tau}\eqref{Zconverges}]
Since $0\leq \tau(i_{s(\lambda)})\leq 1$ for all $\lambda$, we have
\[\sum_{p\in\N^k} \sum_{\lambda \in \Lambda^p} e^{-\beta r\cdot p} \tau(i_{s(\lambda)})=\sum_{v\in \Lambda^0}\sum_{p\in\N^k} \sum_{\lambda \in \Lambda^pv} e^{-\beta r\cdot p} \tau(i_{s(\lambda)})\le \sum_{v\in \Lambda^0}\sum_{p\in\N^k} \sum_{\lambda \in \Lambda^pv} e^{-\beta r\cdot p}.\]
We know from \cite[Theorem~6.1(a)]{aHLRS} that $\sum_{p\in\N^k} \sum_{\lambda \in \Lambda^pv} e^{-\beta r\cdot p}$ converges for all $v\in \Lambda^0$, so the series in \eqref{suminthm} converges because $\Lambda^0$ is finite.

To see the second assertion, we follow the structure of \cite[Theorem~6.1]{lrrw}. Let $\tau$ be a tracial state of $C^*(G)$, and let $(\pi_\tau, \xi_\tau, K_\tau)$ be the GNS-triple corresponding to $\tau$.
For $p\in \N^k$, let $X(\Lambda^p)$ be the graph correspondence of the directed graph $(\Lambda^0,\Lambda^p,r,s)$, and $H_p$ the Hilbert space 
$H_p:=X(\Lambda^p)\otimes_{C(\Lambda^0)} K_\tau$. Let $H:=\bigoplus_{p\in \N^k}^\infty H_p$. Observe that 
\[
H=\clsp\{\delta_\mu\otimes \kappa:\mu\in \Lambda, \kappa\in K_\tau\text{ and }\pi_\tau(i_{s(\mu)})\kappa=\kappa\}. 
\]
A similar argument to the proof of \cite[Lemma~6.2]{lrrw} shows that there is a unitary representation $U: G\to U(H)$, and 
a family $\{T_\lambda:\lambda\in \Lambda\}\subseteq B(H)$ 
such that 
\begin{equation}\label{Fock-U}
U_g(\delta_\mu \otimes \kappa)= \begin{cases}\delta_{g \cdot \mu }\otimes \pi_\tau (i_{g|_\mu})\kappa & \text {if } r(\mu) =\D(g)\\
0 &\text{ otherwise,}
\end{cases}
\end{equation}
\begin{equation}\label{Fock-T}
T_\lambda(\delta_\mu\otimes \kappa)=\delta_{s(\lambda),r(\mu)} (\delta_{\lambda\mu}\otimes \kappa),
\end{equation}
and that $(U,T)$ satisfies the relations \eqref{us1}--\eqref{us3} of Proposition~\ref{prop:universal}. To see that $(U,T)$ satisfies \eqref{us2}, note that 
the operator $T_\lambda$ is adjointable with
\begin{equation}\label{adjoint-T}
T^*_\lambda(\delta_\nu \otimes \kappa)= \begin{cases}
\delta_\xi \otimes \kappa & \text {if }\nu = \lambda \xi\\
0 &\text{ otherwise.}
\end{cases}
\end{equation}
Straightforward calculations show that $\{T_\lambda:\lambda\in\Lambda\}$ is a family of partial isomorphisms satisfying (\ref{TCK1}) and (\ref{TCK2}). 

 For (\ref{TCK3}), take $\lambda,\mu\in \Lambda$. The formulas \eqref{Fock-T} and \eqref{adjoint-T} imply that
\begin{equation}\label{right-hand Fock}
T^*_\lambda T_\mu(\delta_\nu \otimes \kappa)= \begin{cases}
\delta_\xi\otimes\kappa & \text {if $r(\nu) =s(\mu)$ and $\mu\nu=\lambda \xi$}\\
0 &\text{ otherwise.}
\end{cases}
\end{equation}
On the other hand, we have
\begin{equation}\label{left-hand Fock}
\sum_{(\eta,\zeta)\in \Lambda^{\min}(\lambda,\mu)}T_\eta T^*_\zeta(\delta_\nu \otimes \kappa)= \begin{cases}
\delta_{\alpha \nu''}\otimes \kappa & \text {if $r(\nu) =s(\mu)$, $\nu=\nu'\nu''$ and $(\alpha,\nu')\in \Lambda^{\min}(\lambda,\mu)$ }\\
0 &\text{ otherwise.}
\end{cases}
\end{equation}
Now if $\mu\nu=\lambda\xi$, then we can factor $\nu=\nu'\nu''$ and $\xi=\alpha\nu''$ such that $(\alpha,\nu')\in \Lambda^{\min}(\lambda,\mu)$. Conversely, if $r(\nu) =s(\mu)$, $\nu=\nu'\nu''$ and $(\alpha,\nu')\in \Lambda^{\min}(\lambda,\mu)$, then $\mu\nu=\mu\nu'\nu''=\lambda\alpha \nu''$ and with $\xi=\alpha \nu''$ the right-hand sides of \eqref{right-hand Fock} and \eqref{left-hand Fock} are equal. This gives (\ref{TCK3}).

We can now use the universal property of $\TT(G,\Lambda)$ to get a homomorphism $\pi_{U,T}:\TT(G,\Lambda)\to B(H)$. For $a\in \TT(G,\Lambda)$, let
\begin{equation}\label{psi_a}
\phi_{\tau}(a):={Z(\beta,\tau)\inv} \sum_{p\in \N^k} e^{-\beta r\cdot p} \sum_{\lambda \in \Lambda^p} \big( \pi_{U,T}(a)(\delta_\lambda\otimes \xi_\tau) \mid \delta_\lambda \otimes \xi_\tau\big).
\end{equation}
Since 
\begin{align*}
\phi_{\tau}(1)&=Z(\beta,\tau)\inv \sum_{p\in\N^k}e^{-\beta r\cdot p} \sum_{\lambda \in \Lambda^p} \big(\delta_\lambda\otimes \xi_\tau \mid \delta_\lambda \otimes \xi_\tau\big) \\
&=Z(\beta,\tau)\inv \sum_{p\in\N^k}e^{-\beta r\cdot p} \sum_{\lambda \in \Lambda^p}\tau(i_{s(\lambda)})\\
&=1, 
\end{align*} 
the function $a\mapsto \phi_{\tau}(a)$ is a state of $\TT(G,\Lambda)$. Since the action $G$ on $\Lambda$ is degree preserving, a similar argument to \cite[Theorem~6.1(2)]{LRRW} shows that the sum in \eqref{psi_a} reduces to \eqref{KMS_tau_formula}; $ \phi_{\tau}\circ \psi_0$ is a trace; and $ \phi_{\tau}$ satisfies \eqref{char:span}. Thus
 $ \phi_{\tau}$ is a KMS$_\beta$-state for $(\TT(G,\Lambda),\sigma)$.
\end{proof}

To prove Theorem~\ref{Thm:KMS_beta_tau}\eqref{KMS_tau_affine}  we need two lemmas.

\begin{lemma}\label{projection}
Let $\Lambda$ be a finite $k$-graph with no sources, and $(G,\Lambda)$ a self-similar groupoid action. Let
 \begin{align}\label{def-p}
P:= \sum_{v\in \Lambda^0} \prod_{i=1}^k\Big(u_v - \sum_{\lambda\in v\Lambda^{e_i} } t_\lambda t_\lambda^*\Big),
\end{align}
and for each $q\in \N^k$, define 
\begin{align}\label{def-Mq}
M_q := \sum_{0\leq p\leq q} \sum_{\lambda\in \Lambda^p} t_\lambda P t_\lambda^*.
\end{align}
Then $P$ and $M_q$ are projections in $\TT(G,\Lambda)$, and we have 
\begin{align}\label{pug}
Pu_g=u_g P \quad \text{ for all } g\in G. 
\end{align}
\end{lemma}
\begin{proof}
Since $u_v - \sum_{\lambda\in v\Lambda^{e_i} } t_\lambda t_\lambda^*$ is a projection and the product in \eqref{def-p} is commuting by \cite[Lemma~4.2]{aHLRS}, we see that $\prod_{i=1}^k\Big(u_v - \sum_{\lambda\in v\Lambda^{e_i} } t_\lambda t_\lambda^*\Big)$ is a projection. Also (\ref{TCK1}) implies that for $v\neq w$, $u_v - \sum_{\lambda\in v\Lambda^{e_i} } t_\lambda t_\lambda^*$ and $u_w - \sum_{\mu\in w\Lambda^{e_i} } t_\mu t_\mu^*$ are mutually orthogonal. It follows that $P$ is a projection. 

We next prove \eqref{pug}. Let $g\in G$. A computation using \eqref{UT} and \eqref{UT*} shows that if $\D(g) = r(\lambda)$, then $u_g t_\lambda t_\lambda^* = t_{g \cdot \lambda} t_{g \cdot \lambda}^* u_g$ (see also \cite[page 297]{LRRW}). For each $v\in \Lambda^0$ and nonempty subset $J$ of $\{1,\dots ,k\}$, let $e_J:=\sum_{j\in J}e_j$ and $b_{v,J}:=\sum_{\mu\in v\Lambda^{e_J}}t_\mu t_\mu^*$. Then \cite[Lemma~4.2]{aHLRS} implies that 
\begin{align}\label{tf-help}
\prod_{i=1}^k\Big(u_v - \sum_{\lambda\in v\Lambda^{e_i}} t_\lambda t_\lambda^*\Big)=u_v+\sum_{\emptyset\neq J\subseteq \{1,\dots ,k\}}(-1)^{|J|}b_{v,J},
\end{align}
 and therefore 
\begin{equation}\label{p-formula-help}
P=\sum_{v\in \Lambda^0}\Big(u_v+\sum_{\emptyset\neq J\subseteq \{1,\dots ,k\}}(-1)^{|J|}b_{v,J}\Big)=1+\sum_{v\in \Lambda^0}\sum_{\emptyset\neq J\subseteq \{1,\dots ,k\}}(-1)^{|J|}b_{v,J}.
\end{equation}
Now we have 
\begin{align*}
u_g P&=u_g +\sum_{\emptyset\neq J\subseteq \{1,\dots ,k\}}(-1)^{|J|}u_g b_{\D(g),J}\\
&=u_g +\sum_{\emptyset\neq J\subseteq \{1,\dots ,k\}}(-1)^{|J|}\sum_{\mu\in \D(g)\Lambda^{e_J}}t_{g\cdot\mu }t_{g\cdot\mu}^*u_g.
\end{align*}
Since for each $J\subseteq \{1,\dots ,k\}$, $\mu\mapsto g\cdot \mu$ is a bijection of $\D(g)\Lambda^{e_J}$ onto $\c(g)\Lambda^{e_J}$, we can rewrite the sum as
\begin{align*}
u_g P&=u_g +\sum_{\emptyset\neq J\subseteq \{1,\dots ,k\}}(-1)^{|J|}\sum_{\nu\in \c(g)\Lambda^{e_J}}t_{\nu}t_{\nu}^*u_g\\
& =\sum_{v\in \Lambda^0}\big(u_v+\sum_{\emptyset\neq J\subseteq \{1,\dots ,k\}}(-1)^{|J|}b_{v,J}\big)u_g\\
&=Pu_g.
\end{align*}

Finally, it remains to show that $M_q$ is a projection. For this, we compute using \eqref{pug}:
\[t_\lambda P t_\lambda^*t_\lambda P t_\lambda^*=t_\lambda P u_{s(\lambda)}Pt_\lambda^*=t_\lambda u_{s(\lambda)}P^2t_\lambda^*= t_\lambda P t_\lambda^*.\]
It follows that each summand in \eqref{def-Mq} is a projection. So we need to show that for $\lambda\neq \mu$, the two summnds $t_\lambda P t_\lambda^*$ and $t_\mu P t_\mu^*$ are orthogonal. By (\ref{TCK3}) it suffices to check this for $\lambda,\mu$ with $r(\lambda)=r(\mu)$ and $d(\lambda)\neq d(\mu)$ (otherwise $\Lambda^{\min}(\lambda,\mu)$ is empty). So we assume $d(\lambda)\neq d(\mu)$. This assumption implies that 
$d(\lambda)\vee d(\mu)$ is strictly bigger that at least one of $d(\lambda)$ and $d(\mu)$. We claim that if $d(\lambda)\vee d(\mu)>d(\mu)$, then $t_\lambda^*t_\mu P=0$. (Note that if $d(\lambda)\vee d(\mu)>d(\lambda)$, we can take adjoints to get $Pt_\lambda^*t_\mu=0$.)

We compute using (\ref{TCK3}):
\begin{equation}\label{Mq-summand}
t_\lambda^*t_\mu P=\Big(\sum_{(\eta,\zeta)\in\Lambda^{\min}(\lambda,\mu) }t_\eta t_\zeta^*\Big) P=\sum_{(\eta,\zeta)\in\Lambda^{\min}(\lambda,\mu) }\Big(t_\eta t_\zeta^*
\prod_{i=1}^k\Big(u_{s(\mu)} - \sum_{\nu\in s(\mu)\Lambda^{e_i} } t_\nu t_\nu^*\Big)\Big).
\end{equation}
Fix $(\eta,\zeta)\in\Lambda^{\min}(\lambda,\mu)$. Since $d(\lambda)\vee d(\mu)>d(\mu)$ and $d(\mu)+d(\zeta)=d(\lambda)\vee d(\mu)$, we have $d(\zeta)>0$. Therefore we can factor $\zeta=e\zeta'$ where $e\in \Lambda^{e_j}$ for some $1\leq j\leq k$. Since 
\[t_e^*\Big(u_{s(\mu)} - \sum_{\nu\in s(\mu)\Lambda^{e_j} } t_\nu t_\nu^*\Big)=t_e^*-t_e^*=0,\]
 we have 
\begin{align*}
t_\zeta^*\prod_{i=1}^k \Big(u_{s(\mu)} - &\sum_{\nu\in s(\mu)\Lambda^{e_i} } t_\nu t_\nu^*\Big)\\
&=t_{\zeta'}^* t_e^*\Big(u_{s(\mu)} - \sum_{\nu\in s(\mu)\Lambda^{e_j} } t_\nu t_\nu^*\Big)\prod_{i=1, i\neq j}^k \Big(u_{s(\mu)} - \sum_{\nu\in s(\mu)\Lambda^{e_i} } t_\nu t_\nu^*\Big)\\
&=0
\end{align*}
Hence $t_\lambda^*t_\mu P=0$, as claimed. 
\end{proof}

\begin{lemma}\label{Lemma:phi-P}
Suppose that $\beta r_i>\ln\rho(B_i)$ for all $1\leq i\leq k$ and let $\phi$ be a KMS$_\beta$-state of $(\TT(G,\Lambda),\sigma)$. Let $P$ be as in \eqref{def-p}. Then there is a state $\omega_{P,\phi}$ of $\TT(G,\Lambda)$ given by
\begin{align}\label{omegap}
\omega_{P,\phi}(a)=\phi(P)\inv \phi(PaP) \quad\text{ for } a\in \TT(G,\Lambda).
\end{align}
Moreover, $\omega_{P,\phi}\circ \psi_0$ is a trace on $C^*(G)$, and we have
\begin{align}\label{KMS-reconstr-formula}
\phi(a)&=\phi(P) \sum_{p\in\N^k}^\infty \sum_{\lambda \in \Lambda^p} e^{-\beta r\cdot p} \omega_{P,\phi}(t_\lambda^* a t_\lambda)\quad \text{ for } a \in \TT(G,\Lambda).
\end{align}
\end{lemma}

\begin{remark}
As is customary in the literature, we call equation \eqref{KMS-reconstr-formula} the reconstruction formula.
\end{remark}

\begin{proof}[Proof of Lemma~\ref{Lemma:phi-P}]
The function $\omega_{P,\phi}$ is clearly a state on $\TT(G,\Lambda)$ because $\phi$ is a state. Suppose that $g,h\in G$ satisfy $\D(g)=\c(h)$. Note that the action $\sigma$ fixes $u_g$ and $u_gP$. Then applying \eqref{pug} in the first and KMS condition at the second equality imply that
\[
\phi(Pu_gu_hP)=\phi(u_gP^2u_h)=\phi(Pu_h\sigma_{i\beta}(u_{g}P))=\phi(Pu_hu_gP).
\]
Therefore $i_g\mapsto \phi(P)\inv\phi(u_gP)=\omega_{P,\phi}\circ \psi_0(i_g)$ descends to a trace on $C^*(G)$.

To see that \eqref{KMS-reconstr-formula} holds, we first claim that $\phi(M_q) \to 1$ as $q \to \infty$ (in the sense that $q_i\to \infty$ for each $1\leq i\leq k$). Define $m^\phi\in[0,\infty)^{\Lambda^0}$ by $m^\phi_v :=\phi(u_v)$ and let $W$ be the vector $\big(\prod_{i=1}^k \big(1 - e^{-\beta r_i}B_i\big)\big)m^\phi$. For each $v\in \Lambda$, the equality~\eqref{tf-help} implies that 
\[\phi\Big(\prod_{i=1}^k\Big(u_v - \sum_{\lambda\in v\Lambda^{e_i}} t_\lambda t_\lambda^*\Big)\Big)=\phi\Big(u_v+\sum_{\emptyset\neq J\subseteq \{1,\dots ,k\}}(-1)^{|J|}b_{v,J}\Big),\]
which is equal the $v$th entry of the vector $W$ by the computation of \cite[Lemma~4.2]{aHLRS}. Now 
we compute using the formula for $M_q$ and the KMS condition:
\begin{align*}
\phi(M_q)&=\sum_{0\leq p\leq q} \sum_{\lambda\in \Lambda^p} \phi(t_\lambda P t_\lambda^*)\\
&= \sum_{0\leq p\leq q} \sum_{\lambda\in \Lambda^p} e^{-\beta r\cdot p}\phi(t_\lambda^*t_\lambda P )\\
&=\sum_{0\leq p\leq q} \sum_{\lambda\in \Lambda^p} e^{-\beta r\cdot p}\phi\Big(\prod_{i=1}^k\Big(u_{s(\lambda)} - \sum_{\mu\in s(\lambda)\Lambda^{e_i} } t_\mu t_\mu^*\Big) \Big)\\
&=\sum_{0\leq p\leq q} \sum_{\lambda\in \Lambda^p} e^{-\beta r\cdot p}W_{s(\lambda)}.
\end{align*}
We continue as in \cite[Lemma~2.2]{aHLRS} we get
\begin{align*}
\phi(M_q)&=\sum_{0\leq p\leq q} \sum_{v\in \Lambda^0}\sum_{w\in \Lambda^0} e^{-\beta r\cdot p}B^p(w,v)W_{v}\\
&=\sum_{0\leq p\leq q} \sum_{w\in \Lambda^0} e^{-\beta r\cdot p}\big(B^pW\big)_{w}\\
&=\sum_{w\in \Lambda^0} \Big(\sum_{0\leq p\leq q} e^{-\beta r\cdot p}B^pW\Big)_{w}\\
&=\sum_{w\in \Lambda^0} \Big(\prod_{i=1}^k\Big(\sum_{p_i=0} ^{q_i}e^{-\beta r_ip_i}B_i^{p_i}\Big)W\Big)_{w}.
\end{align*}
Substituting $W$ back in this equation gives
\[\phi(M_q)=\sum_{w\in \Lambda^0} \Big(\prod_{i=1}^k\Big(\sum_{p_i=0} ^{q_i}e^{-\beta r_ip_i}B_i^{p_i}\Big) \big(1 - e^{-\beta r_i}B_i\big)\Big)m^\phi\Big)_{w}.\]
For each $1\leq i\leq k$, since $\beta r_i>\ln\rho(B_i)$, we have $e^{\beta r_i}>\rho(B_i)$. Then the series $\sum_{p_i}e^{-\beta p_i}B_i^{p_i}$ converges with sum $(1-e^{-\beta r_i}B_i)^{-1}$. Now letting $q_i\to \infty$, we have 
\[
\phi(M_q)\to \sum_{w\in \Lambda^0} \Big(\prod_{i=1}^k\Big(\big(1-e^{-\beta r_i}B_i\big)^{-1} \big(1 - e^{-\beta r_i}B_i\big)\Big)m^\phi\Big)_{w} = \sum_{w \in \Lambda^0} m_w^\phi =1,
\]
proving the claim.

Next, since $\phi(M_q)\to 1$, \cite[Lemma 7.3]{lrr} implies that $\phi(M_q a M_q) \to \phi(a)$ as $q\to \infty$. 
Since $t_\lambda P t_\lambda^*$ is fixed by the action $\sigma$, the KMS condition implies 
$
\phi\big((t_\lambda P t_\lambda^*)a(t_\mu Pt_\mu^*)\big)=0$ whenever $\lambda\neq \mu$.
Now, another application of the KMS condition and \eqref{pug} gives
\begin{align*}
\phi(a)&=\sum_{p\in \N^k}\sum_{\lambda \in\Lambda^p} \phi(t_\lambda P t_\lambda^* a t_\lambda P t_\lambda ^*)\\
&=\sum_{p\in \N^k}\sum_{\lambda \in\Lambda^p} e^{-\beta r\cdot p} \phi(P t_\lambda ^* a t_\lambda Pt_\lambda ^*t_\lambda ) \\
&=\sum_{p\in \N^k} \sum_{\lambda \in\Lambda^p} e^{-\beta r\cdot p} \phi(P t_\lambda^* at_\lambda t_\lambda^*t_\lambda P),
\end{align*}
and after rewriting each summand using \eqref{omegap}, we get 
\[\phi(a)
=\sum_{p\in \N^k}\sum_{\lambda \in\Lambda^p} e^{-\beta r\cdot p} \phi(P) \omega_{P,\phi}(t_\lambda^* a t_\lambda),\]
giving the reconstruction formula~\eqref{KMS-reconstr-formula}.
\end{proof}

\begin{proof}[Proof of Theorem~\ref{Thm:KMS_beta_tau}\eqref{KMS_tau_affine}]
Since the set of tracial states on $C^*(G)$ and the set of KMS$_\beta$-states of $(\TT(G,\Lambda), \sigma)$ are compact in the weak$^*$ topology, it suffices to show that the map $\tau \mapsto \phi_{\tau}$ is continuous and bijective. 

An argument using the monotone convergence theorem shows that $\tau \mapsto \phi_{\tau}$ is affine and continuous. 
To see that it is injective, let $P$ be as in \eqref{def-p}. For each tracial state $\tau$ of $C^*(G)$, the formulas \eqref{Fock-U} and \eqref{Fock-T} for the family $(U,T)$ of part~\eqref{Zconverges} together with the adjoint formula \eqref{adjoint-T} show that $\pi_{U,T}(P)$ vanishes on elements $\delta_\lambda \otimes \xi_\tau$ with $d(\lambda)\neq 0$, and fixes elements $\delta_v \otimes \xi_\tau$ for all $v\in \Lambda^0$. Therefore the formula~\eqref{psi_a} for $\phi_\tau$ implies that 
\begin{align*}
 \phi_{\tau}(Pu_gP) &= Z(\beta, \tau)\inv \sum_{v\in \Lambda^0} \big(\pi_{U,T}(Pu_g)(\delta_v \otimes \xi_\tau) |\delta_v \otimes \xi_\tau\big)\\
 &= Z(\beta, \tau)^{-1}\big(\delta_{\D(g)} \otimes \pi_\tau(i_g) \xi_\tau) |\delta_{\D(g)} \otimes \xi_\tau\big)\\
 &=Z(\beta, \tau)^{-1}\tau(i_{g}). 
 \end{align*}
Since $\sum_{g\in G^0}u_g=1$, the linearity of $ \phi_{\tau}$ implies that $ \phi_{\tau}(P) = \phi_{\tau}(P(\sum_{g\in G^0}u_g)P)= Z(\beta, \tau)\inv$. Therefore $\omega_{P,\phi_\tau}(u_g)=\tau(i_g)$. Since 
 $u_g=\psi_0(i_g)$, we have $\tau = \omega_{P, \phi_{\tau}} \circ \psi_0$, and hence the map $\tau \mapsto \phi_{\tau}$ is injective. 

For the surjectivity, let $\phi$ be a KMS$_\beta$-state for $(\TT(G,\Lambda),\sigma)$. By Lemma~\ref{Lemma:phi-P}, $\tau:=\omega_{P,\phi}\circ \psi_0$ is a tracial state on $C^*(G)$. Applying part~\eqref{Zconverges} to $\tau$ gives a KMS$_\beta$-state $\phi_\tau$. The argument of the previous paragraph shows that $\omega_{P,\phi}\circ \psi_0 = \omega_{P,\phi_\tau}\circ \psi_0$. Therefore $\omega_{P,\phi}(u_g) = \omega_{P,\phi_\tau}(u_g)$ for all $g\in G$. Lemma~\ref{Lemma:phi-P} now implies that $\phi(u_g) = \phi_\tau(u_g)$ for all $g\in G$, and hence by Proposition~\ref{prop:algebraic-chara} we have $\phi=\phi_\tau$.
\end{proof}

\section{The $G$-periodicity group}

In this section we define the $G$-periodicity group for a self-similar action of a groupoid $G$ on a strongly connected $k$-graph $\Lambda$. We associate two functions $\theta, h$ to the $G$-periodicity group, which play an important role in the KMS structure of $\OO(G,\Lambda)$. Since our groupoid actions are local and faithful, our $G$-periodicity group generalises all three different notions of periodicity for group actions in \cite{LY}.  When the groupoid acting is trivial, $\theta$ is the bijection introduced in \cite{aHLRS}. 

Let $(G,\Lambda)$ be a self-similar action, where $\Lambda$ is a strongly connected finite $k$-graph. Let $\Lambda^\infty$ be the infinite path space, and $\sigma$ the shift map. We define the action of $G$ on $\Lambda^\infty$ by 
\[
(g\cdot x)(p,q):=g|_{x(0,p)}\cdot x(p,q) \text { for all } p\leq q \in \N^k.
\]
It is straightforward to see that for all composable elements $g,h\in G$, we have 
\[(gh)\cdot x=g\cdot(h\cdot x)\text{ for } x\in \D(h)\Lambda^\infty,\]
and each $\lambda\in \D(g)\Lambda$ satisfies 
\[g\cdot(\lambda y)=(g\cdot \lambda)\big(g|_\lambda \cdot y\big) \text{ for }y\in \D(g|_\lambda)\Lambda^\infty.\]
We define
\begin{equation}\label{eq:theGperiodicitygroup}
\Per(G,\Lambda) := \left\{ p-q \ : \ \begin{tabular}{@{}p{\textwidth-19em}@{}}$p,q\in \N^k$, and there exists $g\in G$ such that $\sigma^p( x)=\sigma^q( g\cdot x)$ for all $x\in \D(g)\Lambda^\infty$\end{tabular} \right\}.
\end{equation}

\begin{prop}\label{prop:per}
Let $\Lambda$ be a strongly connected finite $ k$-graph and let $(G,\Lambda)$ be a self-similar groupoid action. Suppose that $ g\in G$ and $p,q\in \N^k$ satisfy $\sigma^p(x)=\sigma^q (g\cdot x)$ for all $x\in \D(g)\Lambda^\infty$. 
\begin{enumerate}
\item\label{thetah1} For each $v\in \Lambda^{0}$, there is $g_v\in G$ such that $\D(g_v)=v$, and $\sigma^p( x)=\sigma^q( g_v\cdot x)$ for all $x\in \D(g_v)\Lambda^\infty$.
\item \label{prop:per2}If $p-q=m-n$, then there exists $h\in G$ such that 
\[\sigma^{m}( y)=\sigma^{ n} (h\cdot y) \text{ for all }y\in \D(h)\Lambda^\infty.\]
\item \label{prop:per3}The set $\Per(G,\Lambda)$ from \eqref{eq:theGperiodicitygroup} is a subgroup of $\Z^k$, we call the $G$-periodicity group of $(G,\Lambda)$.
\end{enumerate} 
\end{prop}

\begin{proof}
For \eqref{thetah1}, fix $v\in \Lambda^0$. Since $\Lambda$ is strongly connected, there is $\lambda\in \D(g)\Lambda v$. Writing $g_v:=g|_\lambda$, we have $\D(g_v)=v$. Fix $x\in \D(g_v)\Lambda^\infty$. Since $\lambda x \in \D(g)\Lambda^\infty$, we have
\[\sigma^p(x)=\sigma^{p+d(\lambda)}(\lambda x)=\sigma^{d(\lambda)}\big(\sigma^{q}(g\cdot (\lambda x))\big)=\sigma^{q}\big(\sigma^{d(\lambda)}((g\cdot \lambda)( g|_\lambda\cdot x))\big)=\sigma^q( g_v\cdot x).\]

For \eqref{prop:per2}, note that since $\Lambda$ is strongly connected, there is $\eta\in\Lambda^{p}\c(g)$. Let $g_{r(\eta)}\in G$ be as in part~\eqref{thetah1}, and let $h:=g_{r(\eta)}|_{\eta}$. By (\ref{S3}) we have $\D(h)=s(\eta)=\c(g)$, and hence for each $y\in \D(h)\Lambda^\infty$, the element $\eta y$ belongs to $\D(g|_{r(\eta)})\Lambda^\infty$. Applying part~\eqref{thetah1} now gives 
\begin{align*}
\sigma^{m}(y)&=\sigma^{m+p}(\eta y)=\sigma^{m}\big(\sigma^{p}(\eta y)\big)=\sigma^{m}\big(\sigma^{q}\big(g_{r(\eta)}\cdot (\eta y)\big)\big)
=\sigma^{n+p}\big(g_{r(\eta)}\cdot (\eta y)\big)\\
&=\sigma^{n+p}\big((g_{r(\eta)}\cdot \eta)( h\cdot y)\big)=\sigma^{n}(h\cdot y).
\end{align*}

For \eqref{prop:per3}, take $v\in \Lambda^0$. Since $ \sigma^{0}( x)=x= \id_v\cdot x=\sigma^{0}(\id_v\cdot x)$ for all $x\in \D(\id_v)\Lambda^\infty$, we have $0\in \Per(G,\Lambda)$. To see that $\Per(G,\Lambda)$ is closed under inverses, fix $p-q\in \Per(G,\Lambda)$. Then there is $g\in G$ such that $\sigma^p( x)=\sigma^q(g\cdot x)$ for all $x\in \D(g)\Lambda^\infty$. Now, for $h:=g^{-1}$ we have
 $\sigma^p( h\cdot y)=\sigma^{q}(y)$ for all $y\in \D(h)\Lambda^\infty$, and hence $q-p\in \Per(G,\Lambda)$. It remains to show that $\Per(G,\Lambda)$ is closed under addition. Fix $p-q, m-n\in \Per(G,\Lambda)$. Then there is $g\in G$ such that $\sigma^p( x)=\sigma^{q}(g\cdot x)$ for all $x\in \D(g)\Lambda^\infty$. Applying part~\eqref{prop:per2} with $m-n$ and $v=\c(g)$ gives $g'\in G$ such that $\D(g')=\c(g)$ and $ \sigma^{m}( y)=\sigma^{n}(g'\cdot y)$ for all $y\in \D(g')\Lambda^\infty$. Let $h:=g'g$. Then for each $x\in \D(h)\Lambda^\infty= \D(g)\Lambda^\infty$, we have
\begin{align*}
\sigma^{p+m}(x)&=\sigma^{m}(\sigma^{p}(x))=\sigma^{m}\big(\sigma^{q}(g\cdot x)\big)=\sigma^{q}\big(\sigma^{m}(g\cdot x)\big).
\end{align*}
Since $g\cdot x\in \c(g)\Lambda^\infty=\D(g')\Lambda^\infty$, we get 
\begin{align*}
\sigma^{p+m}(x)=\sigma^{q}\big(\sigma^{n}\big(g'\cdot (g\cdot x)\big)\big)=\sigma^{q+n}(h\cdot x),
\end{align*}
and hence $(p-q)+(m-n)\in \Per(G,\Lambda)$, as required.
\end{proof}

\begin{lemma}\label{lemma:htehta}
Let $\Lambda$ be a strongly connected finite $ k$-graph, and $(G,\Lambda)$ be a self-similar groupoid action. Suppose that $p-q\in \Per(G,\Lambda)$.
\begin{enumerate}
\item \label{lemma:htehta-1}For each $\mu\in \Lambda^p$, there is a unique pair $\big(\theta_{p,q}(\mu), h_{p,q}(\mu)\big)\in \Lambda^q\times G$ such that 
\begin{align}\label{pershift}
\mu (h_{p,q}(\mu)\cdot x)=\theta_{p,q}(\mu)x, \quad\text{for all } x\in \D(h_{p,q}(\mu))\Lambda^\infty.
\end{align}
The function $\theta_{p,q}:\Lambda^p\to \Lambda^q$ is range-preserving, and the function $h_{p,q}:\Lambda^p\to G$ satisfies 
\[\c\big(h_{p,q}(\mu)\big)=s(\mu) \text{ and }\D\big(h_{p,q}(\mu)\big)=s\big(\theta_{p,q}(\mu)\big), \quad\text{for all } \mu\in \Lambda^p.\]
\item \label{lemma:htehta-2}If in addition $q-n\in \Per(G,\Lambda)$, then $p-n\in \Per(G,\Lambda)$ and $\theta_{q,n}\circ\theta_{p,q}=\theta_{p,n}$. The function $h_{p,n}:\Lambda^p\to G$ satisfies 
\begin{equation} \label{h-compo}
h_{p,n}(\mu)= h_{p,q}(\mu)h_{q,n}\big(\theta_{p,q}(\mu)\big), \quad\text{for all } \mu\in \Lambda^p.
\end{equation}
\item \label{lemma:htehta-3}For each $p\in \N^k$, the function $\theta_{p,p}:\Lambda^p\to \Lambda^p$ is the identity map, and $h_{p,p}:\Lambda^p\to G$ is given by $h_{p,p}(\mu)=\id_{s(\mu)}$.
\item \label{lemma:htehta-4}Each $\theta_{p,q}:\Lambda^p\to \Lambda^q$ 
is a bijection
with $\theta_{p,q}=\theta_{q,p}^{-1}$, and for each $ \mu\in \Lambda^p$, the groupoid element $h_{p,q}(\mu)$ satisfies
\begin{equation}\label{formula h invers}
h_{p,q}(\mu)^{-1}=h_{q,p}\big(\theta_{p,q}(\mu)\big).
\end{equation}
\end{enumerate} 
\end{lemma}

\begin{proof}
For \eqref{lemma:htehta-1}, we first check the uniqueness. Suppose that there are pairs $(\nu_1,g_1),(\nu_2,g_2)\in\Lambda^q\times G$ that satisfy $\eqref{pershift}$. Since $\c(g_1)=s(\mu)=\c(g_2)$, $\D(g_2^{-1})=\c(g_1)$. Therefore $g_2^{-1}g_1$ is an element of $G$ and for all $x\in \D(g_1)\Lambda^\infty$, we have 
\[\nu_1 x=\mu(g_1\cdot x)=\mu((g_2g_2^{-1}g_1)\cdot x)=\nu_2(g_2^{-1}g_1)\cdot x\]
Since $d(\nu_1)=d(\nu_2)$, we must have $\nu_1=\nu_2$. This then implies that $x=(g_2^{-1}g_1)\cdot x$, and hence $\D(g_1)=\D(g_2)$ and $g_2\cdot \lambda=g_1\cdot \lambda$ for all $\lambda\in\D(g_1)\Lambda$. Since the action is faithful, we have $g_1=g_2$.

For the existence, fix $\mu\in \Lambda^p$. Since $\Lambda$ is strongly connected, there is $\eta\in\Lambda^q r(\mu)$. Let $g$ be the element of $G$ obtained by applying Proposition~\ref{prop:per}\eqref{thetah1} with the vertex $r(\eta)$. Let $\zeta:=(\eta\mu)(0,p)$ and $\nu:=(\eta\mu)(p,p+q)$. Then for $x\in s(\mu)\Lambda^\infty$, we have 
\begin{equation}\label{equ-lemma:htehta-1}
\nu x=\sigma^{p}(\zeta \nu x)=\sigma^{q}\big(g\cdot (\zeta\nu x)\big)=\sigma^{q}\big(g\cdot (\eta\mu x)\big)=g|_\eta\cdot (\mu x).
\end{equation}
This means $g|_\eta ^{-1}\cdot (\nu x) = \mu x$. We define 
\[
\theta_{p,q}(\mu):=g|_\eta ^{-1}\cdot\nu\quad\text{and}\quad h_{p,q}(\mu):=((g|_\eta ^{-1})|_\nu) ^{-1}. 
\]
For all $y\in \D(h_{p,q}(\mu))\Lambda^\infty$, we have $r(h_{p,q}(\mu)\cdot y)=\c(h_{p,q}(\mu))=s(\nu)=s(\mu)$. Then applying~\eqref{equ-lemma:htehta-1} with $x=h_{p,q}(\mu)\cdot y$ gives
\[
\mu(h_{p,q}(\mu)\cdot y)=(g|_\eta)^{-1}\cdot \big(\nu(h_{p,q}(\mu)\cdot y)\big )=\theta_{p,q}(\mu)\big((g|_\eta ^{-1})|_\nu( h_{p,q}\big(\mu)\cdot y\big)\big)=\theta_{p,q}(\mu)y,
\]
and so \eqref{lemma:htehta-1} holds.

For \eqref{lemma:htehta-2}, first note that since $\Per(G,\Lambda)$ is a group, we have $p-n\in\Per(G,\Lambda)$. Two applications of \eqref{pershift} gives
\begin{align*}
\theta_{q,n}\big(\theta_{p,q}(\mu)\big)x=\theta_{p,q}(\mu)\Big(h_{q,n}\big(\theta_{p,q}(\mu)\big)\cdot x\Big)
=\mu\Big(\Big( h_{p,q}(\mu)h_{q,n}\big(\theta_{p,q}(\mu)\big)\Big)\cdot x\Big).
\end{align*}
Now, the uniqueness in part~\eqref{lemma:htehta-1} gives $\theta_{q,n}\circ\theta_{p,q}=\theta_{p,n}$ and \eqref{h-compo} simultaneously. 

For \eqref{lemma:htehta-3}, since $\mu(\id_{s(\mu)}\cdot x)=\mu x$ for all $x\in s(\mu)\Lambda^\infty$, the uniqueness in part~\eqref{lemma:htehta-1} implies that $\theta_{p,p}(\mu)=\mu$ and $h_{p,p}(\mu)=\id_{s(\mu)}$ for all $\mu\in \Lambda^p$. Finally, part~\eqref{lemma:htehta-2} implies that $\theta_{q,p}\circ \theta_{p,q}=\theta_{p,p}$ and $\theta_{p,q}\circ \theta_{q,p}=\theta_{q,q}$ and therefore each $\theta_{p,q}$ 
is a bijection. The formula \eqref{formula h invers} follows from \eqref{h-compo} by letting $n=p$.
\end{proof}

The following technical lemma describes the interaction between the functions $\theta_{p,q}$ and $h_{p,q}$.

\begin{lemma}\label{lemma:mor-prop-htheta}
Let $p-q\in \Per(G,\Lambda)$, $\theta_{p,q},h_{p,q}$ be the functions of Lemma~\ref{lemma:htehta}, and $\mu\in \Lambda^p$. 
\begin{enumerate}
\item \label{lemma:mor-prop-htheta-1}For each $\eta\in s(\mu)\Lambda$, we have
\[
\theta_{p+d(\eta),q+d(\eta)}(\mu\eta)=\theta_{p,q}(\mu)\big(h_{p,q}(\mu)^{-1}\cdot \eta\big)\quad\text{and}\quad h_{p+d(\eta),q+d(\eta)}(\mu\eta)=h_{p,q}(\mu)|_{h_{p,q}(\mu)^{-1}\cdot \eta}.
\]
\item \label{lemma:mor-prop-htheta-2}For each $\beta \in \Lambda r(\mu)$, we have 
\[
\theta_{d(\beta)+p,d(\beta)+q}(\beta\mu)=\beta \theta_{p,q}(\mu)\quad\text{and}\quad h_{d(\beta)+p,d(\beta)+q}(\beta \mu)=h_{p,q}(\mu).
\]
\item \label{lemma:mor-prop-htheta-3}For each $g\in G$ with $\D(g)=r(\mu)$, we have
\[
\theta_{p,q}(g\cdot \mu)=g\cdot \theta_{p,q}(\mu)\quad\text{and}\quad h_{p,q}(g\cdot \mu)=g|_\mu h_{p,q}(\mu)g|_{\theta_{m,n}(\mu)}^{-1}.
\]
\end{enumerate}
\end{lemma}
\begin{proof}
For convenience we write $g:=h_{p,q}(\mu)$. Since $\D(g^{-1})=\c(g)=s(\mu)=r(\eta)$, the right-hand sides of the equations in part~\eqref{lemma:mor-prop-htheta-1} make sense. Now, (\ref{S3}) implies that $\c(g|_{g^{-1}\cdot \eta})=s(\eta)$, and for all $ x\in s(g^{-1}\cdot \eta)\Lambda^\infty$ we have
\[
\mu\eta\big(g|_{g^{-1}\cdot \eta}\cdot x\big)=\mu\big((g g^{-1})\cdot \eta\big) \big(g|_{g^{-1}\cdot \eta}\cdot x\big)=\mu(g \cdot(g^{-1}\cdot \eta))\big(g|_{g^{-1}\cdot \eta}\cdot x\big)=\mu\big (g\cdot\big((g^{-1}\cdot \eta)x\big)\big ). 
\]
The right-hand expression is equal to $\theta_{p,q}(\mu)(g^{-1}\cdot \eta)x$ by \eqref{pershift}. The uniqueness in Lemma~\ref{lemma:htehta}\eqref{lemma:htehta-1} now gives part~\eqref{lemma:mor-prop-htheta-1}. 

Similarly, for each $x\in \D(h_{p,q}(\mu))$, equation \eqref{pershift} implies that $\beta \mu\big(h_{p,q}(\mu)\cdot x\big)=\beta \theta_{p,q}(\mu) x$, and another application of uniqueness in Lemma~\ref{lemma:htehta}\eqref{lemma:htehta-1}
gives part~\eqref{lemma:mor-prop-htheta-2}

To see that \eqref{lemma:mor-prop-htheta-3} holds, first note that Lemma~\ref{lemma:htehta}\eqref{lemma:htehta-1} says that the function $\theta_{p,q}$ is range-preserving, and therefore $g\cdot \theta_{p,q}(\mu)$ make sense. This also means we can compose the groupoid elements $g|_\mu, h_{p,q}(\mu)$, and $g|_{\theta_{p,q}(\mu)}^{-1}$ to form $g|_\mu h_{p,q}(\mu)g|_{\theta_{p,q}(\mu)}^{-1}$. Now, for each $x\in \D(g|_{\theta_{p,q}(\mu)}^{-1})\Lambda^\infty$ we can apply equation \eqref{pershift} to get
\[
\mu \Big(h_{p,q}(\mu)\cdot \Big(g|_{\theta_{p,q}(\mu)}^{-1}\cdot x\Big)=\theta_{p,q}(\mu) \Big(g|_{\theta_{p,q}(\mu)}^{-1}\cdot x\Big).
\]
Hence we have
\begin{align*}
(g\cdot \mu)\Big(\Big (g|_\mu h_{p,q}(\mu)g|_{\theta_{p,q}(\mu)}^{-1}\Big)\cdot x\Big)&=g\cdot\Big(\mu \Big(h_{p,q}(\mu)\cdot \Big(g|_{\theta_{p,q}(\mu)}^{-1}\cdot x\Big)\Big)\\
&=g\cdot\Big(\theta_{p,q}(\mu) \Big(g|_{\theta_{p,q}(\mu)}^{-1}\cdot x\Big)\Big)\\
&= \big(g\cdot\theta_{p,q}(\mu) \big)x.
\end{align*}
Again, the uniqueness in Lemma~\ref{lemma:htehta}\eqref{lemma:htehta-1} now gives part~\eqref{lemma:mor-prop-htheta-3}.
\end{proof}

\begin{remark}
If $\Lambda$ has a single vertex, and hence $G$ is a \emph{group} acting self-similarly on $\Lambda$, then for each $p-q\in \Per(G,\Lambda)$ and $\lambda\in \Lambda^p$, the triple $(\lambda,h_{p,q}(\lambda),\theta_{p,q}(\lambda))$ is a cycline triple in the sense of \cite{LY}. The group $\Per(G,\Lambda)$ also coincides with the $G$-periodicity group from \cite{LY}. While our approach to defining the $G$-periodicity group is different to that in \cite{LY}, some of our results follow from their results in the single-vertex setting, when the action is pseudo free and source invariant. In particular, Lemmas~\ref{lemmayl0} and \ref{lem:minimalex-measure} follow from \cite[Lemma~6.5 and Lemma~6.6]{LY}; and Proposition~\ref{prop:rep-per} and Lemma~\ref{lemma:rep-per} follow from \cite[Theorem~4.9 and Lemma~6.9]{LY}. 
\end{remark}

We finish this section, by defining the notion of $G$-aperiodicity for a self-similar action $(G,\Lambda)$. We will see in Section~\ref{main-cuntz} that this condition reduces the number KMS states of $\OO(G,\Lambda)$.

\begin{definition}
Let $\Lambda$ be a strongly connected finite $k$-graph, and $(G,\Lambda)$ be a self-similar groupoid action. We say $\Lambda$ is $G$-\textit{aperiodic} if for every $g\in G$, and $p\neq q\in \N^k$, there exists a $x\in \c(g) \Lambda^\infty$ such that 
\[\sigma^p( x)\neq \sigma^q( g\cdot x).\]
\end{definition}

\begin{remark}\label{rem:aperiodicity}
Observe that $\Lambda$ is $G$-aperiodic if and only if $\Per(G,\Lambda)=\{0\}.$
\end{remark}

\section{A central representation of $C^*(\Per(G,\Lambda))$ on $\OO(G,\Lambda)$}

In our main result on the KMS structure of $\OO(G,\Lambda)$, we will prove that KMS states are determined by states of $C^*(\Per(G,\Lambda))$. For this we first need to represent $C^*(\Per(G,\Lambda))$ in $\OO(G,\Lambda)$.

\begin{prop}\label{prop:rep-per}
Let $\Lambda$ be a strongly connected finite $ k$-graph, and $(G,\Lambda)$ be a self-similar groupoid action. For $p-q\in \Per(G,\Lambda)$, let $\theta_{p,q}$ and $h_{p,q}$ be the functions from Lemma~\ref {lemma:htehta}. Then there is a central unitary representation $V$ of $\Per(G,\Lambda)$ in $\OO(G,\Lambda)$ given by
\[V_{p-q}=\sum_{\lambda\in \Lambda^p}t_\lambda u_{h_{p,q}(\lambda)}t_{\theta_{p,q}(\lambda)}^*.\]
\end{prop}

To prove Proposition~\ref{prop:rep-per} we need the following two lemmas.

\begin{lemma}
Let $\Lambda$ be a strongly connected finite $ k$-graph, and $(G,\Lambda)$ be a self-similar groupoid action. Suppose that $p-q\in \Per(G,\Lambda)$ and $\lambda\in \Lambda^p$. Let $m:=(p\vee q)-p$ and $n:=(p\vee q)-q$. Then 
\[\Lambda^{\min}\big(\theta_{p,q}(\lambda),\lambda\big)=\big\{\big(\eta\,,\,h_{p,q}(\lambda)\cdot\theta_{n,m}(\eta)\big): \eta\in s(\theta_{p,q}(\lambda))\Lambda^n\big\}.\]
\end{lemma}

\begin{proof}
Suppose that $(\eta,\zeta)\in \Lambda^{\min}\big(\theta_{p,q}(\lambda),\lambda\big)$. Then $\theta_{p,q}(\lambda)\eta=\lambda\zeta$, $d(\eta)=n$ and $d(\zeta)=m$. Now let $y\in \D(h_{n,m}(\eta))\Lambda^\infty$, and $x=h_{n,m}(\eta)\cdot y$. Then $x\in s(\eta)\Lambda^\infty=s(\zeta)\Lambda^\infty$ by Lemma~\ref{lemma:htehta}\eqref{lemma:htehta-1}, and we have
\[\zeta x=\sigma^p(\lambda \zeta x )=\sigma^p\big(\theta_{p,q}(\lambda)\eta x\big)=\sigma^p\big(\lambda \big(h_{p,q}(\lambda)\cdot(\eta x)\big)\big)=h_{p,q}(\lambda)\cdot(\eta x).\]
An application of \eqref{pershift} implies that
\[\zeta x=h_{p,q}(\lambda)\cdot\big(\theta_{n,m}(\eta) y\big)=\big(h_{p,q}(\lambda)\cdot \theta_{n,m}(\eta)\big) \big(h_{p,q}(\lambda)|_{\theta_{n,m}(\eta)} \cdot y \big).\]
Since the action is degree preserving, by the factorisation property of $\Lambda$ we have $\zeta=h_{p,q}(\lambda)\cdot \theta_{n,m}(\eta)$. This gives the $\subseteq$ containment. 

For the reverse containment, let $\eta\in s(\theta_{p,q}(\lambda))\Lambda^n$, and $x=h_{n,m}(\eta)\cdot y$ for some $y\in \D(h_{n,m}(\eta))\Lambda^\infty$. Then 
\[\theta_{p,q}(\lambda)\eta x=\lambda \big(h_{p,q}(\lambda)\cdot (\eta x)\big)=\lambda \big(h_{p,q}(\lambda)\cdot (\theta_{n,m}(\eta)y\big)=\lambda\big(h_{p,q}(\lambda)\cdot \theta_{n,m}(\eta)\big) \big(h_{p,q}(\lambda)|_{\theta_{n,m}(\eta)} \cdot y \big).\]
Since $q+n=p+m=p\vee q$, we have $\theta_{p,q}(\lambda)\eta=\lambda\big(h_{p,q}(\lambda)\cdot \theta_{n,m}(\eta)\big)$, as required.
\end{proof}

\begin{lemma}\label{lemma:rep-per}
Let $\Lambda$ be a strongly connected finite $ k$-graph, and $(G,\Lambda)$ be a self-similar groupoid action. Let $p-q\in \Per(G,\Lambda)$. Then $t_\lambda t_\lambda^*=t_{\theta_{p,q}(\lambda)}t_{\theta_{p,q}(\lambda)}^*$ for all $\lambda\in \Lambda^p$, and 
$V_{p-q}=\sum_{\lambda\in \Lambda^p}t_\lambda u_{h_{p,q}(\lambda)}t_{\theta_{p,q}(\lambda)}^*$ is a unitary in $\OO(G,\Lambda)$.
\end{lemma}

\begin{proof}
Let $m:=(p\vee q)-p$ and $n:=(p\vee q)-q$. Using relation (\ref{CK}) gives
\[
t_{\theta_{p,q}(\lambda)}t_{\theta_{p,q}(\lambda)}^*=t_{\theta_{p,q}(\lambda)}\sum_{\eta\in s(\theta_{p,q}(\lambda))\Lambda^n}t_\eta t_\eta^*t_{\theta_{p,q}(\lambda)}^*=\sum_{\eta\in s(\theta_{p,q}(\lambda))\Lambda^n}t_{\theta_{p,q}(\lambda)\eta} t_{\theta_{p,q}(\lambda)\eta}^*.
\]
Applying \eqref{pershift} to each summand now gives
\begin{align*}
t_{\theta_{p,q}(\lambda)}t_{\theta_{p,q}(\lambda)}^*&=\sum_{\eta\in s(\theta_{p,q}(\lambda))\Lambda^n}t_{\lambda(h_{p,q}(\lambda)\cdot\theta_{n,m}(\eta))} t_{\lambda(h_{p,q}(\lambda)\cdot\theta_{n,m}(\eta))}^*\\
&=\sum_{\eta\in s(\theta_{p,q}(\lambda))\Lambda^n}t_{\lambda}t_{h_{p,q}(\lambda)\cdot\theta_{n,m}(\eta)} t_{h_{p,q}(\lambda)\cdot\theta_{n,m}(\eta)}^*t_\lambda^*.
\end{align*}
Since $t_{g\cdot \nu}t_{g\cdot \nu}^*=t_{g\cdot \nu}u_{g|_{\nu}}u_{g|_{\nu}}^*t_{g\cdot \nu}^*=u_gt_\nu t_\nu^* u_g$ for all $\nu\in \D(g)\Lambda$, we have
\begin{align}\label{thetakomaki}
t_{\theta_{p,q}(\lambda)}t_{\theta_{p,q}(\lambda)}^*
\notag&=\sum_{\eta\in s(\theta_{p,q}(\lambda))\Lambda^n}t_{\lambda}u_{h_{p,q}(\lambda)}t_{\theta_{n,m}(\eta)}t_{\theta_{n,m}(\eta)}^*u_{h_{p,q}(\lambda)}t_\lambda^*\\
&=t_{\lambda}u_{h_{p,q}(\lambda)}\sum_{\eta\in s(\theta_{p,q}(\lambda))\Lambda^n}t_{\theta_{n,m}(\eta)}t_{\theta_{n,m}(\eta)}^*u_{h_{p,q}(\lambda)}t_\lambda^*.
\end{align}
Since $\theta_{p,q}$ is a range-preserving bijection, we can use (\ref{CK}) to get
\[\sum_{\eta\in s(\theta_{p,q}(\lambda))\Lambda^n}t_{\theta_{n,m}(\eta)}t_{\theta_{n,m}(\eta)}^*=\sum_{\alpha\in s(\theta_{p,q}(\lambda))\Lambda^m}t_{\alpha}t_{\alpha}^*=u_{s(\theta_{n,m}(\lambda))}=u_{\D(h_{n,m}(\lambda))}.\]
Substituting this into \eqref{thetakomaki} then gives
\[t_{\theta_{p,q}(\lambda)}t_{\theta_{p,q}(\lambda)}^*=t_{\lambda}u_{h_{p,q}(\lambda)}u_{\D(h_{n,m}(\lambda))}u_{h_{p,q}(\lambda)}t_\lambda^*=t_{\lambda}u_{\c(h_{n,m}(\lambda))}t_\lambda^*=t_{\lambda}u_{s(\lambda))}t_\lambda^*=t_{\lambda}t_\lambda^*.\]

For the second statement, we compute
\[V_{p-q}^*V_{p-q}=\sum_{\lambda,\mu\in \Lambda^p}t_{\theta_{p,q}(\lambda)} u_{h_{p,q}(\lambda)}^*t_\lambda^* t_\mu u_{h_{p,q}(\mu)}t_{\theta_{p,q}(\mu)}^*=\sum_{\lambda\in \Lambda^p}t_{\theta_{p,q}(\lambda)} u_{h_{p,q}(\lambda)}^*u_{s(\lambda)}u_{h_{p,q}(\lambda)}t_{\theta_{p,q}(\lambda)}^*.\]
Since $s(\lambda)=\c\big(h_{p,q}(\lambda)\big)$ and $\D\big(h_{p,q}(\lambda)\big)=s\big(\theta_{p-q}(\lambda)\big)$, we have 
\[V_{p-q}^*V_{p-q}=\sum_{\lambda\in \Lambda^p}t_{\theta_{p,q}(\lambda)}t_{\theta_{p,q}(\lambda)}^*=\sum_{v\in \Lambda^0}\sum_{\mu\in v\Lambda^q}t_\mu t_\mu^*=\sum_{v\in \Lambda^0}t_v=1.\]
Similarly, we can show that $V_{p-q}V_{p-q}^*=1.$
\end{proof}

\begin{proof}[Proof of Proposition~\ref{prop:rep-per}]
We first show that $V$ is well defined. Fix $p-q\in \Per(G,\Lambda)$ and $n\in \N^k$. We claim that $V_{p+n,q+n}=V_{p,q}$. Using Proposition~\ref{prop:per}\eqref{prop:per3} we get
\begin{align}
V_{p+n,q+n}\notag&=\sum_{\lambda\in \Lambda^{p+n}}t_\lambda u_{h_{p+n,q+n}(\lambda)}t_{\theta_{p+n,q+n}(\lambda)}^*\\
\notag&=\sum_{\mu\in \Lambda^p}\sum_{\nu\in s(\mu)\Lambda^n}t_\mu t_\nu u_{h_{p+n,q+n}(\mu\nu)}t_{\theta_{p+n,q+n}(\mu\nu)}^*\\
&=\sum_{\mu\in \Lambda^p}\sum_{\nu\in s(\mu)\Lambda^n}t_\mu t_\nu u_{h_{p,q}(\mu)|_{h_{p,q}(\mu)^{-1}\cdot \nu}}t_{h_{p,q}(\mu)^{-1}\cdot \nu}^*t_{\theta_{p,q}(\mu)}^*\label{eq:Vwelldefined1}.
\end{align}
Now, for each $\mu\in \Lambda^p$ and $\nu\in s(\mu)\Lambda^n$, we have $r(h_{p,q}(\mu)^{-1}\cdot \nu)=\D(h_{p,q}(\mu))$. So the action of $h_{p,q}(\mu)^{-1}$ is a bijection of $s(\mu)\Lambda^n$ onto $\D(h_{p,q}(\mu))\Lambda^n$. We now rewrite \eqref{eq:Vwelldefined1} as 
\[
V_{p+n,q+n}=\sum_{\mu\in \Lambda^p}\sum_{\eta\in \D(h_{p,q}(\mu))\Lambda^n}t_\mu t_{h_{p,q}(\mu)\cdot \eta} u_{h_{p,q}(\mu)|_\eta}t_{ \eta}^*t_{\theta_{p,q}(\mu)}^*.
\]
We now use \eqref{UT} on the right-hand side of this expression to get
\[
V_{p+n,q+n}=\sum_{\mu\in \Lambda^p}t_\mu u_{h_{p,q}(\mu)}u_{\D(h_{p,q}(\mu))}t_{\theta_{p,q}(\mu)}^*=\sum_{\mu\in \Lambda^p}t_\mu u_{h_{p,q}(\mu)}t_{\theta_{p,q}(\mu)}^*=V_{p,q},
\]
and hence $V$ is well defined.

We next we show that $a\mapsto V_a$ is a homomorphism. Let $a,b\in \Per(G,\Lambda) $. As in \cite[Proposition~6.1]{aHLRS}, we can find $n,p,q\in \N^k$ such that $a=p-q$ and $b=q-n$. Since 
$t_{\theta_{p,q}(\lambda)}^* t_\mu=\delta_{\theta_{p,q}(\lambda),\mu} u_{s(\theta_{p,q}(\lambda))}$ and $s(\theta_{p,q}(\lambda))=\D(h_{p,q}(\lambda))$, we have 
\begin{align}
V_{a}V_b=V_{p-q}V_{q-n}\notag&=\sum_{\lambda\in \Lambda^{p}}\sum_{\mu\in \Lambda^{q}}t_\lambda u_{h_{p,q}(\lambda)}t_{\theta_{p,q}(\lambda)}^*t_{\mu} u_{h_{q,n}(\mu)}t_{\theta_{q,n}(\mu)}^*\\
&=\notag\sum_{\lambda\in \Lambda^{p}}t_\lambda u_{h_{p,q}(\lambda)} u_{h_{q,n}(\theta_{p,q}(\lambda))}t_{\theta_{q,n}(\theta_{p,q}(\lambda))}^*.
\end{align}
We can now apply Lemma~\ref{lemma:htehta}\eqref{lemma:htehta-2} on this sum to get 
\[
V_aV_b=\sum_{\lambda\in \Lambda^{p}}t_\lambda u_{h_{p,n}(\lambda)}t_{\theta_{p,n}(\lambda)}^*=V_{p,n}=V_{a+b}.
\]

It remains to see that $V$ is central. We claim that $t_\mu V_a =V_a t_\mu$ and $u_gV_a =V_a u_g$ for all $a\in \Per(G,\Lambda)$, $\mu\in \Lambda$ and $g\in G$. To see this, fix $a\in \Per(G,\Lambda)$, $\mu\in \Lambda$. 
Choose $p,q\in \N^k$ such that $a=p-q$ and let $m:=p+d(\mu)$ and $n:=q+d(\mu)$. By factoring $\nu\in \Lambda^p$, we can rewrite the sum
\begin{align}
V_{a}=V_{m,n}\notag&=\sum_{\nu\in \Lambda^{m}}t_\nu u_{h_{m,n}(\nu)}t_{\theta_{m,n}(\nu)}^*=
\sum_{\eta\in \Lambda^{d(\mu)}}\sum_{\lambda\in s(\eta)\Lambda^{p}}t_{\eta\lambda }u_{h_{d(\mu)+p,d(\mu)+q}(\eta\lambda)}t_{\theta_{d(\mu)+p,d(\mu)+q}(\eta\lambda)}^*\\
&=\notag\sum_{\eta\in \Lambda^{d(\mu)}}\sum_{\lambda\in s(\eta)\Lambda^{p}}t_{\eta\lambda }u_{h_{p,q}(\lambda)}t_{\theta_{p,q}(\lambda)}^*t_\eta^*\quad\text{ by Lemma }\ref{lemma:mor-prop-htheta}\eqref{lemma:mor-prop-htheta-2}
\end{align}
Since $t_\eta^* t_\mu=\delta_{\eta,\mu}u_{s(\eta)}$ and $s(\eta)=r(\lambda)=r\big(\theta_{p,q}(\lambda)\big)$, we have

\begin{align}
V_{a}t_\mu\notag&=\sum_{\lambda\in s(\mu)\Lambda^{p}}t_{\mu\lambda }u_{h_{p,q}(\lambda)}t_{\theta_{p,q}(\lambda)}^*=t_\mu\sum_{\lambda\in \Lambda^{p}}t_{\lambda }u_{h_{p,q}(\lambda)}t_{\theta_{p,q}(\lambda)}^*=t_\mu V_{a}.
\end{align}

To see $u_gV_a =V_a u_g$, note that $g|_\lambda h_{p,q}(\lambda)=h_{p,q}(g\cdot \lambda)g|_{\theta_{p,q}(\lambda)}$. Then the equation \eqref{UT} implies that
\begin{align}
u_gV_{a}=u_gV_{p-q}\notag&=\sum_{\lambda\in\D(g)\Lambda^{p}}t_{g\cdot \lambda }u_{g|_\lambda}u_{h_{p,q}(\lambda)}t_{\theta_{p,q}(\lambda)}^*=\sum_{\lambda\in\D(g)\Lambda^{p}}t_{g\cdot \lambda }u_{h_{p,q}(g\cdot \lambda)}u_{g|_{\theta_{p,q}(\lambda)}}t_{\theta_{p,q}(\lambda)}^*.
\end{align}
Since $\theta_{p,q}(g\cdot \lambda)=g\cdot \theta_{p,q}(\lambda)$, we get 
\begin{align}
u_gV_{a}=\sum_{\lambda\in\D(g)\Lambda^{p}}t_{g\cdot \lambda }u_{h_{p,q}(g\cdot \lambda)}u_{g|_{g^{-1}\cdot\theta_{p,q}(g\cdot \lambda)}}t_{g^{-1}\cdot\theta_{p,q}(g\cdot\lambda)}^*.
\end{align}
For all $\zeta\in \Lambda, g\in G$ with $r(\zeta)=\c(g)$, (\ref{S7}) and the formula \eqref{UT} imply that
\[u_{g|_{g^{-1}\cdot \zeta}} t_{g^{-1}\cdot\zeta}^*=u_{g^{-1}|_ \zeta}^* t_{g^{-1}\cdot\zeta}^*=\big(t_{g^{-1}\cdot\zeta}u_{g^{-1}|_ \zeta}\big)^*=(u_{g^{-1}}t_\zeta)^*=t_\zeta^* u_{g^{-1}}.\]
Therefore
\begin{align}
u_gV_{a}\notag&=\sum_{\lambda\in\D(g)\Lambda^{p}}t_{g\cdot \lambda }u_{h_{p,q}(g\cdot \lambda)}t_{\theta_{p,q}(g\cdot\lambda)}^*u_{g}\\
\notag&=\sum_{\eta\in\c(g)\Lambda^{p}}t_{\eta }u_{h_{p,q}(\eta)}t_{\theta_{p,q}(\eta)}^*u_{g}\\
\notag&=\notag\sum_{\eta\in\Lambda^{p}}t_{\eta }u_{h_{p,q}(\eta)}t_{\theta_{p,q}(\eta)}^*u_{g}\\
\notag&=V_au_g.
\end{align}
Now, Lemma~\ref{lemma:htehta}\eqref{lemma:htehta-4} allows us to write 
\[V_{p-q}^*=\sum_{\lambda\in \Lambda^p}t_{\theta_{p,q}(\lambda)} u_{h_{p,q}(\lambda)^{-1}}t_\lambda^*=\sum_{\mu\in \Lambda^q}t_{\mu} u_{h_{q,p}(\mu)}t_{\theta_{q,p}(\mu)}^*=V_{q-p}.\]
This means, for instance, that $V_at_\mu^*=(t_\mu V_{-a})^*= (V_{-a}t_\mu)^*=t_\mu^*V_a$. Similarly, we know that $V_a$ commutes with each $u_g$. It follows that $V$ is central.
\end{proof}

\section{Perron--Frobenius measure on the infinite-path space}

Suppose that $\Lambda$ is a strongly connected finite $k$-graph and let  $B_1,\dots ,B_k \in M_{\Lambda^0}[0,\infty)$ be the  adjacency matrices of $\Lambda$. Then \cite[Corollary~4.2]{aHLRS1} shows that  each $\rho(B_i)>0$, and there is a unique vector $x_\Lambda\in(0,\infty)^{\Lambda^0}$ with unit $1$-norm
such that $B_i x_\Lambda=\rho(B_i)x_\Lambda$ for all $1\leq i\leq k$. The vector $x_\Lambda$ is called the \textit{unimodular Perron--Frobenius eigenvector of $\Lambda$}. 

For each $\lambda\in \Lambda$, let 
\[
Z(\lambda):=\{x\in \Lambda^\infty: x=\lambda y \text{ for some } y\in \Lambda^\infty\}.
\]
The collection $\{Z(\lambda):\lambda\in \Lambda\}$ forms a basis of clopen sets for a Hausdorff topology on $\Lambda ^\infty$. 
By \cite[ Proposition~8.2]{aHLRS1}, there is a unique Borel probability measure $M$ on $\Lambda^\infty$ such that 
\begin{equation}\label{peron-measure-prop}
M(Z(\lambda))=\rho(\Lambda)^{d(\lambda)}M(Z(s(\lambda)))=\rho(\Lambda)^{d(\lambda)}x_\Lambda(s(\lambda))\quad \text{ for all } \lambda\in \Lambda,
\end{equation}
where  $\rho(\Lambda)^p:=\prod_{i=1}^k\rho(B_i)^{p_i}$ for all $p=(p_1,\dots, p_k)\in \N^k$. Following \cite[Proposition~4.2]{KP1}, we call $M$ the \textit{Parry measure}, and reserve the letter $M$ for this measure.
The $C^*$-algebra $C(\Lambda^\infty)$ is isomorphic to the commutative subalgebra $\overline{\operatorname{span}}\{t_\lambda t_\lambda^*: \lambda\in\Lambda\}$ of $\OO(G,\Lambda)$, where the characteristic function $1_{Z(\lambda)}$ is mapped to $t_\lambda t_\lambda^*$. Therefore  each KMS$_1$-state $\phi$ of $\OO(G,\Lambda)$ gives rise to a measure $M_\phi$ on $\Lambda^\infty$ satisfying \eqref{peron-measure-prop} and $M_\phi(Z(\lambda))=\phi(t_\lambda t_\lambda^*)$. This means  if $\phi$ is a KMS$_1$-state of $\OO(G,\Lambda)$, then we have 
\begin{equation}\label{peron-measure-prop1}
\phi(t_\lambda t_\lambda^*)=M(Z(\lambda))\quad \text{ and }\quad \phi(t_v)=M(Z(v))=x_\Lambda (v)
\end{equation}
 for all $\lambda\in \Lambda$ and $v\in \Lambda^0$. 
 
The measure $M$ plays a crucial role in building the unique KMS$_1$-state in \cite{aHLRS1,LY}. We will also need $M$ to construct KMS$_1$-states for our preferred dynamics on $\OO(G,\Lambda)$, but we first need to study $M$ in the presence of a self-similar groupoid action. 

\begin{lemma}\label{lemma:cg}
Let $\Lambda$ be a strongly connected finite $ k$-graph, $(G,\Lambda)$ be a self-similar groupoid action, and $x_\Lambda$ be the unimodular Perron--Frobenius eigenvector of $\Lambda$. Let $p-q\in \Per(G,\Lambda)$, $\lambda\in \Lambda^p$, and $ h_{p,q}(\lambda)$ be as in Lemma~\ref{lemma:htehta}. Write $N:=(1,\dots,1)\in \N^k$ and suppose that $g\in G$ satisfies $\c(g)=\c(h_{p,q}(\lambda))$, $\D(g)=\D(h_{p,q}(\lambda))$ and $g\neq h_{p,q}(\lambda)$. 
 For $v\in \Lambda^{0}$, $l\in \N$ define
\[F_{p,q,\lambda}^l(g,v):=\big\{\mu\in \D(g)\Lambda^{lN} v: g\cdot \mu=h_{p,q}(\lambda)\cdot \mu, g|_\mu=h_{p,q}(\lambda)|_ \mu\big\}\]
and
\[c_{p,q,\lambda}^l(g):=\rho(\Lambda)^{-lN}\sum_{v\in \Lambda^{0}}\big|F^l_{p,q,\lambda}(g,v)\big|x_\Lambda(v).\]
Then $\{c_{p,q,\lambda}^{l}(g)\}$ converges to a limit $c_{p,q,\lambda}(g)\in\big[0,x_\Lambda(\D(g))\big)$ as $l\to \infty$.
\end{lemma}

\begin{remark}\label{cgremark}

When $\Lambda$ has a single vertex, and so $G$ is a group, if $(G,\Lambda)$ is pseudo free in the sense of \cite{LY}, we claim that the set $F_{p,q,\lambda}^l(g,v)$ is empty, and hence $c_{p,q,\lambda}(g)=0$.  This holds because if $\mu\in F_{p,q,\lambda}^l(g,v)$, then by (\ref{S6}) and (\ref{S7}) we have $(h_{p,q}(\lambda)^{-1}g)|_\mu=s(\mu)=v$ and $(h_{p,q}(\lambda)^{-1}g)\cdot \mu=\mu$. Now the pseudo freeness implies $g=h_{p,q}(\lambda)$, which is a contradiction.
\end{remark}

\begin{proof}[Proof of Lemma~\ref{lemma:cg}]
We adjust the arguments in the proof of \cite[Proposition~8.2]{LRRW} to fit our setting. We claim that the sequence $\{c_{p,g,\lambda}^l(g)\}$ is increasing and bounded.
For boundedness, first note that since $g\neq h_{p,q}(\lambda)$, there is $\eta\in \Lambda$ such that $g\cdot \eta\neq h_{p,q}(\lambda)\cdot \eta$. We can assume that $d(\eta)=jN$ for some $j\in \N$. Hence $\eta\notin F_{p,q,\lambda}^{j}(g,s(\eta))$. An application of (\ref{S3}) implies that for every $\zeta\in s(\eta)\Lambda^{lN-jN}$ with $l> j$, the element $\eta\zeta$ does not belong to $F_{p,q,\lambda}^{l}(g,s(\zeta))$. Therefore for $l>j$ and $v\in \Lambda^0$, we have 
\[\big| F_{p,q,\lambda}^{l}(g,v)\big|\leq \big|\D(g)\Lambda^{lN}v\big|-\big|s(\eta)\Lambda^{lN-jN}v\big|=B^{lN}(\D(g),v)-B^{lN-jN}(s(\eta),v).\]
Now we have 
\begin{align*}
c_{p,q,\lambda}^{l}(g)&\leq \rho(\Lambda)^{-lN}\Big(\sum_{v\in \Lambda^{0}}B^{lN}(\D(g),v)x_\Lambda(v)-\sum_{v\in \Lambda^{0}}B^{lN-jN}(s(\eta),v)x_\Lambda(v)\Big)\\
&=\rho(\Lambda)^{-lN}\Big(\rho(\Lambda)^{lN}x_\Lambda(\D(g))-\rho(\Lambda)^{lN-jN}x_\Lambda(s(\eta))\Big)\\
&=x_\Lambda(\D(g))-\rho(\Lambda)^{-jN}x_\Lambda(s(\eta)),
\end{align*}
and hence $0\leq c_{p,q,\lambda}^{l}(g)<x_\Lambda(\D(g))$.

It remains to show that $\{c_{p,q,\lambda}^l(g)\}$ is an increasing sequence. Let $\mu\in F_{p,q,\lambda}^{l}(g,v)$ and $\nu \in v\Lambda^N$. Then (\ref{S1}) implies that 
\[g\cdot(\mu \nu)=h_{p,q}(\lambda)\cdot( \mu\nu)\quad \text{ and } g|_{\mu\nu}=h_{p,q}(\lambda)|_ {\mu\nu}.\]
Therefore $\mu\nu\in F_{p,q,\lambda}^{l+1}(g, s(\nu))$ and hence $\big|F_{p,q,\lambda}^{l+1}(g,v)|\geq \sum_{w\in \Lambda^0}|F_{p,q,\lambda}^{l}(g,w)\big|B^N(w,v)$. Now we have 
\begin{align*}
c_{p,q,\lambda}^{l+1}(g)&=\rho(\Lambda)^{-lN-N}\sum_{v\in \Lambda^{0}}\big|F^{l+1}_{p,q,\lambda}(g,v)\big|x_\Lambda(v)
\\
&\geq\rho(\Lambda)^{-lN-N}\sum_{v\in \Lambda^{0}}\sum_{w\in \Lambda^0}|F_{p,q,\lambda}^{l}(g,w)|B^N(w,v)x_\Lambda(v).
\end{align*}
Recalling that $x_\Lambda$ is the Perron--Frobenius eigenvector of $\Lambda$, we get
\[c_{p,q,\lambda}^{l+1}(g)\geq\rho(\Lambda)^{-lN-N}\sum_{w\in \Lambda^{0}}|F_{p,q,\lambda}^{l}(g,w)|\rho(\Lambda)^{N}x_\Lambda(w) 
=c_{p,q,\lambda}^{l}(g),
\]
as required.
\end{proof}

\begin{cor}\label{cor:cg}
Let $\Lambda$ be a strongly connected finite $ k$-graph, $(G,\Lambda)$ be a self-similar groupoid action, and $x_\Lambda$ be the unimodular Perron--Frobenius eigenvector of $\Lambda$. Let $N:=(1,\dots,1)\in \N^k$, and $g\in G\setminus{\Lambda^0}$. 
 For $v\in \Lambda^{0}$, $l\in \N$ define
\[F_{g}^l(v):=\big\{\mu\in \D(g)\Lambda^{lN} v: g\cdot \mu= \mu, g|_\mu=v\big\}\quad \text{and}\quad c_{g}^l:=\rho(\Lambda)^{-lN}\sum_{v\in \Lambda^{0}}\big|F^l_{g}(v)\big|x_\Lambda(v).\]
Then $\{c_{g}^{l}\}$ converges to a limit $c_{g}\in\big[0,x_\Lambda(\D(g))\big)$ as $l\to \infty$.
\end{cor}
\begin{proof}
Since $0\in\Per(G,\Lambda)$ and $h_{0,0}$ is the identity function on $\Lambda^0$, the proof follows from applying Lemma~\ref{lemma:cg} with $p=q=0$.
\end{proof}
\begin{remark}\label{cgremark2}
For a directed graph $E$ and self-similar action $(G,E)$,  the set $ F_g^l(v)$ and the number $c_g$ coincide with those of \cite[Proposition~8.2]{LRRW}. 
\end{remark}

\begin{definition}\label{def:finitestate}
Following \cite{lrrw, LRRW, LY}, we say a self-similar action $(G,\Lambda)$ is \textit{finite-state} if for all $g\in G\setminus\Lambda^0$, the set 
$
\{g|_\lambda: \lambda\in \D(g)\Lambda\}
$
is finite. Given $\lambda,\mu\in \Lambda$ and $g\in G$ with $s(\lambda)=g\cdot s(\mu)$, we define
\begin{equation}\label{Z-set}
Z(\lambda,g,\mu):=\{x\in 
\Lambda^\infty: x=\lambda (g\cdot y)=\mu y \text{ for some } y\in \Lambda^\infty\}.
\end{equation}
A continuity argument shows that $Z(\lambda,g,\mu)$ is closed and hence is $M$-measurable.
\end{definition}

\begin{thm}\label{prop:peron-measure}
Let $\Lambda$ be a strongly connected finite $ k$-graph, $(G,\Lambda)$ be a finite-state self-similar groupoid action, $M$ be the Parry measure on $\Lambda^\infty$, and $x_\Lambda$ be the unimodular Perron--Frobenius eigenvector of $\Lambda$. Then for $\lambda,\mu\in \Lambda$ and $g\in G$ with $s(\lambda)=g\cdot s(\mu)$, we have 
\begin{align}\label{measure-triple}
M\big(Z(\lambda,g,\mu)\big)
&= \begin{cases}\rho(\Lambda)^{-d(\mu)}x_\Lambda(s(\mu))& \text {if } d(\lambda)-d(\mu)\in \Per(G,\Lambda), \mu=\theta_{d(\lambda),d(\mu)}(\lambda )\\
\quad\quad\quad \quad\quad\quad&\text{ and }g=h_{d(\lambda),d(\mu)}(\lambda )\\
\rho(\Lambda)^{-d(\mu)}c_{d(\lambda),d(\mu),\lambda}(g)& \text {if } d(\lambda)-d(\mu)\in \Per(G,\Lambda), \mu=\theta_{d(\lambda),d(\mu)}(\lambda )\\
\quad\quad\quad \quad\quad\quad&\text{ and }g\neq h_{d(\lambda),d(\mu)}(\lambda )\\
0 &\text{otherwise.}\\
\end{cases}
\end{align}
\end{thm}

We spend the rest of this section proving Theorem~\ref{prop:peron-measure}. We first need some technical lemmas.

\begin{lemma}\label{meaure:lemma1}
Let $\Lambda$ be a strongly connected finite $ k$-graph, $(G,\Lambda)$ be a finite-state self-similar groupoid action. Let $M$ be the Parry measure on $\Lambda^\infty$, $p-q\in \Per(G,\Lambda)$, $\lambda\in \Lambda^p$, and $h:=h_{p,q}(\lambda)$. Take $g\in G$ such that $g\neq h$ and suppose that there exists $\omega\in \Lambda\setminus\Lambda^0$ such that 
 $g|_\omega\neq h|_\omega$. Then for $N\coloneqq (1,1,\dots,1)\in \N^k$ there exists $b:=jN$ for $j\in\N$, and a real number $0<K<1$, such that 
\begin{align}\label{approximatio-meaure 1}
M\bigg(\bigcup_{\substack{\mu\in s(\lambda)\Lambda^{lb}\\g\cdot \mu=h\cdot \mu, \, g|_\mu\neq h|_\mu}}Z(\lambda \mu)\bigg)
&\leq K^l M\big(Z(\lambda)\big)\quad \text{ for } l\in \N.
\end{align}
\end{lemma} 

\begin{proof}
Since $(G,\Lambda)$ is finite-state, and the existence of $\omega$, the set $\{\nu\in \Lambda\setminus\{0\}: g|_\nu\neq h|_\nu\}$ is nonempty and finite; say $\{\nu_1,\dots, \nu_r\}$. Then there are (not necessarily distinct) $\xi_1, \dots,\xi_r$ such that 
\begin{equation}\label{equ:finite-state}
g|_{\nu_i}\cdot \xi_i\neq h|_{\nu_i}\cdot \xi_i \quad \text{ for all }1\leq i\leq r.
\end{equation} 
Since $\Lambda$ is strongly connected, it has no sources (see \cite[Lemma~2.1]{aHLRS}), and we can assume that for each $1\le i\le r$ we have $d(\xi_i)=jN$ for some $j\in \N$. Let $b\coloneqq jN$. 

Now, \eqref{peron-measure-prop} implies that $M(Z(\xi_i))>0$ for each $1\leq i\leq j$. Hence there is a real number $0<K_i<1$ for each $1\leq i\leq j$ such that 
\begin{equation}\label{kiequation}
M\big(Z(r(\xi_i)\setminus Z(\xi_i)\big)< K_iM\big(Z(r(\xi_i)\big).
\end{equation}
We let $K:=\max\{K_1\dots K_j\}$. We claim that $b$ and $K$ satisfy \eqref{approximatio-meaure 1}.

We proceed by induction on $l$. The case $l=0$ is trivial. Now suppose \eqref{approximatio-meaure 1} is true for $l$. First observe the general fact that given $\mu= \eta\zeta$, we have 
\begin{align*} 
g\cdot \mu=h\cdot \mu\text{ and }g|_\mu\neq h|_\mu \iff g\cdot \eta=h\cdot \eta, g|_\eta\cdot \zeta=h|_\eta\cdot \zeta, g|_\eta\neq h|_\eta, (g|_\eta)|_\zeta\neq(h|_\eta)|_\zeta. 
\end{align*}
Using this fact we see that 
\begin{align*}
\bigcup_{\substack{\mu\in s(\lambda)\Lambda^{(l+1)b}\\g\cdot \mu=h\cdot \mu\\ g|_\mu\neq h|_\mu}}Z(\lambda \mu)
\notag&=\bigcup_{\substack{\eta\in s(\lambda)\Lambda^{lb}\\g\cdot \eta=h\cdot \eta\\ g|_\eta\neq h|_\eta}}\bigcup_{\substack{\zeta\in s(\eta)\Lambda^{b}\\g|_\eta\cdot \zeta=h|_\eta\cdot \zeta \\ (g|_\eta)|_\zeta\neq(h|_\eta)|_\zeta}}Z(\lambda \eta\zeta).
\end{align*}
For each $\eta\in s(\lambda)\Lambda^{lb}$ in the subscript of the exterior union, since $g|_\eta\neq h|_\eta$, we have $\eta=\nu_i$ for some $1\le i\le r$. Then the corresponding $\xi_i$ does not belong to the interior union, and hence
\begin{align}\label{equ:finite-state-measure 3}
\bigcup_{\substack{\mu\in s(\lambda)\Lambda^{(l+1)b}\\g\cdot \mu=h\cdot \mu\\ g|_\mu\neq h|_\mu}}Z(\lambda \mu)
&\subseteq\bigcup_{\substack{\eta\in s(\lambda)\Lambda^{lb}\\g\cdot \eta=h\cdot \eta\\ g|_\eta\neq h|_\eta}}\bigcup_{\substack{\zeta\in s(\eta)\Lambda^{b}\setminus\{\xi_i\}\\}}Z(\lambda \eta\zeta).
\end{align}
Two applications of \eqref{peron-measure-prop} now gives
\begin{align}\label{similar-measure1}
M\bigg(\bigcup_{\substack{\mu\in s(\lambda)\Lambda^{(l+1)b}\\g\cdot \mu=h\cdot \mu\\ g|_\mu\neq h|_\mu}}Z(\lambda \mu)\bigg)
\notag&\leq \sum_{\substack{\eta\in s(\lambda)\Lambda^{lb}\\g\cdot \eta=h\cdot \eta\\ g|_\eta\neq h|_\eta}}\sum_{\zeta\in s(\eta)\Lambda^{b}\setminus\{\xi_i\}}M\big(Z(\lambda \eta\zeta)\big)\\
\notag&=\sum_{\substack{\eta\in s(\lambda)\Lambda^{lb}\\g\cdot \eta=h\cdot \eta\\ g|_\eta\neq h|_\eta}}\rho(\Lambda)^{d(\lambda \eta)}\sum_{\zeta\in s(\eta)\Lambda^{b}\setminus\{\xi_i\}}M\big(Z(\zeta)\big)\\
&=\sum_{\substack{\eta\in s(\lambda)\Lambda^{lb}\\g\cdot \eta=h\cdot \eta\\ g|_\eta\neq h|_\eta}}\rho(\Lambda)^{d(\lambda \eta)}M\Big(Z(r(\xi_i))\setminus Z(\xi_i)\Big).
\end{align}
Since $ s(\eta)=r(\xi_i)$, \eqref{kiequation} and the inductive hypothesis now gives
\begin{align}\label{similar-measure2}
M\bigg(\bigcup_{\substack{\mu\in s(\lambda)\Lambda^{(l+1)b}\\g\cdot \mu=h\cdot \mu\\ g|_\mu\neq h|_\mu}}Z(\lambda \mu)\bigg)
\notag&\leq K_i\sum_{\substack{\eta\in s(\lambda)\Lambda^{lb}\\g\cdot \eta=h\cdot \eta\\ g|_\eta\neq h|_\eta}}\rho(\Lambda)^{d(\lambda \eta)}M\big(Z(s(\eta)\big)\\
\notag&\leq KM\bigg(\bigcup_{\substack{\eta\in s(\lambda)\Lambda^{lb}\\g\cdot \eta=h\cdot \eta\\ g|_\eta\neq h|_\eta}}Z(\lambda \eta)\bigg)\\
&\leq K^{l+1}M\big(Z(\lambda)\big).
\end{align}
\end{proof}

\begin{lemma}\label{lemmayl0}
Let $\Lambda$ be a strongly connected finite $ k$-graph, and $(G,\Lambda)$ be a self-similar groupoid action. Let $\kappa\in \Z^k\setminus\Per(G,\Lambda)$. Let $\{g_1,\dots, g_j\}$ be a finite subset of $G$. Then there exist $a\in \N^k\setminus\{0\}$ and 
for each $g_i$ a path $\nu_i\in \D(g_i)\Lambda^a$ such that for $\lambda, \mu\in \Lambda$ with $s(\lambda)=g_i\cdot s(\mu)$ and $d(\lambda)-d(\mu)=\kappa$, we have
\begin{equation}\label{measure-new}
\Lambda^{\min}(\lambda(g_i\cdot \nu_i),\mu\nu_i)= \emptyset. 
\end{equation}
\end{lemma}

 \begin{proof}
As in \cite[Lemma~8.3]{aHLRS}, $p:=\kappa \vee 0$, and $q:=-\kappa \vee 0$ satisfies $\kappa=p-q$, and whenever $\kappa=p'-q'$, we have $p'\leq p$ and $q'\leq q$. Since $p-q\not\in \Per(G,\Lambda)$, for each $g_i$ there must be $x_i\in \D(g_i)\Lambda^\infty$ such that $\sigma^{p}( x_i)\neq \sigma^{q}(g_i\cdot x_i)$. Therefore there are $m_i\in \N^k$ such that $\sigma^{p}(x_i)(0,m_i)\neq \sigma^{q}( g_i\cdot x_i)(0,m_i)$. We let $a:=p+q+\sum_{i=1}^j m_i$, and $\nu_i:=x_i(0,a)$ for each $1\leq i\leq I$. 
 
Fix $i$, and let $\lambda, \mu\in \Lambda$ with $s(\lambda)=g_i\cdot s(\mu)$ and $d(\lambda)-d(\mu)=\kappa$. Then $d(\lambda)=n+p$ and $d(\mu)=n+q$ for some $n\in \N^k$. Let $\lambda=\eta \nu$ and $\mu=\zeta \tau$ such that $d(\eta)=d(\zeta)=n$. If $\eta\neq \zeta$, clearly $\Lambda^{\min}(\lambda(g_i\cdot \nu_i),\mu\nu_i)= \emptyset$. So we suppose that $\eta=\zeta$. Writing $m:=\sum_{i=1}^j m_i$, we have 
 \[\nu(g_i\cdot \nu_i)(p+q,p+q+m)=(g_i\cdot \nu_i)(q,q+m)=(g_i\cdot x_i)(q,q+m)=\sigma^q(g_i\cdot x_i)(0,m).\]
 A simpler argument shows that $\tau\nu_i(p+q,p+q+m)=\sigma^p(x_i)(0,m)$. 
 Since $m\geq m_i$, $\sigma^{p}( x_i)(0,m_i)\neq \sigma^{q}(g_i\cdot x_i)(0,m_i)$ implies that $\sigma^{p}(x_i)(0,m)\neq \sigma^{q}( g_i\cdot x_i)(0,m)$m giving $\Lambda^{\min}(\nu(g_i\cdot \nu_i),\tau\nu_i)= \emptyset$. Hence $\Lambda^{\min}(\lambda(g_i\cdot \nu_i),\mu\nu_i)= \emptyset$.
 \end{proof}

\begin{lemma}\label{lem:minimalex-measure}
Let $\Lambda$ be a strongly connected finite $ k$-graph, $(G,\Lambda)$ be a finite-state self-similar groupoid action, and $M$ be the Parry measure on $\Lambda^\infty$. Let $g\in G$, and $\kappa\in \Z^k\setminus\Per(G,\Lambda)$. Then there exists $a\in \N^k\setminus\{0\}$ and a real number $ 0<K<1$ such that for all $\lambda, \mu\in \Lambda$ with $s(\lambda)=g\cdot s(\mu)$ and $d(\lambda)-d(\mu)=\kappa$, we have
\begin{equation}\label{equ:minimalex-measure}
M\bigg(\bigcup_{\substack{\nu\in s(\mu)\Lambda^{la}\\ \Lambda^{\min}(\lambda(g\cdot \nu),\mu\nu)\neq\emptyset}}Z(\mu \nu)\bigg)
\leq K^l M(Z(\mu))\quad \text{ for } l\in \N. 
\end{equation}
\end{lemma}

\begin{proof}
Since $(G,\Lambda)$ is finite-state, the set $\{g|_\eta:\eta\in \Lambda\}$ is finite and we can label it as $\{g_1,\dots ,g_j\}$. Applying Lemma~\ref{lemmayl0} gives $a\in \N^k\setminus\{0\}$ and paths $\nu_i\in \D(g_i)\Lambda^a$ that satisfy \eqref{measure-new}.
For each $i$, \eqref{peron-measure-prop} gives a real number
$0<K_i<1$ such that 
 \[M\big(Z\big(\D(g_i)\big)\setminus Z(\nu_i)\big)<K_i M\big(Z\big(\D(g_i)\big)\big).\]
Let $K:=\max\{K_1,\cdots ,K_j\}$. We claim that $a$ and $K$ satisfy \eqref{equ:minimalex-measure}.
 
We follow by induction. The base case $l=0$ is clear. Suppose that \eqref{equ:minimalex-measure} is true for $l$. We first note that
\begin{align*}
\bigcup_{\substack{\nu\in s(\mu)\Lambda^{(l+1)a}\\ \Lambda^{\min}(\lambda(g\cdot \nu),\mu\nu)\neq\emptyset}}Z(\mu \nu)
&=\bigcup_{\substack{\eta\in s(\mu)\Lambda^{la}\\ \Lambda^{\min}(\lambda(g\cdot \eta),\mu\eta)\neq\emptyset}}\,\,\bigcup_{\substack{\zeta\in s(\eta)\Lambda^{a}\\ \Lambda^{\min}(\lambda(g\cdot \eta)(g|_\eta\cdot \zeta),\mu\eta\zeta)\neq\emptyset}}Z(\mu \eta\zeta).
\end{align*}
Fix $\eta$ in the subscript of the exterior union. Then $g|_\eta$ belongs to $\{g_1,\dots, g_I\}$; say $g|_\eta=g_i$. Since $d(\lambda(g\cdot \eta))-d(\mu\eta)=d(\lambda)-d(\mu)=\kappa$, the path $\nu_i$ corresponding to $g|_\eta$ satisfies $\Lambda^{\min}\big(\lambda(g\cdot \eta)(g|_\eta\cdot\nu_i),\mu\eta\nu_i\big)=\emptyset$, and therefore $\nu_i$ does not belong to the interior union. So we have 
\begin{align*}
\bigcup_{\substack{\nu\in s(\mu)\Lambda^{(l+1)a}\\ \Lambda^{\min}(\lambda(g\cdot \nu),\mu\nu)\neq\emptyset}}Z(\mu \nu)
&\subseteq\bigcup_{\substack{\eta\in s(\mu)\Lambda^{la}\\ \Lambda^{\min}(\lambda(g\cdot \eta),\mu\eta)\neq\emptyset}}\bigcup_{\substack{\zeta\in s(\eta)\Lambda^{a}\setminus\{\nu_i\}\\ }}Z(\mu \eta\zeta).
\end{align*}
Now, similar computations to those used to establish \eqref{similar-measure1} and \eqref{similar-measure2}, using induction hypothesis, show that
\begin{align*}
M\bigg(\bigcup_{\substack{\nu\in s(\mu)\Lambda^{(l+1)a}\\ \Lambda^{\min}(\lambda(g\cdot \nu),\mu\nu)\neq\emptyset}}Z(\lambda \mu)\bigg)
&\leq K^{l+1}M\big(Z(\mu)\big),
\end{align*}
finishing the proof.
\end{proof}

\begin{proof}[Proof of Theorem~\ref{prop:peron-measure}]
First suppose that $d(\lambda)-d(\mu)\notin \Per(G,\Lambda)$. Let $b\in \N^k$ the number obtained by applying Lemma~\ref{meaure:lemma1} with $g$ and $d(\lambda)-d(\mu)$. Then for each $l\in \N$, we have
\begin{align*}
Z(\lambda,g,\mu)&=\bigcup_{\substack{\nu\in s(\mu)\Lambda^{lb}\\ \Lambda^{\min}(\lambda(g\cdot \nu),\mu\nu)\neq\emptyset}} \{x\in 
\Lambda^\infty: x=\lambda(g\cdot \nu)(g|_\nu\cdot y)=\mu\nu y \text{ for some } y\in \Lambda^\infty\}\\
&\subseteq \bigcup_{\substack{\nu\in s(\mu)\Lambda^{lb}\\ \Lambda^{\min}(\lambda(g\cdot \nu),\mu\nu)\neq\emptyset}}Z(\mu\nu).
\end{align*}
By Lemma~\ref{meaure:lemma1}, we have
\begin{align*}
M\big(Z(\lambda,g,\mu)\big)&\leq\sum_{\substack{\nu\in s(\mu)\Lambda^{lb}\\ \Lambda^{\min}(\lambda(g\cdot \nu),\mu\nu)\neq\emptyset}}M(Z(\mu\nu))\leq K^{l}M(Z(\mu)).
\end{align*}
Now letting $l\to \infty$, we get $M\big(Z(\lambda,g,\mu)\big)$.

Second, suppose that $d(\lambda)-d(\mu)\in \Per(G,\Lambda)$ and $\mu\neq\theta_{d(\lambda),d(\mu)}(\lambda )$. The equation~\eqref{pershift} implies that $Z\big(\theta_{d(\lambda),d(\mu)}(\lambda )\big)\subseteq Z(\lambda)$. Also rewriting ~\eqref{pershift} as 
\[\lambda y= \theta_{d(\lambda),d(\mu)}(\lambda)\big(h_{d(\lambda),d(\mu)}(\lambda)^{-1}\cdot y\big) \text{ for } y\in s(\lambda))\Lambda^\infty\]
 gives $Z(\lambda)\subseteq Z\big(\theta_{d(\lambda),d(\mu)}(\lambda )\big) $. Therefore $Z(\lambda)=Z\big(\theta_{d(\lambda),d(\mu)}(\lambda ) \big)$. Now we have 
\[Z(\lambda,g,\mu)\subseteq Z(\lambda)\cap Z(\mu)=Z\big(\theta_{d(\lambda),d(\mu)}(\lambda )\big)\cap Z(\mu)=\emptyset,\]
giving $M\big(Z(\lambda,g,\mu)\big)=0$.

Third, suppose that $d(\lambda)-d(\mu)\in \Per(G,\Lambda), \mu=\theta_{d(\lambda),d(\mu)}(\lambda )$ and 
$g=h_{d(\lambda),d(\mu)}(\lambda )$. Then by definitions of the functions $\theta_{d(\lambda),d(\mu)}$ and $h_{d(\lambda),d(\mu)}$, $Z(\lambda,g,\mu)=Z(\mu)$ and hence $M\big(Z(\lambda,g,\mu)\big)=\rho(\Lambda)^{-d(\mu)}x_\Lambda(s(\mu))$.

Finally, let $d(\lambda)-d(\mu)\in \Per(G,\Lambda), \mu=\theta_{d(\lambda),d(\mu)}(\lambda )$ and 
$g\neq h_{d(\lambda),d(\mu)}(\lambda )$. Let $N:=(1,\dots,1)\in \N^k$, $l\in \N$ and write $h:=h_{d(\lambda),d(\mu)}(\lambda )$. For each $\nu\in s(\lambda)\Lambda$ and $y\in s(\nu)\Lambda^\infty$, we have $\mu \nu y=\lambda (h\cdot (\nu y))=\lambda (h\cdot \nu)(g|_\nu\cdot y)$. Then by factorization property of $\Lambda$, we have
\[Z(\lambda,g,\mu)=\bigcup_{\substack{\nu\in s(\lambda)\Lambda^{lN}\\g\cdot \nu=h\cdot \nu}}\{x\in 
\Lambda^\infty: x=\lambda (g\cdot \nu)(g|_\nu\cdot y)=\mu \nu y \text{ for some } y\in \Lambda^\infty\}.\]

Since $\{x\in 
\Lambda^\infty: x=\lambda (g\cdot \nu)(g|_\nu\cdot y)=\mu \nu y \text{ for some } y\in \Lambda^\infty\}\subseteq Z(\mu\nu)$, we have 
\begin{align*}
M\bigg(\bigcup_{\substack{\nu\in s(\lambda)\Lambda^{lN}\\g\cdot \nu=h\cdot \nu, \, g|_\nu\neq h|_\mu}}& \{x\in 
\Lambda^\infty: x=\lambda (g\cdot \nu)(g|_\nu\cdot y)=\mu \nu y \text{ for some } y\in \Lambda^\infty\}\bigg)\\
&\leq M\bigg(\bigcup_{\substack{\nu\in s(\lambda)\Lambda^{lN}\\g\cdot \nu=h\cdot \nu, \, g|_\nu\neq h|_\mu}}Z(\mu\nu)
\end{align*}
which will vanish as $l\to \infty$ by Lemma~\ref{meaure:lemma1}. Therefore 
\begin{align*}
M\big(Z(\lambda,g,\mu)\big)&=M\bigg(\bigcup_{\substack{\nu\in s(\lambda)\Lambda^{lN}\\g\cdot \nu=h\cdot \nu, \, g|_\nu= h|_\mu}} \{x\in 
\Lambda^\infty: x=\lambda (g\cdot \nu)(g|_\nu\cdot y)=\mu \nu y \text{ for some } y\in \Lambda^\infty\}\bigg)\\
&=M\bigg(\bigcup_{\substack{\nu\in s(\lambda)\Lambda^{lN}\\g\cdot \nu=h\cdot \nu, \, g|_\nu= h|_\mu}} \{x\in 
\Lambda^\infty: x=\lambda (h\cdot (\nu y))=\mu \nu y \text{ for some } y\in \Lambda^\infty\}\bigg)\\
&=M\bigg(\bigcup_{\substack{\nu\in s(\lambda)\Lambda^{lN}\\g\cdot \nu=h\cdot \nu, \, g|_\nu= h|_\mu}} Z(\mu\nu)\bigg)\\
&=\sum_{\substack{\nu\in s(\lambda)\Lambda^{lN}\\g\cdot \nu=h\cdot \nu, \, g|_\nu= h|_\mu}} M\big(Z(\mu\nu)\big)\\
&=\rho(\Lambda)^{-d(\mu)}\rho(\Lambda)^{-lN}\sum_{\substack{\nu\in s(\lambda)\Lambda^{lN}\\g\cdot \nu=h\cdot \nu, \, g|_\nu= h|_\mu}} M\big(Z(s(\nu))\big)\\
&=\rho(\Lambda)^{-d(\mu)}\rho(\Lambda)^{-lN}\sum_{v\in \Lambda^{0}}\big|F^l_{d(\lambda),d(\mu),\lambda}(g,v)\big|x_\Lambda(v)
\end{align*}
by definition of $F^l_{d(\lambda),d(\mu),\lambda}(g,v)$. By letting $l\to \infty$, we get 
\[M\big(Z(\lambda,g,\mu)\big)=\rho(\Lambda)^{-d(\mu)}c_{d(\lambda),d(\mu),\lambda}(g)\]
 proving the second case of \eqref{measure-triple}.
\end{proof}

\section{A characterisation of KMS states of $\OO(G,\Lambda)$ for the preferred dynamics}

Suppose that $ \Lambda$ is a finite strongly connected $k$-graph. By the \textit{preferred dynamics} on $\OO(G,\Lambda)$ we mean the dynamics $\sigma$ given by 
\begin{equation}\label{pref-dyn}
\sigma_t(t_\lambda u_g t_\mu^*)=\rho(\Lambda)^{it(d(\lambda)-d(\mu))}t_\lambda u_g t_\mu^*,
\end{equation}
which is in fact the dynamics \eqref{action on T} with $r=(\ln(\rho(B_1)),\dots, \ln(\rho(B_k)) )$.

\begin{thm}\label{thm:formula for kms1}
Let $\Lambda$ be a strongly connected finite $k$-graph, $(G,\Lambda)$ be a finite-state self-similar groupoid action, and $x_\Lambda$ be the unimodular Perron--Frobenius eigenvector of $\Lambda$. Let $\sigma$ be the preferred dynamics \eqref{pref-dyn} on $\OO(G,\Lambda)$, and $V$ be the representation of $ \Per(G,\Lambda)$ as in Proposition~\ref{prop:rep-per}. For each $p-q\in \Per(G,\Lambda)$ and $\lambda\in \Lambda^p$, let $\theta_{p,q}(\lambda), h_{p,q}(\lambda)$ be as in Lemma~\ref{lemma:htehta}. For $g\in G$ with $g\neq h_{p,q}(\lambda)$, let $c_{p,q,\lambda}(g)$ be the number defined in Lemma~\ref{lemma:cg}.
If $\phi$ is a KMS$_{1}$-state for $\big(\OO(G,\Lambda),\sigma\big)$, then
\begin{align}\label{thm10}
\phi(t_\lambda u_g t_\mu^*)
&= \begin{cases}\rho(\Lambda)^{-d(\mu)}x_\Lambda(s(\mu))\phi\big(V_{d(\lambda)-d(\mu)}\big)& \text {if } d(\lambda)-d(\mu)\in \Per(G,\Lambda), \\
\quad\quad\quad \quad&\mu=\theta_{d(\lambda),d(\mu)}(\lambda )\text{ and }g=h_{d(\lambda),d(\mu)}(\lambda )\\
\rho(\Lambda)^{-d(\mu)} c_{d(\lambda),d(\mu),\lambda}(g)\phi\big(V_{d(\lambda)-d(\mu)}\big)& \text {if } d(\lambda)-d(\mu)\in \Per(G,\Lambda), \\
\quad\quad\quad \quad\quad\quad&\mu=\theta_{d(\lambda),d(\mu)}(\lambda )\text{ and }g\neq h_{d(\lambda),d(\mu)}(\lambda )\\
0 &\text{otherwise.}\\
\end{cases}
\end{align}
\end{thm}

\begin{remark}
As we mentioned in Remark~\ref{cgremark}, when $\Lambda$ has a single vertex and the action is pseudo free, the numbers $c_{p,q,\lambda}(g)$ are zero. Also, \cite[Corollary~3.11]{LY} shows that for $p-q\in \Per(G,\Lambda)$, we have $\rho(\Lambda)^{-p}=\rho(\Lambda)^{-q}$. It follows that Theorem~\ref{thm:formula for kms1} coincides with \cite[Theorem~6.10]{LY} in the setting of a psedo free action of a single vertex $k$-graph.
\end{remark}

\begin{proof}[Proof of Theorem~\ref{thm:formula for kms1}]
First suppose that $d(\lambda)-d(\mu)\notin \Per(G,\Lambda)$. If $\rho(\Lambda)^{d(\lambda)}\neq\rho(\Lambda)^{d(\mu)}$, then two applications of KMS condition show that 
$\phi(t_\lambda u_g t_\mu^*)=0$. If $\rho(\Lambda)^{d(\lambda)}=\rho(\Lambda)^{d(\mu)}$, then let $a\in \N^k$ and $K$ be the numbers obtained by applying Lemma~\ref{lem:minimalex-measure} with $g$ and $\kappa=d(\lambda)-d(\mu)$. For each $l\in \N$, similar calculations to those in the proof of \cite[Theorem~6.10]{LY} show that 
\begin{align}\label{komaki-way}
\phi(t_\lambda u_g t_\mu^*)&=\sum_{\substack{\nu\in \D(g) \Lambda^n\\\Lambda^{\min}(\lambda(g\cdot \nu),\mu\nu)\neq\emptyset}}\phi\big(t_{\lambda (g\cdot \nu)}u_{g|_\nu}t_{\mu\nu}^*\big).
\end{align}
Since $s(g\cdot \nu)=g|_\nu\cdot s(\nu)$, Lemma~\ref{lemma-kms-g-invariance}\eqref{kms-action-inv-4} now implies that 
\begin{align}
\big|\phi(t_\lambda u_g t_\mu^*)\big|\notag&\leq\sum_{\substack{\nu\in \D(g) \Lambda^{la}\\\Lambda^{\min}(\lambda(g\cdot \nu),\mu\nu)\neq\emptyset}}\phi(t_{\mu\nu}t_{\mu\nu}^*).
\end{align}
Then Lemma~\ref{lem:minimalex-measure} implies that
\begin{align}
\big|\phi(t_\lambda u_g t_\mu^*)\big|\notag&\leq M\Big(\bigcup_{\substack{\nu\in \D(g) \Lambda^{la}\\ \Lambda^{\min}(\lambda(g\cdot \nu),\mu\nu)\neq\emptyset}}Z(\mu\nu)\Big)
\leq K^{l}M\big(Z(\mu)\big).
\end{align}
Letting $l\to \infty$, we see that $\phi(t_\lambda u_g t_\mu^*)=0$.

We assume for the rest of the proof that $d(\lambda)-d(\mu)\in \Per(G,\Lambda)$. Suppose that $\mu\neq\theta_{d(\lambda),d(\mu)}(\lambda )$. By Lemma~\ref{lemma:rep-per} we have
\[t_\lambda u_g t_\mu^*=t_\lambda t_\lambda^*t_\lambda u_g t_\mu^*=t_{\theta_{d(\lambda),d(\mu)}(\lambda )}t_{\theta_{d(\lambda),d(\mu)}(\lambda )}^*t_\lambda u_g t_\mu^*.\]
The KMS condition then implies that 
\[\phi(t_\lambda u_g t_\mu^*)=\phi\big(t_\mu^*t_{\theta_{d(\lambda),d(\mu)}(\lambda )}t_{\theta_{d(\lambda),d(\mu)}(\lambda )}^*t_\lambda u_g \big),\]
which is zero because $d\big(\theta_{d(\lambda),d(\mu)}(\lambda )\big)=d(\mu)\implies t_\mu^*t_{\theta_{d(\lambda),d(\mu)}}=0$.

Now suppose that $\mu=\theta_{d(\lambda),d(\mu)}(\lambda )$ and 
$g=h_{d(\lambda),d(\mu)}(\lambda )$. Then we have a cycline triple $(\lambda,h_{d(\lambda),d(\mu)}(\lambda), \theta_{d(\lambda),d(\mu)}(\lambda))$ in the sense of \cite{LY}, and a similar argument to the last paragraph in the proof of \cite[Theorem~6.1]{LY} shows that 
\begin{equation}\label{cyc-triple}
\phi\big(t_\lambda u_{g} t_\mu^*\big)=\phi\big(t_\lambda u_{h_{d(\lambda),d(\mu)}(\lambda)} t_{\theta_{d(\lambda),d(\mu)}(\lambda)}^*\big)=\rho(\Lambda)^{-p}x_\Lambda(s( \lambda))\phi(V_{d(\lambda)-d(\mu)}).
\end{equation}

Finally, suppose
 that $\mu=\theta_{d(\lambda),d(\mu)} (\lambda)$ and $g\neq h_{d(\lambda),d(\mu)}(\lambda)$. For convenience let $h:=h_{d(\lambda),d(\mu)}(\lambda)$. Let $N:=(1,\dots,1)\in \N^k$ and $b\in \N^k$ be as in Lemma~\ref{meaure:lemma1}. For each $l\in \N$, the equation ~\eqref{komaki-way} implies that 
\begin{align}\label{psi-thm2-g,h}
\phi(t_\lambda u_g t_\mu^*)\notag&=\sum_{\nu\in \D(g) \Lambda^{lb}}\phi\big(t_{\lambda(g\cdot \nu)}u_{g|_\nu}t_{\mu\nu}^*\big).
\end{align}
We aim to reveal the number $c_{d(\lambda),d(\mu),\lambda}(g)$ in \eqref{thm10}, by reducing the sum step by step: Fix a $\nu$-summand. If $g\cdot \nu\neq h\cdot \nu$, then since $d(\lambda(g\cdot \nu\big))-d(\mu\nu)=d(\lambda)-d(\mu)\in\Per(G,\Lambda)$, Proposition~\ref{prop:per}\eqref{prop:per3} implies that
\[\theta_{d(\lambda)+n,d(\mu)+m}\big(\lambda(g\cdot \nu)\big)=\theta_{d(\lambda),d(\mu)} (\lambda)((h^{-1}g)\cdot\nu)=\mu((h^{-1}g)\cdot\nu)\neq \mu\nu.\]
An argument as in the previous case now shows that $\phi\big(t_{\lambda(g\cdot \nu)}u_{g|_\nu}t_{\mu\nu}^*\big)$ vanishes, and we get
\begin{align}\label{finally we get}
\phi(t_\lambda u_g t_\mu^*)\notag&=\sum_{\substack{\nu\in \D(g) \Lambda^{lb}\\g\cdot \nu=h\cdot \nu}}\phi\big(t_{\lambda(g\cdot \nu)}u_{g|_\nu}t_{\mu\nu}^*\big)\\
&=\sum_{\substack{\nu\in \D(g) \Lambda^{lb}\\g\cdot \nu=h\cdot \nu,g|_\nu= h|_\nu}}\phi\big(t_{\lambda(g\cdot \nu)}u_{g|_\nu}t_{\mu\nu}^*\big)+\sum_{\substack{\nu\in \D(g) \Lambda^{lb}\\g\cdot \nu= h\cdot\nu, g|_\nu\neq h|_\nu}}\phi\big(t_{\lambda(g\cdot \nu)}u_{g|_\nu}t_{\mu\nu}^*\big).
\end{align}
Next, we apply  Lemma~\ref{lemma-kms-g-invariance}\eqref{kms-action-inv-4} and Lemma~\ref{meaure:lemma1} to get
\begin{align}\label{second vanish}
\Bigg|\sum_{\substack{\nu\in \D(g) \Lambda^{lb}\\g\cdot \nu= h\cdot\nu, g|_\nu\neq h|_\nu}}\phi\big(t_{\lambda(g\cdot \nu)}u_{g|_\nu}t_{\mu\nu}^*\big)\Bigg|\notag&\leq
\sum_{\substack{\nu\in \D(g) \Lambda^{lb}\\g\cdot \nu= h\cdot\nu, g|_\nu\neq h|_\nu}}\big|\phi\big(t_{\mu\nu}t_{\mu\nu}^*\big)\big|\\
\notag&=M\Big(\bigcup_{\substack{\nu\in \D(g) \Lambda^{lb}\\g\cdot \nu= h\cdot\nu, g|_\nu\neq h|_\nu}}Z\big(\mu\nu\big)\Big)\\
&\leq K^l M\big(Z(\mu)\big).
\end{align}
Since $K^l M\big(Z(\mu)\big)$ approaches zero as $l\to \infty$, we have
\begin{align}\label{thm2:equ22}
\phi(t_\lambda u_g t_\mu^*)
&=\lim_{lb\to \infty}\sum_{\substack{\nu\in \D(g) \Lambda^{lb}\\g\cdot \nu=h\cdot \nu,g|_\nu= h|_\nu}}\phi\big(t_{\lambda(g\cdot \nu)}u_{g|_\nu}t_{\mu\nu}^*\big).
\end{align}
 Now for $\nu\in \D(g) \Lambda^{lb}$ with $g\cdot \nu=h\cdot \nu$ and $g|_\nu= h|_\nu$ the formulas for $h$ and $\theta $ functions in Lemma~\ref{lemma:mor-prop-htheta} implies that
$\theta_{d(\lambda)+lb,d(\mu)+lb}\big(\lambda(g\cdot \nu)\big)=\mu\nu$ and $h_{d(\lambda)+lb,d(\mu)+lb}\big(\lambda(g\cdot \nu)\big)=h|_\nu=g|_\nu$. Since $\phi$ is a KMS$_1$-state and $d(\lambda(g\cdot \nu))-d(\mu\nu)=d(\lambda)-d(\mu)\in \Per(G,\Lambda)$, \eqref{cyc-triple} gives us 
\[\phi\big(t_{\lambda(g\cdot \nu)}u_{g|_\nu}t_{\mu\nu}^*\big)=\rho(\Lambda)^{-(d(\lambda)+lb)}x_\Lambda(s(h\cdot \lambda))\phi\big(V_{d(\lambda)-d(\mu)}\big)\]
Putting this in \eqref{thm2:equ22}, we have
\begin{align}
\phi(t_\lambda u_g t_\mu^*)
\notag&=\rho(\Lambda)^{-d(\lambda)}\phi\big(V_{d(\lambda)-d(\mu)}\big)\lim_{lb\to \infty}\sum_{\substack{\nu\in \D(g) \Lambda^{lb}\\g\cdot \nu=h\cdot \nu\\g|_\nu= h|_\nu}}\rho(\Lambda)^{-lb}x_\Lambda(s(h\cdot \lambda))\\
&\notag=\rho(\Lambda)^{-d(\nu)} c_{d(\lambda),d(\mu),\lambda}(g)\phi\big(V_{d(\lambda)-d(\mu)}\big)
\end{align}
by definition of $c_{d(\lambda),d(\mu),\lambda}(g)$, and we have finished the proof.
\end{proof}

\section{Construction of KMS states of $\OO(G,\Lambda)$ for a preferred dynamics}
In this section, we use the Parry measure $M$ on $\Lambda^\infty$ to build a KMS$_1$-state of $(\OO(G,\Lambda),\sigma)$.  We will use this state to prove the surjectivity of the isomorphism in Theorem~\ref{thm3}. To build the KMS$_1$-state, we need to  assume that the unimodular Perron--Frobenius eigenvector $x_\Lambda$ is \textit{$G$-invariant} in the sense that $x_\Lambda(v)=x_\Lambda(w)$ if  $v=g\cdot w$ for a $g\in G$. Equivalently, $x_\Lambda(\D(g))=x_\Lambda(\c(g))$ for all $g\in G$. 

We start with a lemma. We leave the proof to the reader.

\begin{lemma}\label{infinit-path-rep}
Let $\Lambda$ be a finite strongly-connected $k$-graph, and $(G,\Lambda)$ be a self-similar groupoid action. Let $\{\delta_x:x\in \Lambda^\infty\}$ be the orthogonal basis of point masses in $\Lambda^\infty$. Then there is a Cuntz--Krieger $\Lambda$-family $\{T_\lambda:\lambda\in \Lambda\}$ in $B(\ell^2(\Lambda^\infty))$, and a unitary representation $U: G\to B(\ell^2(\Lambda^\infty))$ 
such that 
\begin{equation}\label{ut-infinite}
T_\lambda(\delta_ x)=\delta_{s(\lambda),r(x)} \delta_{\lambda x}\,\,\text { and }\,\, U_g(\delta_ x)=\delta_{\D(g),r(x)} \delta_{g\cdot x}.
\end{equation}
The pair $(U,T)$ satisfies the relations \eqref{us1}--\eqref{us3} of Proposition~\ref{prop:universal}.
\end{lemma}

\begin{prop}\label{prop:KMS1}
Let $\Lambda$ be a finite strongly-connected $k$-graph, and $(G,\Lambda)$ be a finite-state self-similar groupoid action. Suppose that $x_\Lambda$ is $G$-invariant and let  $\sigma$ be  the preferred dynamics \eqref{pref-dyn}. Let $(U,T)$ be the family of Lemma~\ref{infinit-path-rep} and $\pi_{U,T}: \OO(\Lambda,G)\to B(\ell^2(\Lambda^\infty))$ be the corresponding homomorphism. Let $M$ be the Parry measure on $\Lambda^\infty$. Then for each $a\in \OO(\Lambda,G)$ the function $\omega_a: x\mapsto (\pi_{U,T}(a)\delta_x|\delta_x)$ is $M$-measurable, and the formula 
\begin{equation}\label{KMS1-formula}
\phi(a)= \int_{\Lambda^\infty} \big(\pi_{U,T}(a)\delta_x|\delta_x\big)\,dM(x) \quad\text{ for all } a\in \OO(\Lambda,G)
\end{equation}
defines a KMS$_1$-state for $(\OO(\Lambda,G),\sigma)$.
\end{prop}
\begin{remark}
Let $E$ be strongly connected  directed graph and $(G,E)$ be a  self-similar groupoid action. By \cite[Proposition~8.4]{LRRW},  the unimodular Perron--Frobenius eigenvector $x_E$ is $G$-invariant. We note that for a strongly connected  finite  $k$-graph $\Lambda$, and a self-similar action $(G,\Lambda)$, the self-similarity does not imply the $G$-invariance of $x_\Lambda$.
However, if $\Lambda$ is coordinatewise-irreducible in the sense of  \cite{aHLRS}, then by \cite[Lemma~2.1]{aHLRS},  the unimodular Perron-Frobenius eigenvectors of the matrices $B_1, \dots, B_k$ are all the same  and hence  equal to $x_\Lambda$ (see \cite[Corollary~4.2(b)]{aHLRS}). Now since the directed graph $E_1=(\Lambda^0,\Lambda^{e_1}, r,s)$ is strongly connected and $(G,E_1)$ is self-similar as in \cite{LRRW}, an application of \cite[Proposition~8.4]{LRRW} implies that  $x_\Lambda$ is $G$-invariant.
\end{remark}

We need the following lemma to prove Proposition~\ref{prop:KMS1}.

\begin{lemma}\label{simplifyingkms}
Let $\Lambda$ be a finite strongly-connected $k$-graph, and $(G,\Lambda)$ be a finite-state self-similar groupoid action. Suppose that $\sigma$ is the preferred dynamics \eqref{pref-dyn}, and that $\phi$ is a state of $\big(\OO(G,\Lambda),\sigma\big)$ such that 
\begin{equation}\label{KMS1-cond-forgens}
\phi(t_\lambda u_g t_\mu^*t_\nu u_f t_\xi^*)=\rho(\Lambda)^{-(d(\lambda)-d(\mu))}\phi(t_\nu u_f t_\xi^*t_\lambda u_g t_\mu^*),
\end{equation}
for all $f,g\in G$ and $\lambda,\mu,\nu,\xi\in \Lambda$ with $s(\lambda)=g\cdot s(\mu)$, $s(\nu)=f\cdot s(\xi)$ and $d(\nu),d(\xi)\geq d(\lambda)\vee d(\mu)$. Then $\phi$ is a KMS$_1$-state. 
\end{lemma}

\begin{proof}
It suffices to take two typical spanning elements $b=t_\lambda u_g t_\mu^*$ and $c=t_\nu u_f t_\xi^*$ with $s(\lambda)=g\cdot s(\mu)$ and $s(\nu)=f\cdot s(\xi)$ and show that 
\begin{equation}\label{lemma:KMS1-cond}
\phi(bc)=\rho(\Lambda)^{-(d(\lambda)-d(\mu))}\phi(cb).
\end{equation}
Let $p:=d(\lambda)\vee d(\mu)$. A computation using (\ref{CK}) and \eqref{UT} gives
\begin{align*}
bc=t_\lambda u_g t_\mu^*t_\nu u_f t_\xi^*
&=\sum_{\eta\in s(\xi)\Lambda^p}t_\lambda u_g t_\mu^*t_\nu u_ft_\eta t_\eta^* t_\xi^*\\
\notag&=\sum_{\eta\in s(\xi)\Lambda^p}t_\lambda u_gt_{\mu}^*t_\nu t_{f\cdot \eta} u_{f|_\eta} t_{\xi\eta}^* \\
\notag&=\sum_{\eta\in s(\xi)\Lambda^p}t_\lambda u_gt_{\mu}^*t_{\nu (f\cdot \eta)} u_{f|_\eta} t_{\xi\eta}^*.
\end{align*}
Now, applying \eqref{KMS1-cond-forgens} and reversing the above computation gives
\begin{align*}
\phi(bc)&=\sum_{\eta\in s(\xi)\Lambda^p}\phi\big(t_\lambda u_gt_{\mu}^*t_{\nu (f\cdot \eta)} u_{f|_\eta} t_{\xi\eta}^*\big)\\
&=\rho(\Lambda)^{d(\lambda)-d(\mu)}\sum_{\eta\in s(\xi)\Lambda^p}\phi\big(t_{\nu (f\cdot \eta)} u_{f|_\eta} t_{\xi\eta}^*t_\lambda u_gt_{\mu}^*\big)\\
&=\rho(\Lambda)^{-(d(\lambda)-d(\mu))}\phi(cb),
\end{align*}
giving \eqref{lemma:KMS1-cond}.
\end{proof}

\begin{proof}[Proof of Proposition~\ref{prop:KMS1}]
Take $\lambda,\mu\in \Lambda$ and $g\in G$ such that $s(\lambda)=g\cdot s(\mu)$. The set $Z(\lambda,g,\mu):=\{x\in \Lambda^\infty: x=\lambda (g\cdot y)=\mu y \text{ for some } y\in \Lambda^\infty\}$ is $M$-measurable and by \eqref{ut-infinite} we have
\begin{align}\label{kms1-comp}
\big(\pi_{U,T}(t_{\lambda} u_{g} t_\mu^*)\delta_x|\delta_x\big)&=\big( U_{g} T_\mu^*\delta_x|T_{\lambda}^*\delta_x\big)
= \begin{cases}1& \text {if } x\in Z(\lambda,g,\mu)\\
0 &\text{otherwise.}
\end{cases}
\end{align}
Hence $\omega_{t_\lambda u_g t_\mu^*}$ is the characteristic function of the set $Z(\lambda,g,\mu)$, and so is $M$-measurable. A continuity argument similar to the proof of \cite[Lemma~10.1(b)]{aHLRS} now shows that for each $a\in\OO(G,V)$ the function $\omega_a$ is $M$-measurable, and that $\phi$ is a well-defined norm decreasing map. Since $M$ is a probability measure, we have $\phi(1)= \int_{\Lambda^\infty}\|\delta_x\|\,dM(x)=1$ and hence $\phi$ is a state.

To see that $\phi$ is a KMS$_1$-state, it suffices by Lemma~\ref{simplifyingkms} to take two typical elements $b=t_\lambda u_g t_\mu^*$ and $c=t_\nu u_f t_\xi^*$ with $d(\nu),d(\xi)\geq d(\lambda)\vee d(\mu)$ and prove 
\begin{equation}\label{KMS1-cond}
\phi(bc)=\rho(\Lambda)^{-(d(\lambda)-d(\mu))}\phi(cb).
\end{equation}

 We start by computing the left-hand side. Since $d(\nu)\geq d(\mu)$, $t_\mu^*t_\nu$ vanishes unless $\nu=\mu \alpha$ for some $\alpha \in \Lambda$. Now, since $\c(f)=s(\nu)$, we can apply \eqref{kms1-comp} to get 
\begin{align*}
\phi\big(bc\big)
\notag&= \begin{cases}\phi\big(t_\lambda u_g t_\alpha u_f t_\xi^*\big) &\text {if } \nu=\mu \alpha \\
0 &\text{ otherwise.}
\end{cases}\\
\notag&= \begin{cases}\phi\big(t_{\lambda(g\cdot \alpha)} u_{g|_\alpha f} t_\xi^*\big)& \text {if } \nu=\mu \alpha \\
0 &\text{ otherwise.}
\end{cases}\\
&= \begin{cases}M\big(Z\big(\lambda(g\cdot \alpha),g|_\alpha f,\xi\big)\big)& \text {if } \nu=\mu \alpha \\
0 &\text{otherwise.}
\end{cases}
\end{align*}
By the formula \eqref{measure-triple} we have 
\begin{align*}
\phi\big(bc\big)=
&\begin{cases}M\big(Z\big(\lambda(g\cdot \alpha),g|_\alpha f,\xi\big)\big)&\text {if } \nu=\mu \alpha, d(\lambda(g\cdot \alpha))-d(\xi)\in \Per(G,\Lambda)\\
&\quad\quad\text{ and }\xi=\theta_{ d(\lambda(g\cdot \alpha)),d(\xi)}(\lambda(g\cdot \alpha) )\\
0 &\text{ otherwise.}
\end{cases}
\end{align*}
Similar computation together with $\c(g|_{g^{-1}\cdot \beta})=s(\beta)=s(\xi)=\D(f)$ show that
\begin{align*}
\phi\big(cb\big)&=\begin{cases}\phi\big(t_\nu u_fu_{g|_{g^{-1}\cdot \beta}} t_{\mu(g^{-1}\cdot \beta)}^*\big)& \text {if } \xi=\lambda \beta\\
0 &\text{otherwise.}
\end{cases}\\
&=\begin{cases}M\big(Z\big(\nu, f (g|_{g^{-1}\cdot \beta}), \mu(g^{-1}\cdot \beta)\big)\big)& \xi=\lambda \beta,d(\nu)-d(\mu(g^{-1}\cdot \beta))\in \Per(G,\Lambda)\\
& \quad\quad\text{ and }\theta_{ d(\nu),d(\mu(g^{-1}\cdot \beta))}(\nu)=\mu(g^{-1}\cdot \beta)\\
0 &\text{otherwise.}
\end{cases}\\
\end{align*}

Next we show that the conditions appearing in $\phi(bc)$ and $\phi(cb)$ are equivalent. Suppose that we have 
\begin{align}\label{dagger}
\nu=\mu \alpha,\quad d(\lambda(g\cdot\alpha))-d(\xi)&\in \Per(G,\Lambda)\text{ and }\quad \xi=\theta_{d(\lambda(g\cdot\alpha)),d(\xi)}\big(\lambda (g\cdot \alpha)\big).
 \end{align}
Since $d(\alpha)-(d(\xi)-d(\lambda))=d(\lambda(g\cdot\alpha))-d(\xi)\in \Per(G,\Lambda)$, we can write $\beta:=\theta_{d(\alpha),d(\xi)-d(\lambda)}(g\cdot \alpha)$. Then Lemma~\ref{lemma:mor-prop-htheta}\eqref{lemma:mor-prop-htheta-2} implies that 
 $\xi=\theta_{d(\lambda(g\cdot\alpha)),d(\xi)}\big(\lambda (g\cdot \alpha)\big)=\lambda \beta.$
 Then 
\[d(\nu)-d(\mu(g^{-1}\cdot \beta))=d(\alpha)-d(\beta)=d(\alpha)-d(\xi)+d(\lambda)= d(\lambda(g\cdot\alpha))-d(\xi)\in \Per(G,\Lambda).\]
Applying Lemma~\ref{lemma:mor-prop-htheta}\eqref{lemma:mor-prop-htheta-2} twice and using Lemma~\ref{lemma:htehta}\eqref{lemma:htehta-2} in the third equality, we get
\begin{align*}
\theta_{d(\mu(g^{-1}\cdot \beta)),d(\nu)}\big(\mu(g^{-1}\cdot \beta)\big)&=\mu\theta_{d(\beta),d(\alpha)}(g^{-1}\cdot \beta)=\mu\theta_{d(\beta),d(\alpha)}\big(\theta_{d(\alpha),d(\beta)}( \alpha)\big)\\
&=\mu\theta_{d(\alpha),d(\alpha)}(\alpha)=\mu\alpha =\nu,
\end{align*}
giving $\mu(g^{-1}\cdot \beta)=\theta_{d(\nu),d(\mu(g^{-1}\cdot \beta))}(\nu)$. Thus we have proven 
\begin{align}\label{ddagger}
\xi=\lambda \beta, d(\nu)-d(\mu(g^{-1}\cdot \beta))\in \Per(G,\Lambda) \text{ and }\mu(g^{-1}\cdot \beta)=\theta_{d(\nu),d(\mu(g^{-1}\cdot \beta))}(\nu).
\end{align}
Similarly, we can deduce $\eqref{dagger}$ from $\eqref{ddagger}$ by taking $\alpha:=\theta_{d(\beta),d(\nu)-d(\mu))}(g^{-1}\cdot \beta).$

To prove \eqref{KMS1-cond}, it suffices to consider \eqref{dagger} and \eqref{ddagger}, and establish 
\begin{equation}\label{measure:KMS1-cond}
M\big(Z\big(\lambda(g\cdot \alpha),g|_\alpha f,\xi\big)\big)=\rho(\Lambda)^{-(d(\lambda)-d(\mu))}M\big(Z\big(\nu, f (g|_{g^{-1}\cdot \beta}), \mu(g^{-1}\cdot \beta)\big)\big).
\end{equation}
For this, we first write $h:=h_{ d(\alpha),d(\beta))}(\alpha )$ and apply Lemma~\ref{lemma:mor-prop-htheta} to get
\begin{align}\label{h1h2}
h_{ d(\lambda(g\cdot \alpha)),d(\xi)}(\lambda(g\cdot \alpha) )\notag=h_{ d( \alpha),d(\beta)}(g\cdot \alpha )=&g|_\alpha h \big(g|_{g^{-1}\cdot\beta}\big)^{-1}=g|_\alpha h g^{-1}|_\beta\quad\text{and }\\
h_{ d(\nu),d(\mu(g^{-1}\cdot \beta))}(\nu )&=h.
 \end{align}
 Now if $g|_\alpha f= h_{ d(\lambda(g\cdot \alpha)),d(\xi)}(\lambda(g\cdot \alpha) )$, \eqref{h1h2} implies that $f=h \big(g|_{g^{-1}\cdot\beta}\big)^{-1}$ and hence $f (g|_{g^{-1}\cdot \beta})=h= h_{ d(\nu),d(\mu(g^{-1}\cdot \beta))}(\nu )$. Since $s(g^{-1}\cdot \beta)=g^{-1}|_\beta\cdot s(\beta)=g^{-1}|_\beta\cdot s(\xi)$, the formula \eqref{measure-triple} for $M$, implies that
\begin{align*}
\rho(\Lambda)^{-(d(\lambda)-d(\mu))}& M\big(Z\big(\nu, f (g|_{g^{-1}\cdot \beta}), \mu(g^{-1}\cdot \beta)\big)\big) \\
&=\rho(\Lambda)^{-(d(\lambda)-d(\mu))}\rho(\Lambda)^{-d(\mu(g^{-1}\cdot \beta))}x_\Lambda\big(g^{-1}|_\beta\cdot s(\xi)\big)\\
&=\rho(\Lambda)^{-d(\xi)}x_\Lambda(s(\xi))\quad \text{(since $x_\Lambda$ is $G$-invariant)}\\
&=M\big(Z\big(\lambda(g\cdot \alpha),g|_\alpha f,\xi\big)\big).
\end{align*}

If $g|_\alpha f\neq h_{ d(\lambda(g\cdot \alpha)),d(\xi)}(\lambda(g\cdot \alpha) )$,
 by \eqref{h1h2}, there is a $\eta$ such that $f\cdot \eta\neq \big( h\big(g|_{g^{-1}\cdot\beta}\big)^{-1}\big)\cdot \eta$.
Writing $\eta':= \big(g|_{g^{-1}\cdot\beta}\big)^{-1}\cdot \eta$, we get $(f g|_{g^{-1}\cdot \beta})\cdot \eta'\neq h_{ d(\nu),d(\mu(g^{-1}\cdot \beta))}(\nu )\cdot \eta'$ which implies that 
$f (g|_{g^{-1}\cdot \beta})\neq h_{ d(\nu),d(\mu(g^{-1}\cdot \beta))}(\nu ).$ For each $l\in \N$ and $v\in \Lambda^0$, \eqref{h1h2} implies that 
\begin{align*}
&F_{d(\nu),d(\mu(g^{-1}\cdot \beta)),\lambda(g\cdot\alpha)}^l\big(g|_\alpha f,v\big)\\
\quad&=\{\eta\in \D(f)\Lambda^{lN}; (g|_\alpha f)\cdot \eta=(g|_\alpha hg^{-1}|_\beta)\cdot \eta\text{ and }(g|_\alpha f)|_\eta=(g|_\alpha hg^{-1}|_\beta)|_\eta\},
\end{align*}
and therefore 
\begin{align*}
c_{d(\nu),d(\mu(g^{-1}\cdot \beta)),\lambda(g\cdot\alpha)}^l\big(g|_\alpha f,v\big)&=\sum_{\substack{\eta\in \D(f)\Lambda^{lN}, (g|_\alpha f)\cdot \eta=(g|_\alpha hg^{-1}|_\beta)\cdot \eta\\(g|_\alpha f)|_\eta=(g|_\alpha hg^{-1}|_\beta)|_\eta}}M(Z(\eta)).
\end{align*}
A computation using (\ref{S6}) and (\ref{S7}) shows that $\eta\mapsto (g^{-1}|_\beta)\cdot \eta$ is a bijection from the set $\{\eta\in \D(f)\Lambda^{lN}: (g|_\alpha f)\cdot \eta=(g|_\alpha hg^{-1}|_\beta)\cdot \eta \text{ and }(g|_\alpha f)|_\eta=(g|_\alpha hg^{-1}|_\beta)|_\eta\}$ onto 
the set $\{\zeta\in \D(g|_{g^{-1}\cdot \beta})\Lambda^{lN}:(fg|_{g^{-1}\cdot \beta})\cdot \zeta=h\cdot \zeta\text{ and } (fg|_{g^{-1}\cdot \beta})|_\zeta=h|_\zeta\}$. So we can rewrite the sum:
\begin{align*}
c_{d(\nu),d(\mu(g^{-1}\cdot \beta)),\lambda(g\cdot\alpha)}^l\big(g|_\alpha f,v\big)&=\sum_{\substack{\zeta\in \D(g|_{g^{-1}\cdot \beta})\Lambda^{lN}, (fg|_{g^{-1}\cdot \beta})\cdot \zeta=h\cdot \zeta\\(fg|_{g^{-1}\cdot \beta})|_\zeta=h|_\zeta}}M(Z(\zeta))\\
&=\rho(\Lambda)^{lN}\sum_{v\in \Lambda^0}\Big|F_{d(\lambda(g\cdot \alpha)),d(\xi), \nu}^l\big(fg|_{g^{-1}},v\big)\Big|x_\Lambda(s((\zeta))&\text{ by } \eqref{h1h2}\\
&=c_{d(\lambda(g\cdot \alpha)),d(\xi), \nu}^l\big(fg|_{g^{-1}},v\big).
\end{align*}
By letting $l\to \infty$, we get
\[c_{d(\nu),d(\mu(g^{-1}\cdot \beta)),\lambda(g\cdot\alpha)}\big(g|_\alpha f,v\big)=c_{d(\lambda(g\cdot \alpha)),d(\xi), \nu}\big(fg|_{g^{-1}},v\big),\] 
and therefore 
\begin{align*}
\rho(\Lambda)^{-(d(\lambda)-d(\mu))}&M\big(Z\big(\nu, f (g|_{g^{-1}\cdot \beta}), \mu(g^{-1}\cdot \beta)\big)\big)\\
&= \rho(\Lambda)^{-(d(\lambda)-d(\mu))}\rho(\Lambda)^{-d(\mu)-d(\beta)}c_{d(\lambda(g\cdot \alpha)),d(\xi), \nu}\big(fg|_{g^{-1}},v\big)\\
&=\rho(\Lambda)^{-d(\xi)}c_{d(\nu),d(\mu(g^{-1}\cdot \beta)),\lambda(g\cdot\alpha)}\big(g|_\alpha f,v\big)\\
&=M\big(Z\big(\lambda(g\cdot \alpha),g|_\alpha f,\xi\big)\big),
\end{align*}
finishing the proof of \eqref{measure:KMS1-cond} as we required.
\end{proof}

\section{Main results for a preferred dynamics on $\OO(G,\Lambda)$}\label{main-cuntz}

We now state our main results about the KMS structure of $\OO(G,\Lambda)$ under the preferred dynamics.

\begin{thm}\label{thm3}
Let $\Lambda$ be a finite strongly connected $k$-graph, $(G,\Lambda)$ be a finite-state self-similar groupoid action, and $\sigma$ be the preferred dynamics \eqref{pref-dyn} on $\OO(G,\Lambda)$. Let $V$ be the unitary representation of Proposition~\ref{prop:rep-per}
and let $\pi_V:C^*(\Per(G,\Lambda))\to \OO(G,\Lambda)$ be the induced homomorphism. 
\begin{enumerate}
\item \label{thm3-1}If there is a $g\in G$ such that $x_\Lambda(\D(g))\neq x_\Lambda(\c(g))$, then $\OO(G,\Lambda)$ has no KMS$_1$-state.
\item \label{thm3-2}If $x_\Lambda$ is $G$-invariant, then  the map $\phi\mapsto \phi\circ \pi_V$ is an affine isomorphism from the simplex of KMS$_1$-states of $(\OO(G,\Lambda),\sigma)$ onto the simplex of states of $C^*(\Per(G,\Lambda))$. 
\end{enumerate}
\end{thm}

The proof follows the argument of the proof of \cite[Theorem~7.1]{aHLRS1}.

\begin{proof}[Proof of Theorem~\ref{thm3}]
For \eqref{thm3-1}, note that if  there is a  KMS$_1$-state $\phi$ of $\OO(G,\Lambda)$, then Lemma~\ref{lemma-kms-g-invariance}\eqref{kms-action-inv-1} says that $\phi(\D(g))= \phi(\c(g))$. Now \eqref{peron-measure-prop1} forces $x_\Lambda(\D(g))=x_\Lambda(\c(g))$ which is a contradiction.

For \eqref{thm3-2}, note that the map $\Omega:\phi\mapsto \phi\circ \pi_V$ is continuous and affine, and it follows from Theorem~\ref{thm:formula for kms1} that it is injective. To see the surjectivity, we first claim each character $\chi$ of $\Per(G,\Lambda)$ is in the range of $\Omega$. Let $z\in \T^k$ such that $z^p=\chi(p)$ for all $p\in\Per(G,\Lambda)$. Let $\phi$ be the KMS$_1$-state of the Proposition~\ref{prop:KMS1} and write
 $\gamma_z$ for the gauge automorphism of $\OO(G,\Lambda)$. An easy computation shows that $\phi\circ \gamma_z$ is a KMS$_1$-state of $\OO(G,\Lambda)$, and we have $\phi\circ \gamma_z(t_\lambda u_g t_\mu^*)=\rho(\Lambda)^{d(\lambda)-d(\mu)}M(Z(\lambda,g,\mu))$. Now the formula \eqref{measure-triple} for the Parry measure $M$ implies that 
\begin{align*}
\phi_z(t_\lambda u_g t_\mu^*)
&\!=\! \begin{cases}\rho(\Lambda)^{-d(\mu)}z^{d(\lambda)-d(\mu)}x_\Lambda(s(\mu))& \text {if } d(\lambda)-d(\mu)\in \Per(G,\Lambda), \mu=\theta_{d(\lambda),d(\mu)}(\lambda )\\
\quad\quad\quad \quad\quad\quad&\text{ and }g=h_{d(\lambda),d(\mu)}(\lambda )\\
\rho(\Lambda)^{-d(\mu)} z^{d(\lambda)-d(\mu)}c_{d(\lambda),d(\mu),\lambda}(g)& \text {if } d(\lambda)-d(\mu)\in \Per(G,\Lambda), \mu=\theta_{d(\lambda),d(\mu)}(\lambda )\\
\quad\quad\quad \quad\quad\quad&\text{ and }g\neq h_{d(\lambda),d(\mu)}(\lambda )\\
0 &\text{otherwise.}\\
\end{cases}
\end{align*} 
 Writing $i_{\Per(G,\Lambda)}:\Per(G,\Lambda)\to C^*(\Per(G,\Lambda))$ for the universal representation of $\Per(G,\Lambda)$, a similar computation to that in the proof of \cite[Theorem~7.1]{aHLRS1} shows that 
\[\phi_z\circ \pi_V\big(i_{\Per(G,\Lambda)}(p-q)\big)=\chi(p-q) \text{  for all }p-q\in \Per(G,\Lambda).\]
Thus $\Omega(\phi_z)=\chi$, and the claim holds. We can now appeal to the Krein--Milman theorem and the argument in the proof of \cite[Theorem~7.1]{aHLRS1} to see that $\Omega$ is surjective. Hence $\Omega$ is an isomorphism.
\end{proof}

\begin{thm}\label{thm2}
Let $\Lambda$ be a finite strongly connected $k$-graph, $(G,\Lambda)$ be a finite-state self-similar groupoid action, and $x_\Lambda$ be $G$-invariant. Let  $\sigma$ be the preferred dynamics \eqref{pref-dyn} on $\OO(G,\Lambda)$. For each $g\in G\setminus{\Lambda^0}$, let $c_g$ be as in Corollary~\ref{cor:cg}. 
Then there is a KMS$_1$-state $\Phi$ on $(\OO(G,\Lambda),\sigma)$ such that 
\begin{align}\label{kms1cg}
\Phi (t_\lambda u_g t_\mu^*)
&= \begin{cases}\rho(\Lambda)^{-d(\mu)}x_\Lambda(s(\mu))& \text {if } \lambda=\mu
\text{ and }g\in \Lambda^0\\
\rho(\Lambda)^{-d(\mu)} c_{g}&\text {if } \lambda=\mu
\text{ and } g \notin \Lambda^0\\
0 &\text{otherwise.}\\
\end{cases}
\end{align}
Moreover, $\Phi$ is the only KMS$_1$-state for $(\OO(G,\Lambda),\sigma)$ if and only if 
 $\Lambda$ is $G$-aperiodic.
\end{thm}

\begin{proof}
Let $\omega$ be the trivial state on $C^*\big(\Per(G,\Lambda)\big)$ given by $\omega(i_{\Per(G,\Lambda)}(p))=\delta_{p,0}$ for $p\in \Per(G,\Lambda)$. Then by Theorem~\ref{thm3}, $\Phi:=\Omega^{-1}(\omega)$ is a KMS$_1$-state for $(\OO(G,\Lambda),\sigma)$. Since $\Phi\circ \pi_V =\omega$,
an application of Theorem~\ref{thm:formula for kms1} shows that $\Phi$ satisfies \eqref{kms1cg}. The second assertion follows from Theorem~\ref{thm3} and the fact that $G$-aperiodic graphs have trivial $G$-aperiodic group.
\end{proof}

\section{Computing KMS$_1$ states of $\OO(G,\Lambda)$ in examples}\label{examples}
Suppose that  $\Lambda$ is a finite strongly connected $k$-graph,  $(G,\Lambda)$ is a finite-state self-similar groupoid action, $x_\Lambda$ is $G$-invariant, and $\Lambda$ is $G$-aperiodic.
Theorem~\ref{thm2} implies that $(\OO(G,\Lambda),\sigma)$ has a unique KMS$_1$-state $\Phi$ such that
\[\Phi(t_\lambda u_gt_\mu^*)=\delta_{\lambda,\mu}\rho(\Lambda)^{-d(\lambda)}\Phi(u_g),\quad \text { for } s(\lambda)=g\cdot s(\mu).\]
 Therefore the unique KMS$_1$-state $\Phi$ is determined by the values  $\{\Phi(u_g):g\in G\}$.
 
In this section, we show that both Examples~\ref{ex: single vertex} and \ref{ex: bascilica2} satisfy the conditions of Theorem~\ref{thm2}. We then compute the values of the unique KMS$_1$-state for these examples at partial unitaries corresponding to generators of $G_A$.  We first need a lemma.

\begin{lemma}\label{rho-per}
Let $\Lambda$ be a strongly connected   finite $k$-graph and $(G,\Lambda)$ be a  self-similar groupoid action. Suppose that for each  $p-q\in \Per(G,\Lambda)$ and $\lambda\in \Lambda^p$, we have  $\c(h_{p,q}(\mu))=\D(h_{p,q}(\mu))$. Then 
\begin{equation}\label{equ-cor-per-vec}
\rho(\Lambda)^p=\rho(\Lambda)^q\quad \text{ for all  }p-q\in \Per(G,\Lambda)\setminus\{0\},
\end{equation}
and   $\Per(G,\Lambda)$ is a subgroup of $\{n\in \Z^k:\rho(\Lambda)^n=1\}$.
\end{lemma}
\begin{proof}
Let $p-q\in \Per(G,\Lambda)\setminus\{0\}$. Since $\c(h_{p,q}(\mu))=\D(h_{p,q}(\mu))$,  Lemma~\ref{lemma:htehta}\eqref{lemma:htehta-1}, implies that the function $\theta_{p,q}:\Lambda^p\to \Lambda^q$ is a range and source preserving bijection. Now we are in the situation of \cite[Corollary~7.2]{aHLRS1} and \eqref{equ-cor-per-vec} follows from the computation (7.1) there.
\end{proof}
\begin{prop}
Let $(G_A,\Lambda)$ be the self-similar action of Example~\ref{ex: single vertex}. Then $G_A$ is a finite groupoid with elements $a:=g_a$ and $v:=g_v$. Moreover, $\Per(G_A,\Lambda)$ is trivial, and $\OO(G,\Lambda)$ with the preferred dynamics \eqref{pref-dyn} has a unique KMS$_1$-state $\Phi$ satisfying
\begin{equation}\label{KMS1stateExample1}
\Phi(u_g)=\begin{cases}
1 & \text{ if } g=v \\
1/3 & \text{ if } g=a.
\end{cases}
\end{equation}
\end{prop}
\begin{proof}
Easy computations show that $a^2$ acts trivially on edges and squares. Faithfulness of the action then implies that $a^2=v$, which is the neutral element, so $G_A \cong \Z/2\Z$. 
Since we only have one vertex, $\c(h_{p,q}(\mu))=\D(h_{p,q}(\mu))$ for all $p-q\in \Per(G,\Lambda)$ and $\lambda\in \Lambda^p$, so Lemma~\ref{rho-per} implies that  $\Per(G,\Lambda)\subseteq \{n\in \Z^2:\rho(\Lambda)^n=1\}$. On the other hand, since $\rho(\Lambda)=(2,3)$, the only $n\in \Z^2$ that satisfies $\rho(\Lambda)^n=1$ is $(0,0)$. Therefore $\Per(G_A,\Lambda)$ is trivial. The vector $x_\Lambda$ is automatically $G$-invariant. 
The computations proving \eqref{item: square_preserving_1} in Example~\ref{ex: single vertex} show that $a$ acts trivially on the squares $e_1 f_1\sim f_1 e_1$ and $e_1 f_2\sim f_2 e_1$ with restriction $v$, and non-trivially on all other squares.  From this it follows that $F_{a}^l(v)$ consists of all paths in the cylinder sets $Z(e_1 f_1) \cup Z(e_1f_2)$, so we have $|F_{a}^l(v)|=2\cdot 6^{l-1}$. So we compute
\[
c_{g}^l=(3 \cdot 2)^{-l}(2\cdot 6^{l-1})=1/3.
\]
Thus, using \eqref{kms1cg}, the KMS$_{1}$-state satisfies \eqref{KMS1stateExample1}.
\end{proof}

\begin{prop}\label{ex2}
Let $(G_A,\Lambda)$ be the self-similar action of Example~\ref{ex: bascilica2} and and let $\gamma=\frac{1+\sqrt{5}}{2}$ be the Golden mean. Then  $\Per(G_A,\Lambda)$ is trivial and    $\OO(G,\Lambda)$ with the preferred dynamics \eqref{pref-dyn} has a unique KMS$_1$-state $\Phi$ such that
\begin{equation}\label{KMS1stateExample2}
\Phi(u_g)=\begin{cases}
\gamma^{-1} & \text{ if } g=v \\
\gamma^{-2} & \text{ if } g=w \\
0 & \text{ if } g=a_v \\
0 & \text{ if } g=a_w \\
\frac{\gamma^{-1}}{2} & \text{ if } g=b_v \\
\frac{\gamma^{-2}}{2} & \text{ if } g=b_w.
\end{cases}
\end{equation}
\end{prop}
\begin{proof}
We observe that the adjacency matrices are $\left(\begin{matrix} 1&1\\1&0\\\end{matrix}\right)$ and $\left(\begin{matrix} 2&0\\0&2\\\end{matrix}\right)$, and they have spectral radii $\frac{1+\sqrt{5}}{2}$ and $2$, respectively. The unique $1$-norm eigenvector for both adjacency matrices is given by $x_\Lambda = \left(\begin{matrix} \gamma^{-1} \\ \gamma^{-2} \\ \end{matrix}\right)$. Since all groupoid elements  have the same domain and codomain, Lemma~\ref{rho-per} implies that $\Per(G_A,\Lambda)\subseteq \{n\in \Z^2:\rho(\Lambda)^n=1\}$. Since $\frac{1+\sqrt{5}}{2}$ and $2$ are rationally independent, the only $n\in \Z^2$ that satisfies $\rho(\Lambda)^n=1$ is $(0,0)$. Hence $\Per(G_A,\Lambda)$ is trivial and $\Lambda$ is $G$-aperiodic. 
 Since there is no groupoid element connecting the vertices $v$ and $w$, we deduce that $x_\Lambda$ is $G$-invariant. An argument almost identical to \cite[Proposition~ 2.5]{lrrw} shows that $(G_A,\Lambda)$ contracting, and hence 
  $(G_A,\Lambda)$ is finite-state. 
  
Looking at the relations we checked for \eqref{item: square_preserving_1}   and \eqref{item: square_preserving_2} in Example~\ref{ex: bascilica2}, we see that neither $a_v$ nor $a_w$ act trivially on any square, and hence Corollary~\ref{cor:cg} implies $c_{a_v}=0$ and $c_{a_w}=0$. Another application of  the relations in  \eqref{item: square_preserving_1}   and \eqref{item: square_preserving_2} for $b_v$ and $b_w$ show that
\[ \left| F^1_{b_v}(v)\right|=1, \, \left| F^1_{b_v}(w)\right|=1 \text{ and } \left| F^l_{b_v}(v)\right|=0, \, \left| F^l_{b_v}(w)\right|=0 \text{ for } l >1
\]
and
\[ \left| F^1_{b_w}(v)\right|=1, \, \left| F^1_{b_w}(w)\right|=0 \text{ and } \left| F^l_{b_w}(v)\right|=0, \, \left| F^l_{b_w}(w)\right|=0 \text{ for } l >1.
\]
Combining this with Corollary~\ref{cor:cg} we compute
\[
c_{b_v}=2^{-1}\gamma^{-1}(\gamma^{-1}+\gamma^{-2})=\frac{\gamma^{-1}}{2}\quad\text{and}\quad c_{b_w}=2^{-1}\gamma^{-1}(\gamma^{-1})=\frac{\gamma^{-2}}{2}.
\]
Thus, using \eqref{kms1cg}, the KMS$_{1}$-state satisfies \eqref{KMS1stateExample2}.
\end{proof}

\end{document}